\documentclass[10pt, colorlinks]{article}
\usepackage{latexsym, amssymb, amsmath, amsthm, amscd}
\usepackage[all,cmtip]{xy}
\usepackage{amsmath}
\usepackage{amsthm}
\usepackage{amssymb}
\usepackage{amsfonts}
\usepackage{amsrefs}
\usepackage{color,authblk}
\usepackage{tikz}
\usepackage{amscd} 
\usepackage{comment}
\usepackage{mathrsfs}
\usepackage{comment}
\usepackage[all]{xy}
\usepackage{graphicx}
\usepackage{authblk}
\usepackage{hyperref}
\usepackage{bbm}
\usepackage{fullpage}
\usepackage[all,cmtip]{xy}
\usepackage{relsize}
\usepackage{amsmath,amscd}
\usepackage{tikz-cd}
\usepackage{tikz}
\usetikzlibrary{matrix}
\usepackage[mathcal]{euscript}
\usepackage{txfonts}
\usepackage{enumerate}

\usepackage{hyperref}

\numberwithin{equation}{section} \DeclareMathSizes{2}{10}{12}{13}
\makeatletter
\newcommand*{\doublerightarrow}[2]{\mathrel{
		\settowidth{\@tempdima}{$\scriptstyle#1$}
		\settowidth{\@tempdimb}{$\scriptstyle#2$}
		\ifdim\@tempdimb>\@tempdima \@tempdima=\@tempdimb\fi
		\mathop{\vcenter{
				\offinterlineskip\ialign{\hbox to\dimexpr\@tempdima+1em{##}\cr
					\rightarrowfill\cr\noalign{\kern.5ex}
					\rightarrowfill\cr}}}\limits^{\!#1}_{\!#2}}}
\newcommand*{\triplerightarrow}[1]{\mathrel{
		\settowidth{\@tempdima}{$\scriptstyle#1$}
		\mathop{\vcenter{
				\offinterlineskip\ialign{\hbox to\dimexpr\@tempdima+1em{##}\cr
					\rightarrowfill\cr\noalign{\kern.5ex}
					\rightarrowfill\cr\noalign{\kern.5ex}
					\rightarrowfill\cr}}}\limits^{\!#1}}}
\makeatother

\newtheorem{thm}{Proposition}[section]
\newtheorem{Thm}[thm]{Theorem}

\newtheorem{cor}[thm]{Corollary}
\newtheorem{eg}[thm]{Example}
\newtheorem{lem}[thm]{Lemma}
\newtheorem{defn}[thm]{Definition}

\title{Hochschild theory of multiplicative sequences of algebras and coalgebra measurings}

\author{Abhishek Banerjee \footnote{Department of Mathematics, Indian Institute of Science, Bangalore, India. Email: abhishekbanerjee1313@gmail.com} $\qquad\qquad$ Surjeet Kour \footnote{Department of Mathematics, Indian Institute of Technology, Delhi, India. Email: koursurjeet@gmail.com} 
 $\qquad\qquad$ Dipti Paik \footnote{Department of Mathematics, Indian Institute of Technology, Delhi, India. Email: paikdipti09@gmail.com} }

\date{}

\begin{document}

\maketitle 

\medskip

\begin{abstract}
We study coalgebra measurings between multiplicative sequences of algebras and the maps induced by them on Hochschild homology. The Hochschild theory of multiplicative sequences is introduced as a functor taking values in graded algebras in the symmetric monoidal category of chain complexes, constructed with the help of the shuffle product. We develop the universal measuring coalgebra, or Sweedler Hom for multiplicative sequences, as well as study several other Sweedler operations in this context. In particular, we obtain an enrichment of multiplicative sequences over cocommutative coalgebras. Using an appropriate theory of bimodules over multiplicative sequences, we  study maps induced by comodule measurings  on the Hochschild theory with coefficients, as well as the corresponding enriched categories. Finally, we consider measurings and generalized Sweedler operations between multiplicative sequences induced by comultiplicative sequences of coalgebras, and also the maps in Hochschild theory obtained from them. 
\end{abstract}

\medskip

{MSC (2020) Subject classification: 16T15, 16E40}


	\hypersetup{linktocpage}
	\section{Introduction}

Let $k$ be a field. In this paper, all algebras and coalgebras are taken to be over $k$. The enrichment of algebras over coalgebras is a classical construction that goes back to Sweedler \cite{SW}. More precisely, if $A$, $A'$ are $k$-algebras, then a coalgebra measuring from $A$ to $A'$ consists of a coalgebra $C$ and a family $\Phi:C\longrightarrow Hom_k(A,A')$ of $k$-linear maps parametrized by $C$, which satisfies 
 \begin{equation}\label{eq1.1fh}
        \Phi(x)(a_1a_2)=\sum(\Phi(x_{(1)})(a_1))(\Phi(x_{(2)})(a_2))\qquad\Phi(x)(1_A)=\epsilon(x)1_{A'} \qquad x\in C, a_1,a_2\in A
    \end{equation} In \eqref{eq1.1fh}, the coproduct $\Delta$ on $C$ is written as $\Delta(x)=\sum x_{(1)}\otimes x_{(2)}$ for $x\in C$, and $\epsilon:C\longrightarrow k$ denotes the counit on
    $C$. If $x\in C$ is a grouplike element, i.e., $\Delta(x)=x\otimes x$ and $\epsilon(x)=1$, it follows that $\Phi(x):A\longrightarrow A'$ is an ordinary morphism of algebras. If $x\in C$ is a primitive element (provided the coalgebra $C$ is coagumented), then the condition in \eqref{eq1.1fh} resembles that of a derivation. The collection of coalgebra measurings from $A$ to $A'$ admits a universal object $Q(A,A')$. This is known as the universal measuring coalgebra, also called the Sweedler Hom from $A$ to $A'$. When $A'=k$, then $Q(A,k)$ recovers the notion of finite dual of an algebra, which is adjoint to the linear dual functor that takes coalgebras to algebras (see, for instance, Porst and Street \cite{PS}).  The Sweedler Hom also becomes the starting point for a family of constructions, known as the Sweedler operations (see Anel and Joyal \cite{AJ}) that relate the monoidal category of algebras to that of coalgebras. The theory of coalgebra measurings has also been developed in several other contexts, such as bialgebras (see Grunenfelder and Mastnak \cite{GM1}, \cite{GM2}) and entwining structures (see Brzezi\'{n}ski \cite{Brz}). 
For monoid objects in braided monoidal categories, they have been studied by Hyland, L\'{o}pez Franco and Vasilakopoulou \cite{MLV} and by Vasilakopoulou \cite{Vas}.   For more in the literature on coalgebra measurings, see, for instance \cite{Bat0}, \cite{MB}, \cite{MH},  \cite{Lauve}. 

\smallskip
Since coalgebra measurings can be understood as a generalization of morphisms of algebras, we have used them in previous work in \cite{AB}, \cite{AB2}, \cite{AB3} to induce maps between a number of homology and cohomology theories. The starting point was in \cite{AB}, where we began by showing that if $C$ is a cocommutative coalgebra and $\Phi:C\longrightarrow Hom_k(A,A')$ is a measuring, then we have induced maps 
\begin{equation}
C_*^\Phi(x): C_*(A)\longrightarrow C_*(A') \qquad (a_0\otimes a_1 \otimes ...\otimes a_n)\mapsto \sum (\Phi(x_{(1)})(a_0)\otimes \Phi(x_{(2)})(a_1)\otimes ...\otimes \Phi(x_{(n+1)})(a_n))
\end{equation}
between Hochschild complexes for each $x\in C$, where we know that $C_n(A):=A^{\otimes n+1}$ for each $n\geq 0$. This idea can be extended to several other situations, such as cyclic homology of algebras, the Chevalley-Eilenberg homology of Lie algebras, the homology of Leibniz algebras, or the dihedral homology of involutive algebras (see \cite{AB}, \cite{AB3}). In \cite{AB2}, we showed that coalgebra measurings induce morphisms between homology theories and cohomology theories of Hopf algebroids, and that these are compatible with Hopf-Galois maps. The maps induced by coalgebra measurings are also well behaved with respect to product and coproduct structures, such as the shuffle product on Hochschild homology (see \cite{AB}) or the coproduct on Lie algebra homology or Leibniz homology (see \cite{AB3}). Moreover, they are compatible (see \cite{AB3}) with ``comparison maps,'' such as those identifying the Hochschild (resp. cyclic) homology with the primitive part of the Leibniz homology (resp. Lie homology) of the matrix algebra, or the dihedral homology of an involutive algebra with the primitive part of the Lie algebra homology of symplectic matrices or skew-symmetric matrices. 

\smallskip
In this paper, we  study maps induced by coalgebra measurings on the Hochschild theory of multiplicative sequences of algebras. We recall that a multiplicative sequence  $({A}_\bullet,\tau)$  of algebras is a collection   ${A}_\bullet=\{A_n \mid n \ge 0\}$  algebras, with $A_0=k$, along with a collection of algebra homomorphisms
\begin{equation}\label{1.3eef}
  \tau=\{\tau_{m,n}: A_m \otimes A_n \longrightarrow A_{m+n}\}_{m,n\geq 0}
\end{equation}   satisfying certain associativity conditions. These were introduced by Davydov and Molev \cite{DAV} in order to provide a unifying framework for describing classical Schur-Weyl dualities in terms of properties of functors from certain monoidal categories. While an associative algebra is often understood as the endomorphism ring corresponding to a linear category with a single object, a multiplicative sequence of algebras is obtained from spaces of endomorphisms of tensor powers of an object in a strict monoidal category. For instance, the multiplicative sequence $\{\mbox{$k[S_n]$ $\vert$ $n\geq 0$}\}$ of symmetric group algebras equipped with algebra homomorphisms $k[S_m]\otimes k[S_n]\longrightarrow k[S_{m+n}]$ induced by the canonical inclusions $S_m\times S_n\hookrightarrow S_{m+n}$ is related to the free symmetric abelian monoidal $k$-linear category generated by one object, and Schur-Weyl dualities can be captured by properties of functors from this category to representation categories of general linear Lie algebras. There is also a quantized version of this developed in \cite{DAV}, which relates to multiplicative sequences of Hecke algebras. Similarly, there are multiplicative sequences associated to affine symmetric group algebras, degenerate affine Hecke algebras, braid group algebras and affine braid group algebras, each with its associated monoidal category (see \cite{DAV},  and also Davydov and Elbehiry \cite{DE}). In \cite{DAV}, the authors have also described the usual notions associated with algebras, such as generators and relations, or the free algebra associated to a vector space, in the setup of multiplicative sequences. Our objective is to develop Sweedler operations, coalgebra measurings and the enriched categories associated with them,   in the context of multiplicative sequences, as well the maps induced by them on homology, including with bimodule coefficients.  

\smallskip
The first step is to have a Hochschild theory for multiplicative sequences. If $A$ is an algebra, we note that the collection  $\{\mbox{$A_0:=k$, $A_i:=A$, $i\geq 1$}\}$ equipped with maps $\mu:A\otimes A\longrightarrow A$ given by the product $\mu$ on $A$ is not necessarily a multiplicative sequence in the sense of \eqref{1.3eef}. This is because the product $\mu:A\otimes A
\longrightarrow A$ is not in general an algebra homomorphism. However, this holds when $A$ is a commutative. We observe that this is also the property that allows one to use the shuffle product to induce an internal product structure on the Hochschild complex of a commutative algebra. If $A$, $B$ are $k$-algebras (not necessarily commutative), we recall that there is a shuffle product $sh:C_*(A)\otimes C_*(B)\longrightarrow C_*(A\otimes B)$ on Hochschild complexes (see, for instance, \cite[$\S$ 4.2.3]{LD}). In the case where $A=B$ and $A$ is commutative, this is composed with the map induced by the product $\mu:A\otimes A\longrightarrow A$ on $A$ to define an internal shuffle product (see \cite[$\S$ 4.2.6]{LD})
\begin{equation}
C_*(A)\otimes C_*(A) \xrightarrow{\qquad sh \qquad}C_*(A\otimes A)\xrightarrow{\qquad C_*(\mu)\qquad}C_*(A)
\end{equation} on the Hochschild complex of $A$. This suggests that the Hochschild theory of a multiplicative sequence $(A_\bullet,\tau)$ of algebras be set up in the following manner. For each $n\geq 0$, we have the Hochschild complex $C_*(A_n)$ corresponding to the algebra $A_n$. For $m,n\geq 0$, the shuffle product and the structure maps $\tau_{m,n}: A_m \otimes A_n \longrightarrow A_{m+n}$ of the multiplicative sequence $(A_\bullet,\tau)$ together induce a map of complexes
\begin{equation}\label{1.5}
\tau^{sh}_{m,n}:C_*(A_m)\otimes C_*(A_n) \xrightarrow{\qquad sh \qquad}C_*(A_m\otimes A_n)\xrightarrow{\qquad C_*(\tau_{m,n})\qquad}C_*(A_{m+n})
\end{equation} We let $\mathscr{M}ult_k$ denote the category of multiplicative sequences. Let $\mathcal C_k=\mathcal Ch(\mathcal Vect_k)$ be the monoidal category of chain complexes of $k$-vector spaces and let $\mathscr Gr(\mathcal C_k) $ be the category of non-negatively graded  objects in $\mathcal C_k$. In Section 2, we introduce the Hochschild theory on multiplicative sequences as a functor (see Definition \ref{DL2.2xr}) 
\begin{equation} Hoch:\mathscr{M}ult_k\longrightarrow \mathscr Gr\mathcal Alg(\mathcal C_k) \qquad ({A}_\bullet,\tau)\mapsto (\{C_*(A_n)\}_{n\geq0},\tau^{sh}=\{\tau^{sh}_{m,n}\}_{m,n\geq 0})
\end{equation}
taking values in the category $\mathscr Gr\mathcal Alg(\mathcal C_k) $ of non-negatively graded algebra objects in the monoidal category $\mathcal C_k$.

            
            \smallskip
            Now let $C$ be a cocommutative coalgebra. By a coalgebra measuring $(C,\Phi)$ between multiplicative sequences $(A_\bullet,\tau)$, $(A'_\bullet,\tau')\in \mathscr Mult_k$ (see Definition \ref{D3.r1}), we mean a collection $\Phi=\{\Phi_i: C\longrightarrow  {Hom}_k(A_i,A'_i)\}_{i\geq 0}$ of coalgebra measurings satisfying an additional compatibility condition between the coproduct $\Delta$ on $C$ and the structure maps $\tau=\{\tau_{m,n}\}_{m,n\geq 0}$. Our first main result  (see Theorem \ref{Thm3.6}) is that there is an induced linear map
            \begin{equation}\label{1.8}
            Hoch^\Phi: C\longrightarrow \mathscr Gr(\mathcal C_{k})(Hoch(A_\bullet),Hoch(A'_\bullet))
            \end{equation} which gives a coalgebra measuring $(C,Hoch^\Phi)$  between the graded algebras $Hoch(A_\bullet)$, $Hoch(A'_\bullet)\in \mathscr Gr\mathcal Alg(\mathcal C_k) $. Considering the definition of the product on $Hoch(A_\bullet)$ given in \eqref{1.5}, this builds on the results in \cite{AB}, where we showed that maps induced by coalgebra measurings on Hochschild homology behave well with respect to the shuffle product.

\smallskip In Section 3, we show that convolution with the coalgebra $C$ determines a functor $[C,-]:\mathscr Mult_k\longrightarrow \mathscr Mult_k$  and that coalgebra measurings from $A_\bullet$ to $A'_\bullet$ parametrized by $C$ are in one to one correspondence with maps from 
            $A_\bullet$ to $[C,A']_\bullet$ in $\mathscr Mult_k$ (see Proposition \ref{prop3.3}). In \cite{AJ}, Anel and Joyal showed that the convolution with a coalgebra on the category of $k$-algebras has a left adjoint, which they called the Sweedler product. Using the coalgebra $C$, we construct a functor
           $
C\Box-:\mathscr Mult_k\longrightarrow \mathcal Comm^{\geq 0}_k$  taking values in the category $\mathcal Comm^{\geq 0}_k$ of towers of commutative $k$-algebras. We note that $\mathcal Comm^{\geq 0}_k$ embeds canonically as  a full subcategory of $\mathscr Mult_k$ (see Section 3). Thereafter, we obtain a pair of adjoint functors (see Theorem \ref{Thm3.9})
            \begin{equation}\label{1.3}
            C\Box-:\mathscr Mult_k\longrightarrow \mathcal Comm^{\geq 0}_k\qquad [C,-]: \mathcal Comm^{\geq 0}_k\longrightarrow \mathscr Mult_k
            \end{equation} where the right adjoint $[C,-]: \mathcal Comm^{\geq 0}_k\longrightarrow \mathscr Mult_k$ is the restriction of $[C,-]:\mathscr Mult_k\longrightarrow \mathscr Mult_k$  to the full subcategory $\mathcal Comm^{\geq 0}_k$  of $\mathscr Mult_k$.
            
            \smallskip
            In Section 4, we construct the coalgebra $\mathcal M_c(A_\bullet, A_\bullet')$ that is universal among the (cocommutative) measurings between multiplicative sequences 
            $A_\bullet$, $A'_\bullet\in \mathscr{M}ult_k$. This leads to an enrichment $MULT_k$ of multiplicative sequences over the symmetric monoidal category $CoCo\mathcal Alg_k$ of cocommutative $k$-coalgebras. By considering the universal cocommutative measuring coalgebra $Q_c(Hoch(A_\bullet), Hoch(A'_\bullet) )$ between the graded algebras  $Hoch(A_\bullet)$ and $Hoch(A'_\bullet)$, we have another enrichment $\widetilde{MULT_k}$ of mutliplitcative sequences. Then, the  main result (see Theorem \ref{thm4.5}) of Section 4 is that we have a  $CoCo\mathcal Alg_k$-enriched functor $\gamma: MULT_k\longrightarrow \widetilde{MULT_k}$ that is identity on objects and whose maps on Hom objects are obtained from the linear maps induced as in \eqref{1.8}.

            \smallskip
            In Section 5, we consider coefficients for the Hochschild theory of multiplicative sequences. For this, we  introduce an appropriate notion of a bimodule $(U_\bullet=\{U_n\}_{n\geq 0},\vartheta=\{\vartheta_{m,n}\}_{m,n\geq 0})$ over a multiplicative sequence $(A_\bullet,\tau)\in \mathscr{M}ult_k$. Each $U_n$ is a bimodule over $A_n$ and we have additional maps $\vartheta_{m,n}:A_m\otimes U_n\longrightarrow U_{m+n}$ satisfying certain conditions that we describe in Definition \ref{D5.1kos}. Just as coalgebra measurings generalize homomorphisms of algebras, we know (see \cite{MB}, \cite{MH}) that there is a similar concept of comodule measurings, which generalizes maps between modules.  If $C$ is a cocommutative coalgebra and $P$ is a $C$-comodule, we introduce (see Definition \ref{defn5.8}) the notion of a comodule measuring $(P,\Psi)$ over $(C,\Phi)$ from $U_\bullet$ to $U'_\bullet$, where $U_\bullet$ and $U'_\bullet$ are bimodules over 
            $A_\bullet$, $A'_\bullet\in \mathscr{M}ult_k$ respectively, and $(C,\Phi)$ is a measuring from $A_\bullet$ to $A'_\bullet$. We construct the universal object $\mathcal P_C(U_\bullet,U_\bullet')$ among comodule measurings from $U_\bullet$ to $U'_\bullet$. The pairs $(A_\bullet,U_\bullet)$ where $A_\bullet\in \mathscr{M}ult_k$ and $U_\bullet$ is a bimodule over 
            $A_\bullet$, now form a ``global category of bimodules'' over multiplicative sequences, which we denote by $BMod_{\mathscr Mult_k}$ (see Theorem \ref{ils5.11}). The category $BMod_{\mathscr Mult_k}$ is enriched over the symmetric monoidal category $CoMod^c_k$ whose objects are pairs $(C,P)$, with $C$ a cocommutative $k$-coalgebra and $P$ a comodule over $C$. Similar to \eqref{1.5}, the shuffle product on Hochschild complexes with bimodule coefficients now gives us maps
            \begin{equation}\label{1.9}
            \varsigma_{m,n}:C_p(A_m,A_m)\otimes  C_q(A_n,U_n)\xrightarrow {\quad sh_{p,q}\quad}C_{p+q}(A_m\otimes A_n,A_m\otimes U_n)\xrightarrow{\quad C_{p+q}(\tau_{m,n},\vartheta_{m,n})\quad}  C_{p+q}(A_{m+n},U_{m+n})
            \end{equation} Using \eqref{1.9}, we show that $C_*(A_\bullet,U_\bullet)=\{C_*(A_n,U_n)\}_{n\geq 0}\in  \mathscr Gr(\mathcal C_k)$ becomes a graded left module over the graded algebra $Hoch(A_\bullet)=\left\{C_*(A_n)=C_*(A_n,A_n)\right\}_{n\geq 0}\in   \mathscr Gr\mathcal Alg(\mathcal C_k)$. Further, a comodule measuring $(P,\Psi)$ over $(C,\Phi)$ induces a comodule measuring 
            \begin{equation}
              \widetilde{\Psi}:P\longrightarrow \mathscr Gr(\mathcal C_k)(C_*(A_\bullet,U_\bullet),C_*(A_\bullet',U_\bullet')) 
            \end{equation} between the graded modules $C_*(A_\bullet,U_\bullet)$, $C_*(A'_\bullet,U'_\bullet)$ over the coalgebra measuring $(C,Hoch^\Phi)$ between graded algebras $Hoch(A_\bullet)$, $Hoch(A'_\bullet)$ described in \eqref{1.8}. By considering the universal measuring comodule between the modules $C_*(A_\bullet,U_\bullet)$ and $C_*(A'_\bullet,U'_\bullet)$, we have another enrichment $\widetilde{BMod}_{\mathscr Mult_k}$ of the global category of bimodules over multiplicative sequences, as well as a functor $(\gamma,\lambda): BMod_{\mathscr Mult_k}\longrightarrow \widetilde{BMod}_{\mathscr Mult_k}$ between categories enriched over $CoMod^c_k$ (see Theorem \ref{Thm5.17gm}). 
            
            \smallskip
            Finally, in Section 6, we consider the coalgebraic counterpart of multiplicative sequences of algebras. Accordingly, a comultiplicative sequence $(Q_\bullet,\delta)$ consists of a collection $Q_\bullet=\{\mbox{$Q_n$ $\vert$ $n\geq 0$} \}$ of coalgebras, with $Q_0=k$, and a collection of coalgebra morphisms 
            \begin{equation}
          \delta=\{  \delta_{m,n}:Q_{m+n}\longrightarrow Q_m\otimes Q_n \}_{m,n\geq 0}
            \end{equation} satisfying certain coassociativity conditions that we describe in Definition \ref{defn6.1}. By taking finite duals, we show that the usual adjunction between algebras and coalgebras extends to an adjunction between the category $\mathscr Mult_k$ of multiplicative sequences of algebras and the category $ \mathcal Co\mathscr Mult_k$ of comultiplicative sequences of coalgebras. In particular, this gives a  large class of examples of comultiplicative sequences of coalgebras obtained from the finite duals of objects in $\mathscr Mult_k$. 
            
            \smallskip If $C$ is a cocommutative coalgebra, its coproduct $\Delta:C\longrightarrow 
            C\otimes C$ is a morphism of coalgebras, and hence the collection $\{\mbox{$Q_0:=k$, $Q_i:=C$, $i\geq 1$}\}$ equipped with maps $\Delta:C\longrightarrow C
            \otimes C$ becomes a comultiplicative sequence. This suggests that a cocommutative algebra can be replaced by a comultiplicative sequence of coalgebras. Accordingly, we introduce (see Definition \ref{def6.6}) comultiplicative measuring sequences of coalgebras $(Q_\bullet,\mho_\bullet=\left\{\mho_n:Q_n\longrightarrow  {Hom}_k(A_n,A_n')\right\}_{n\geq 0})$ between objects
            $A_\bullet$, $A'_\bullet\in \mathscr Mult_k$. Further, for any  $(Q_\bullet,\delta)\in  \mathcal Co\mathscr Mult_k$, we have as in \eqref{1.3} a pair of adjoint functors
             \begin{equation}\label{1.12}
            (Q\Box-)_\bullet:\mathscr Mult_k\longrightarrow \mathcal Comm^{\geq 0}_k\qquad [Q,-]_\bullet: \mathcal Comm^{\geq 0}_k\longrightarrow \mathscr Mult_k
            \end{equation} For $A_\bullet\in \mathscr Mult_k$, we also show that the Sweedler product like construction $  (Q\Box A)_\bullet\in \mathcal Comm^{\geq 0}_k$  in \eqref{1.12}  gives a universal target for comultiplicative measuring sequences  from $A_\bullet$ to objects $B_\bullet\in \mathcal Comm^{\geq 0}_k$. We conclude by showing that if $A_\bullet,A'_\bullet\in \mathscr{M}ult_k$ and $(Q_\bullet,\mho_\bullet=\{\mho_n\}_{n\geq 0})$ is a comultiplicative measuring sequence of coalgebras from $A_\bullet$ to $A_\bullet'$, then the maps (see 
            Theorem \ref{Thm6.11})
            \begin{equation}
            Hoch_{n}^{\mho}:Q_n\longrightarrow  \mathcal C_k(C_*(A_n),C_*(A'_n)) \quad  x^n\longmapsto Hoch_{n}^{\mho}(x^n):=C^{\mho_n}_*(x^n):Hoch(A_\bullet)_n=C_*(A_n)\longrightarrow C_*(A_n')=Hoch(A'_\bullet)_n
            \end{equation} determine a comultiplicative measuring sequence of coalgebras between the graded algebras $Hoch(A_\bullet)$, $Hoch(A'_\bullet)\in \mathscr Gr\mathcal Alg(\mathcal C_k) $.

            \smallskip
            {\bf Acknowledgements:} Authors AB and SK were partially supported by ANRF Research Grant CRG/2023/004143 from the Government of India.

     \section{Hochschild functor on multiplicative sequences and coalgebra measurings}
   Let $k$ be a field. Unless otherwise mentioned, all algebras are taken over $k$. We write $\mathcal Vect_k$ for the category of $k$-vector spaces and $\mathcal Alg_k$ for the category of unital $k$-algebras. We begin by recalling the notion of a multiplicative sequence of $k$-algebras due to Davydov and Molev \cite{DAV}.  
   
\begin{defn}\label{def2.1}(see \cite[\S 2]{DAV})
A multiplicative sequence $({A}_\bullet,\tau)$ of $k$-algebras consists of:

\smallskip
(a) 
    A collection ${A}_\bullet=\{A_n \mid n \ge 0\}$ of associative unital $k$-algebras.

\smallskip
    (b)  A collection 
    $
        \tau=\{\tau_{m,n}: A_m \otimes A_n \longrightarrow A_{m+n}\}_{m,n\geq 0}
    $
    of unital algebra morphisms satisfying the following associativity condition, $i.e,$ the diagram 
\begin{equation}\label{diag1}
\begin{tikzcd}[column sep=4em, row sep=3em]
A_l \otimes A_m \otimes A_n
\arrow[r, "\tau_{l,m} \otimes A_n"]
\arrow[d, "A_l \otimes \tau_{m,n}"']
& A_{l+m} \otimes A_n \arrow[d, "\tau_{l+m,n}"] \\
A_l \otimes A_{m+n}
\arrow[r, "\tau_{l,m+n}"']
& A_{l+m+n}
\end{tikzcd}
\end{equation}
commutes for all $l,m,n \ge 0$. We always assume that $A_0=k$ and each $\tau_{0,n}=id:A_n=k\otimes A_n=A_0\otimes A_n\xrightarrow{\tau_{0,n}}A_n$ and similarly that each $\tau_{n,0}=id$.

Let $({A}_\bullet,\tau)$, $({A}_\bullet',\tau')$ be multiplicative sequences of algebras. A morphism $f_\bullet:(A_\bullet,\tau)\longrightarrow (A'_\bullet,\tau')$ of multiplicative sequences is a collection $\{f_n:A_n\longrightarrow A'_n\}_{n\geq 0}$  of algebra morphisms such that the following diagram 
\begin{equation}\label{diag2}
\begin{tikzcd}[column sep=3em, row sep=2em]
A_m \otimes A_n
  \arrow[r, "f_m \otimes f_{n}"]
  \arrow[d, "\tau_{m,n}"']
&
A'_{m} \otimes A'_n
  \arrow[d, "\tau_{m,n}'"]
\\
A_{m+n}
  \arrow[r, "f_{m+n}"']
&
A'_{m+n}
\end{tikzcd}
    \end{equation}
commutes for all $m,n\geq 0.$ We denote by $\mathscr{M}ult_k$ the category of multiplicative sequences of $k$-algebras.  We will often denote an 
object $({A}_\bullet,\tau)\in \mathscr{M}ult_k$ simply by ${A}_\bullet$. 
\end{defn}

We now let $\mathcal C_{k}:=\mathcal Ch(\mathcal Vect_k)$ be the category of chain complexes of $k$-vector spaces. We recall (see, for instance, \cite[\S 1.1.3]{LD}) that for a $k$-algebra $A$, its Hochschild complex $(C_*(A),b)\in \mathcal C_k$ is defined as follows
\begin{equation}
    C_*(A):\hspace{.5cm}\cdots\xrightarrow[\qquad]{b}C_n(A)\xrightarrow[\qquad]{b}C_{n-1}(A)\xrightarrow[\qquad]{b}\cdots \xrightarrow[\qquad]{b}C_0(A),
\end{equation}
where $C_n(A)=A\otimes A^{\otimes n}$ and the differentials are given by $$b:C_n(A)\longrightarrow C_{n-1}(A)$$ 
\begin{equation}
(a_0,a_1,\ldots,a_n)
\longmapsto(a_0a_1,a_2,\ldots,a_n)
+ \sum_{i=1}^{n-1} (-1)^i (a_0,a_1,\ldots,a_i a_{i+1},\ldots,a_n)
+ (-1)^n (a_n a_0,a_1,\ldots,a_{n-1}),
\end{equation}
where we write $(a_0,a_1,\ldots,a_n)$ for the element $(a_0\otimes a_1\otimes ...\otimes a_n)\in C_n(A)=A^{\otimes n+1}$.  The association 
$A\mapsto C_*(A)$ determines a functor $\mathcal Alg_k\longrightarrow \mathcal C_k$ to the category of chain complexes.

\smallskip
We now recall the shuffle product on Hochschild complexes. 
  The symmetric group  $S_n$ acts on $C_n(A)=A\otimes A^{\otimes n}$   (see \cite[\S 4.2.1]{LD}) as 
\begin{equation}\label{act2}
    \sigma\cdot(a_0,a_1,a_2,\ldots,a_n)=(a_0,a_{\sigma^{-1}(1)},a_{\sigma^{-1}(2)},\ldots,a_{\sigma^{-1}(n)})
\end{equation} for $\sigma\in S_n$. 
 By definition, a $(p,q)$-shuffle is a permutation $\sigma$ of $\{1,2,\ldots p,p+1,\ldots,p+q\}$ such that $\sigma(1)<\sigma(2)<\ldots<\sigma(p)$ and $\sigma(p+1)<\sigma(p+2)<\ldots<\sigma(p+q)$. Let $S_{p,q}$ denote the set of $(p,q)$-shuffles.
 Then, the shuffle product $sh:C_*(A)\otimes C_*(A')
\longrightarrow C_*(A\otimes A')$ (see \cite[\S 4.2]{LD}) on Hochschild complexes is determined by 
\begin{equation*}
    sh_{p,q}:C_p(A)\otimes C_q(A')\longrightarrow C_{p+q}(A\otimes A'),
\end{equation*}
\begin{equation}\label{SH}
     (a_0,a_1,\ldots,a_p)\otimes(a'_0,a'_1,\ldots,a'_q)\longmapsto \sum_{\sigma\in S_{p,q}}sgn(\sigma)\sigma\cdot(a_0\otimes a'_0,a_1\otimes1,\ldots,a_p\otimes1,1\otimes a'_1,\ldots,1\otimes a'_q).
\end{equation}
Our first main objective is to define a Hochschild functor on multiplicative sequences of algebras. We let  $\mathscr Gr(\mathcal C_k)$ denote the category
of non-negatively graded objects in $\mathcal C_k$. We also know that the tensor product of chain complexes makes  $\mathcal C_k$ into a symmetric monoidal category. Accordingly, we can consider the category  $\mathscr Gr\mathcal Alg(\mathcal C_k)$ of non-negatively graded algebra objects in 
$\mathcal C_k$. In other words, an object $T=(\{T_{n,*}\}_{n\geq 0},\{\mu_{m,n}\}_{m,n\geq 0} )\in \mathscr Gr\mathcal Alg(\mathcal C_k)$ is given by a collection $\{T_{n,*}\}_{n\geq 0}$ of chain complexes along with maps
\begin{equation}
\mu_{m,n}:T_{m,*}\otimes T_{n,*}\longrightarrow T_{m+n,*}\qquad m,n\geq 0
\end{equation} in $\mathcal C_k$ such that $\underset{n\geq 0}{\bigoplus} T_{n,*}$ becomes an algebra object in the symmetric monoidal category 
$\mathcal C_k$. We are now ready to introduce the Hochschild functor on multiplicative sequences that takes values in graded algebra objects in $\mathcal C_k$.

\begin{defn}\label{DL2.2xr} Let $({A}_\bullet=\{A_n\}_{n\geq 0},\tau=\{\tau_{m,n}\}_{m,n\geq 0})$ be a multiplicative sequence of algebras.  For each 
algebra $A_n$, we consider its Hochschild complex $C_*(A_n)\in \mathcal C_k$. For $m$, $n\geq 0$, we consider the composition
\begin{equation}\label{2.7yu}
\tau_{m,n}^{sh}:C_*(A_m)\otimes C_*(A_n)\xrightarrow{\qquad sh\qquad }C_*(A_m\otimes A_n)\xrightarrow{\qquad C_*(\tau_{m,n})\qquad}C_*(A_{m+n})
\end{equation} in $\mathcal C_k$, where the first map in \eqref{2.7yu} is given by the shuffle product and the second map in \eqref{2.7yu} is induced by the morphism $\tau_{m,n}: A_m\otimes A_n\longrightarrow A_{m+n}$ of algebras. The association
\begin{equation} Hoch:\mathscr{M}ult_k\longrightarrow \mathscr Gr\mathcal Alg(\mathcal C_k) \qquad ({A}_\bullet,\tau) \mapsto (\{C_*(A_n)\}_{n\geq0},\tau^{sh}=\{\tau^{sh}_{m,n}\}_{m,n\geq 0})
\end{equation} will be called the Hochschild functor on multiplicative sequences of algebras. 
\end{defn}

For $({A}_\bullet,\tau)\in \mathscr{M}ult_k$, we note here that the associativity of the multiplication $\tau^{sh}$ on the graded algebra object $Hoch(A_\bullet)\in  \mathscr Gr\mathcal Alg(\mathcal C_k)$ follows from the associativity condition \eqref{diag1} in Definition \ref{def2.1}.

\smallskip

       Let $A$, $A'$ be $k$-algebras. We recall (see Sweedler \cite{SW}) that a coalgebra measuring from $A$ to $A'$  consists of  a pair $(C,\Phi)$, where $C$ is a $k$-coalgebra and 
       $\Phi:C\longrightarrow  {Hom}_{k}(A,A')$ is a $k$-linear map that satisfies
    \begin{equation}\label{eq3.1}
        \Phi(x)(a_1a_2)=\sum(\Phi(x_{(1)})(a_1))(\Phi(x_{(2)})(a_2))\qquad\Phi(x)(1_A)=\epsilon(x)1_{A'} \qquad x\in C, a_1,a_2\in A
    \end{equation} where   the coproduct $\Delta$ on $C$ is written as $\Delta(x)=\sum x_{(1)}\otimes x_{(2)}$ for all $x\in C$, and $\epsilon:C\longrightarrow k$ denotes the counit on
    $C$. We will often suppress the summation and write the coproduct simply as $\Delta(x):=x_{(1)}\otimes x_{(2)}$. By abuse of notation, the map  $\Phi:C\longrightarrow  {Hom}_{k}(A,A')$ will often be understood also as a map $\Phi:C\otimes A\longrightarrow A'$. 
    
    \smallskip
    In this section as well as in the rest of this paper, we will always take $C$ to be a cocommutative $k$-coalgebra. We now introduce the notion of coalgebra measuring between multiplicative sequences of algebras.

    \begin{defn}\label{D3.r1}
        Let $(A_\bullet,\tau)$, $(A'_\bullet,\tau')$ be multiplicative sequences of algebras. A coalgebra measuring $(C,\Phi)$ from $(A_\bullet,\tau)$ to $(A'_\bullet,\tau')$ consists of 
        \begin{enumerate}[(a)]
            \item a cocommutative $k$-coalgebra $(C,\Delta,\epsilon)$, and
            \item a collection $\Phi=\{\Phi_i: C\longrightarrow  {Hom}_k(A_i,A'_i)\}_{i\geq 0}$ of $k$-linear maps such that 
            
            \begin{enumerate}[(1)]

            \item for $i=0$, we have
            \begin{equation}\label{ut}
            \Phi_0(x):A_0=k\longrightarrow k=A'_0\qquad s\mapsto \epsilon(x)s
            \end{equation} for each $x\in C$. 
            
            \item for $i\geq 0$, each $(C,\Phi_i)$ is a coalgebra measuring from $A_i$ to $A'_i$. In other words, for each $i\geq0$, the map $\Phi_i:C\longrightarrow  {Hom}_k(A_i,A'_i)$ satisfies  
            \begin{equation}
                \Phi_i(x)(a^i_1a^i_2)=\sum\Phi_i(x_{(1)})(a^i_1)\Phi_i(x_{(2)})(a^i_2)\qquad\textit{ and }\qquad\Phi_i(x)1_{A_i}=\epsilon(x)1_{A'_i}
            \end{equation}
            for all $a^i_1,a^i_2\in A_i$ and $x\in C$.
            
            \item
              the collection $\Phi=\{\Phi_i:C\longrightarrow  {Hom}_k(A_i,A'_i)\}_{i\geq 0}$ of measurings is compatible with the multiplicative structures $\tau=\{\tau_{m,n}\}_{m,n\geq 0}$ and 
              $\tau'=\{\tau'_{m,n}\}_{m,n\geq 0}$ on $A_\bullet$ and $A'_\bullet$ respectively, i.e, the following diagram
        \begin{equation}
            \begin{tikzcd}\label{diag4}
A_m \otimes A_n
  \arrow[r, "\tau_{m,n}"]
  \arrow[d, "\sum \Phi_m(x_{(1)})\otimes\Phi_n(x_{(2)})"']
&
A_{m+n}
  \arrow[d, "\Phi_{m+n}(x)"]
\\
A'_{m}\otimes A'_n
  \arrow[r, "\tau'_{m,n}"']
&
A'_{m+n}
\end{tikzcd}
        \end{equation}
       commutes for all $m,n\geq 0$ and   $x\in C$. 
       
               \end{enumerate}
               
               \end{enumerate}
 \end{defn}   

If we treat $\Phi=\left\{\Phi_i\right\}_{i\geq 0}$ appearing in Definition \ref{D3.r1} as a collection of maps $\left\{\Phi_i:C\otimes A_i\longrightarrow A'_i\right\}_{i\geq 0}$, we note that the condition in \eqref{diag4} may also be expressed as
    \begin{equation}\label{diag4prm}
        \Phi_{m+n}\big(x\otimes\tau_{m,n}(a^m\otimes a^n)\big)=\sum \tau_{m,n}'\Big(\Phi_m(x_{(1)}\otimes a^m)\otimes\Phi_n(x_{(2)}\otimes a^n)\Big),
    \end{equation}
    for all  $a^m\in A_m$, $a^n\in A_n$, $x\in C$, and  $m,n\geq 0$.  
     
\smallskip

\begin{eg}
    \emph{Let $C$ be the cocommutative $k$-coalgebra with vector space basis $\{u,v\}$, where the comultiplication $\Delta$ and the counit $\epsilon$ are determined by}
\[
\begin{array}{c}
     \Delta(u)=u\otimes u\qquad \Delta(v)=v\otimes u+u\otimes v \qquad \epsilon(u)=1 \qquad \epsilon(v)=0\\
\end{array}
\]
\emph{Let $({A}_\bullet,\tau)\in \mathscr Mult_k$ and $d=\{d_n:A_n
\longrightarrow A_n\}_{n\geq 0}$ be a sequence of derivations which satisfy}
\begin{equation}
 d_0=0\qquad\mbox{and}\qquad d_{m+n}\circ\tau_{m,n}=\tau_{m,n}\circ (d_{m}\otimes A_n+A_m\otimes d_n)\qquad m,n\geq 0
\end{equation}
\emph{This determines a coalgebra measuring $(C,\Phi=\{\Phi_n\}_{n\geq 0})$ from $({A}_\bullet,\tau)$ to itself by setting $\Phi_n(u):=id_{A_n}$ and $\Phi_n(v):=d_n$ for $n\geq 0$. For instance, let $A$ be a commutative $k$-algebra and consider ${A}_\bullet=\{A_n\}_{n\geq 0}$ where $A_0:=k$ and $A_n:=A[x_1,x_2,\ldots,x_n]$ for  $n\geq 1$. Then, ${A}_\bullet$ together with the collection $\tau=\{\tau_{m,n}:A_m\otimes A_n\longrightarrow A_{m+n}\}_{m,n\geq 0}$
     given by $
       \tau_{m,n}(f(x_1,x_2,\ldots,x_m)\otimes g(x_1,x_2,\ldots,x_n)):=f(x_1,x_2,\ldots,x_m)\cdot g(x_{m+1},\ldots,x_{m+n})
 $ for $m,n\geq 1$  forms a multiplicative sequence of algebras. It may be verified that the collection $\Phi=\left\{\Phi_n:C\longrightarrow Hom_k(A_n,A_n)\right\}_{n\geq 0}$ determined by setting $\Phi_n(u):=id_{A_n}$ for $n\geq 0$ and}
\begin{equation}
   \Phi_0(v)=0\qquad \Phi_n(v)=\left(\sum_{i=1}^n\frac{ \partial}{\partial x_i}\right):A_n\longrightarrow A_n \qquad \forall n\geq 1 
\end{equation}
\emph{forms a coalgebra measuring from $(A_\bullet,\tau)$ to itself in the sense of Definition \ref{D3.r1}.} 
\end{eg}

\begin{eg}
\emph{Let $k=\mathbb Q$ and let $C$ be the cocommutative $k$-coalgebra with vector space basis $\{x,y\}$, where  the comultiplication $\Delta$ and the counit $\epsilon$ are determined by}
\begin{equation}
\Delta(x)=x\otimes x-y\otimes y\qquad \Delta(y)=x\otimes y+y\otimes x\qquad \epsilon(x)=1\qquad \epsilon(y)=0
\end{equation}
\emph{Let $(A_\bullet,\tau)\in \mathscr{M}ult_k$ be the multiplicative sequence given by setting $A_n:=\mathbb Q[S_n]$, along with structure maps
$\tau_{m,n}:\mathbb Q[S_m]\otimes \mathbb Q[S_n]\longrightarrow \mathbb Q[S_{m+n}]$ induced by the canonical embeddings given by }
\begin{equation}
S_m\times S_n
\longrightarrow S_{m+n}\qquad (\sigma,\tau)\mapsto (\sigma\times \tau)(1,...,m,m+1,...,m+n):=(\sigma(1),...,\sigma(m),
m+\tau(1),...,m+\tau(n))
\end{equation}
\emph{Now let $\mathscr C$ be a strict $\mathbb Q[i]$-linear abelian monoidal category, where $i=\sqrt{-1}$.   We fix some object $T\in \mathscr C$ and put $
A'_n:=\mathscr C(T^{\otimes n},T^{\otimes n})$ for $n\geq 1$. We put $A'_0:=\mathbb Q$. Then $A'_\bullet:=\{A'_n\}_{n\geq 0}$ forms a multiplicative sequence of algebras with structure maps}
\begin{equation} \tau'_{m,n}:A'_m\otimes_{\mathbb Q}A'_n\longrightarrow  A'_m\otimes_{\mathbb Q[i]}A'_n\longrightarrow A'_{m+n}\qquad m,n\geq 1
\end{equation}
\emph{induced by the tensor product on morphisms (see \cite[$\S$ 2.1]{DAV}). For $n\geq 1$, we let 
$\rho_n:\mathbb Q[S_n]\longrightarrow \mathscr C(T^{\otimes n},T^{\otimes n})=A'_n$ be the $\mathbb Q$-algebra homomorphism induced by the canonical map
$S_n\longrightarrow Aut_{\mathscr C}(T^{\otimes n},T^{\otimes n})$. We let $\rho_0$ be the identity map on $\mathbb Q$.  It may be verified that the collection $\Phi^T=\left\{\Phi_n^T:C\longrightarrow Hom_k(A_n,A'_n)\right\}_{n\geq 0}$ determined by setting}
\begin{equation}
 \Phi^T_n(x)(\sigma):=\frac{1}{2}\Big(1+sgn(\sigma)\Big)\rho_n(\sigma)\qquad \Phi^T_n(y)(\sigma):=\frac{i}{2}\Big(1-sgn(\sigma)\Big)\rho_n(\sigma) \qquad \forall \textrm{ }n\geq 0
\end{equation}
\emph{forms a coalgebra measuring from $(A_\bullet,\tau)$ to $(A'_\bullet,\tau')$ in the sense of Definition \ref{D3.r1}.} 
\end{eg}

\begin{eg}\label{epl2.6}
\emph{Let $k=\mathbb F_2$, the finite field with $2$ elements and let $\mathcal A_2$ be the mod-$2$ Steenrod algebra. As a $k$-algebra, $\mathcal A_2$ is generated by symbols
$\{\mbox{$Sq^r$ $\vert$ $r\geq 0$}\}$, where $Sq^0$ is treated as $1$, subject to the Adem relations (see, for instance, \cite[$\S$ 1.1]{Baues}). If $X$ is any space, we know that its cohomology ring $H(X;\mathbb F_2):=\underset{n\geq 0}{\bigoplus}H^n(X;\mathbb F_2)$ with coefficients in $\mathbb F_2$ has an $\mathcal A_2$-module structure determined by  Steenrod operations $Sq^r:H^n(X;\mathbb F_2)\longrightarrow H^{n+r}(X;\mathbb F_2)$ which satisfy the Cartan formula}
\begin{equation}\label{rel220f}
Sq^r(u\cup v)=\underset{i+j=r}{\sum} Sq^i(u)\cup Sq^j(v) \qquad u,v\in H^*(X;\mathbb F_2), m\geq 0
\end{equation}\emph{where $\cup$ denotes the cup product that makes $H(X;\mathbb F_2)$ an algebra. We also have $Sq^r(u)=u\cup u$ if $r=|u|$ and 
$Sq^r(u)=0$ if $r>|u|$, where $|u|$ denotes the degree of $u\in H^*(X;\mathbb F_2)$. If $f:X'\longrightarrow X$ is a map of spaces, the pullback $f^*:H(X;
\mathbb F_2)\longrightarrow H(X';\mathbb F_2)$ is a map of $\mathcal A_2$-modules. For spaces $X$, $Y$ and the canonical cross product $\times: H(X;\mathbb F_2)
\otimes H(Y;\mathbb F_2)\longrightarrow H(X\times Y;\mathbb F_2)$,  we have}
\begin{equation}\label{rel221f} \small
Sq^r(u \times v)=Sq^r(p_X^*(u)\cup p_Y^*(v))=\underset{i+j=r}{\sum} Sq^i(p_X^*(u))\cup Sq^j(p_Y^*(v))= \underset{i+j=r}{\sum}p_X^*(Sq^i(u))\cup p_Y^*(Sq^j(v))= \underset{i+j=r}{\sum}Sq^i(u)\times Sq^j(v)
\end{equation}
\emph{for $u\in H^*(X;\mathbb F_2)$, $v\in H^*(Y;\mathbb F_2)$, where $p_X$ and $p_Y$ are the coordinate projections. We know that $\mathcal A_2$ is a cocommutative Hopf algebra with coproduct $\Delta$ and counit $\epsilon$ determined by}
\begin{equation}\label{cop222}
\Delta(Sq^r)=\underset{i+j=r}{\sum}Sq^i\otimes Sq^j\qquad \epsilon(Sq^0)=1 \qquad \epsilon(Sq^r)=0, \textrm{ }r>0
\end{equation}
\emph{Now let $\{X_n\}_{n\geq 0}$ be a ``multiplicative sequence'' of spaces with $X_0=\{\ast\}$ and connecting maps $t_{m,n}:X_{m+n}\longrightarrow X_m\times X_n$, for $m$, $n\geq 0$ satisfying the obvious ``associativity'' and ``unit'' condition. For instance, we could have $X_m:=BP_m$, the classifying space of the pure braid group on $m$ strands, the collection of which is equipped with canonical maps $BP_{m+n}\longrightarrow B(P_m\times P_n)=BP_m\times BP_n$ (see \cite[$\S$ 1.3]{KT}).  In general, by setting $A_m:=H(X_m,\mathbb F_2)$ along with maps}
\begin{equation}
\tau_{m,n}:A_m\otimes A_n=H(X_m;\mathbb F_2)\otimes H(X_n;\mathbb F_2)\xrightarrow{\quad\times\quad} H(X_m\times X_n;\mathbb F_2)\xrightarrow{\quad t_{m,n}^*\quad }H(X_{m+n};\mathbb F_2)=A_{m+n}
\end{equation}\emph{we obtain a multiplicative sequence $H(X_\bullet)=(H(X_m;\mathbb F_2),\tau=\{\tau_{m,n}\}_{m,n\geq 0})$ of algebras. Using \eqref{rel220f}, 
\eqref{rel221f} and the expression for the coproduct in \eqref{cop222}, we see that the collection}
\begin{equation}
\Phi_m:\mathcal A_2\longrightarrow Hom_k(H(X_m;\mathbb F_2), H(X_m;\mathbb F_2))\qquad \Phi_m(Sq^r)(u):=Sq^r(u)
\end{equation}\emph{determines a measuring from $H(X_\bullet)$ to itself in  the sense of Definition \ref{D3.r1}.}
\end{eg}

\begin{eg}
\emph{Let $k=\mathbb F_p$ be the field of order $p$, where $p>2$ is a prime, and let $\mathcal A_p$ be the mod-$p$ Steenrod algebra. As a $k$-algebra,  $\mathcal A_p$ is generated by symbols
$\{\mbox{$\beta$, $P^r$ $\vert$ $r\geq 0$}\}$ (with $P^0=1$) subject to the Adem relations (see, for instance, \cite[$\S$ 1.1]{Baues}). Also, $\mathcal A_p$ is a cocommutative Hopf algebra with coproduct $\Delta$ and counit $\epsilon$ determined by}
\begin{equation}
\Delta(P^r)=\underset{i+j=r}{\sum} P^i\otimes P^j\qquad \Delta(\beta)=\beta\otimes 1+1\otimes \beta \qquad \epsilon(\beta)=1 \qquad \epsilon(P^0)=1\qquad \epsilon(P^r)=0, \textrm{ }r>0
\end{equation}
\emph{If $X$ is any space, its cohomology groups $H^*(X;\mathbb F_p)$ with coefficients in $\mathbb F_p$ are equipped with Steenrod operations 
$P^r:H^n(X;\mathbb F_p)\longrightarrow H^{n+2r(p-1)}(X;\mathbb F_p)$ and $\beta:H^n(X;\mathbb F_p)\longrightarrow H^{n+1}(X;\mathbb F_p)$ satisfying the Cartan relations}
\begin{equation}
P^r(u\cup v)=\underset{i+j=r}{\sum} P^i(u)\cup P^j(v) \qquad u,v\in H^*(X;\mathbb F_p), r\geq 0
\end{equation}\emph{as well as $P^r(u)=u^p$ if $2r=|u|$ and $P^r(u)=0$ if $2r>|u|$, where $|u|$ denotes the degree of $u\in H^*(X;\mathbb F_p)$. Let $H_{ev}(X;\mathbb F_2)
:=\underset{m\geq 0}{\oplus}H_{ev}^{m}(X;\mathbb F_p)$ be the ring of cohomology groups of $X$ in even degrees, where  $H^m_{ev}(X;\mathbb F_p):=H^{2m}(X;\mathbb F_p)$. Then, $H_{ev}(X;\mathbb F_p)$ is an $\mathcal A_p$-module, with $\beta$ acting as $0$. Now let $\{X_n\}_{n\geq 0}$ be a ``multiplicative sequence'' of spaces as in Example \ref{epl2.6}. By setting $A_m:=H_{ev}(X_m,\mathbb F_p)$ along with maps}
\begin{equation}
\tau_{m,n}:A_m\otimes A_n=H_{ev}(X_m;\mathbb F_p)\otimes H_{ev}(X_n;\mathbb F_p)\xrightarrow{\quad\times\quad} H_{ev}(X_m\times X_n;\mathbb F_p)\xrightarrow{\quad t_{m,n}^*\quad }H_{ev}(X_{m+n};\mathbb F_p)=A_{m+n}
\end{equation}\emph{we obtain a multiplicative sequence $H_{ev}(X_\bullet)=(H_{ev}(X_m;\mathbb F_p),\tau=\{\tau_{m,n}\}_{m,n\geq 0})$ of algebras. By the same reasoning as in 
Example \ref{epl2.6}, we see that the collection}
\begin{equation}
\Phi_m:\mathcal A_p\longrightarrow Hom_k(H_{ev}(X_m;\mathbb F_p), H_{ev}(X_m;\mathbb F_p))\qquad \Phi_m(P^r)(u):=P^r(u) \qquad \Phi_m(\beta)(u):=0
\end{equation}\emph{determines a measuring from $H_{ev}(X_\bullet)$ to itself in  the sense of Definition \ref{D3.r1}.}
\end{eg}

     Our main aim in this section is to show that a coalgebra measuring between multiplicative sequences $A_\bullet, A'_\bullet\in \mathscr Mult_k$ induces a coalgebra measuring between the graded algebras $Hoch(A_\bullet)$ and $ Hoch(A'_\bullet)$. We recall from \cite[Proposition 2.2]{AB} that a coalgebra measuring between algebras induces a morphism between their Hochschild chain complexes.

  \begin{lem}\label{prop3.4}
    Let $A_\bullet,A'_\bullet\in \mathscr{M}ult_k$ be multiplicative sequences of algebras. Let $(C,\Phi=\{\Phi_i\}_{i\geq 0})$ be a  coalgebra measuring from $A_\bullet$ to $A_\bullet'$, with $C$ cocommutative. Then for each $x\in C$, there is a morphism $Hoch^\Phi(x):Hoch(A_\bullet)\longrightarrow Hoch(A'_\bullet)$ in $\mathscr Gr(\mathcal C_{k})$, given by the collection    
    \begin{equation}\label{pi3}
    \begin{array}{c}
    Hoch^\Phi(x)=\{Hoch^\Phi_{i}(x):=C^{\Phi_i}_*(x):Hoch(A_\bullet)_i=C_*(A_i)\longrightarrow C_*(A'_i)=Hoch(A'_\bullet)_i\}_{i\geq 0}
    \\
       C^{\Phi_i}_n(x):C_n(A_i)=A_i^{\otimes n+1}\longrightarrow A'^{\otimes n+1}_i=C_n(A'_i) \quad(a^i_0,a^i_1,\ldots,a^i_n)\longmapsto \sum\left(\Phi_i(x_{(1)})(a^i_0),\Phi_i(x_{(2)})(a^i_1),\ldots, \Phi_i(x_{(n+1)})(a^i_n)\right) \\
        \end{array}
    \end{equation}
    for $(a^i_0,a^i_1,\ldots,a^i_n)\in C_n(A_i)$ and where  $\Delta^{n}(x)=\sum x_{(1)}\otimes\ldots\otimes x_{(n+1)}$.
 \end{lem}

\begin{proof}
By Definition \ref{D3.r1}, we know that each $\Phi_i:C\longrightarrow Hom_k(A_i,A'_i)$ is a coalgebra measuring. Accordingly, it follows from \cite[Proposition 2.2]{AB} that for each 
$x\in C$, the maps $ C^{\Phi_i}_n(x):C_n(A_i)=A_i^{\otimes n+1}\longrightarrow A'^{\otimes n+1}_i=C_n(A'_i) $ as described in \eqref{pi3} determine  a morphism $C^{\Phi_i}_*(x): C_*(A_i)\longrightarrow C_*(A'_i)$ between Hochschild complexes. Together, the collection  $Hoch^\Phi(x)=\{Hoch^\Phi_{i*}(x)=C^{\Phi_i}_*(x)\}_{i\geq 0}$ determines a morphism 
 $Hoch^\Phi(x):Hoch(A_\bullet)\longrightarrow Hoch(A'_\bullet)$ in $\mathscr Gr(\mathcal C_{k})$.   
\end{proof}

We are now ready to show that a coalgebra measuring $(C,\Phi=\{\Phi_i\}_{i\geq 0})$ between multiplicative sequences $A_\bullet,A'_\bullet\in \mathscr{M}ult_k$ determines a coalgebra measuring from $Hoch(A_\bullet)$ to $Hoch(A'_\bullet)$. For this, we need to make explicit the notion of a coalgebra measuring between graded algebras in the monoidal category $\mathcal C_k$.

\begin{defn}\label{def3.5}
     Let $T=\{T_{n,*}\}_{n\geq 0}, T'=\{T'_{n,*}\}_{n\geq 0}$ be graded algebras in $\mathcal C_{k}$, equipped with multiplication maps
     \begin{equation*}
         \mu=\left\{\mu_{m,n}:T_{m,*}\otimes T_{n,*}\longrightarrow T_{m+n,*}\right\}_{m,n\geq 0}\qquad\textit{and}\qquad\mu'=\left\{\mu'_{m,n}:T'_{m,*}\otimes T'_{n,*}\longrightarrow T'_{m+n,*}\right\}_{m,n\geq 0}
     \end{equation*}respectively. A coalgebra measuring of graded algebras in $\mathcal C_k$  from $T$ to $T'$ is a pair $(C,\Xi)$ consisting of
     \begin{enumerate}[(a)]
         \item a cocommutative coalgebra $(C,\Delta,\epsilon)$, and
         \item a map $\Xi:C\longrightarrow \mathscr Gr(\mathcal C_{k})(T,T')$ such that the collection  $\left\{\Xi(x)=\left\{\Xi_{i}(x):T_{i,*}\longrightarrow T'_{i,*}\right\}_{i\geq 0}\right\}_{x\in C}$ of $k$-linear maps satisfies
     \begin{equation}\label{pi15}
        \Xi_{m+n}(x)\left(\mu_{m,n}(t^m_{p}\otimes t^n_{q})\right)=\sum \mu'_{m,n}\left(\left(\Xi_m(x_{(1)})(t^m_{p})\right)\otimes\left(\Xi_n(x_{(2)})(t^n_{q}) \right)\right) \qquad \Xi_0(x)(1_{T_{0,0}})=\epsilon(x)(1_{T'_{0,0}})
     \end{equation}
     for all $m,n,p,q\geq 0$ and  $t^m_{p}\in T_{m,p}$, $t^n_{q}\in T_{n,q}$.
     \end{enumerate}
\end{defn}

\begin{Thm}\label{Thm3.6}
     Let $A_\bullet,A'_\bullet\in \mathscr{M}ult_k$ be multiplicative sequences of algebras. Let $C$ be a cocommutative coalgebra and $(C,\Phi=\{\Phi_i\}_{i\geq 0})$ be a coalgebra measuring from $A_\bullet$ to $A'_\bullet$. Then, the map
\begin{equation} \label{eq3.26pp}
\begin{array}{c} Hoch^\Phi: C\longrightarrow \mathscr Gr(\mathcal C_{k})(Hoch(A_\bullet),Hoch(A'_\bullet)) \\ x
\mapsto Hoch^\Phi(x)=\{Hoch^\Phi_{i}(x):=C^{\Phi_i}_*(x):Hoch(A_\bullet)_i=C_*(A_i)\longrightarrow C_*(A'_i)=Hoch(A'_\bullet)_i\}_{i\geq 0}\\
\end{array}
\end{equation}  determines a coalgebra measuring $(C,Hoch^\Phi)$ of graded algebras in $\mathcal C_k$  from $Hoch(A_\bullet)$ to $Hoch(A'_\bullet)$.
\end{Thm}

\begin{proof} 
For $r\geq 0$ we note that the maps $C_r(\tau_{m,n}):C_r(A_m  \otimes A_n)\longrightarrow C_r(A_{m+n})$ induced by the morphism 
    $\tau_{m,n}:A_m\otimes A_n\longrightarrow A_{m+n}$ of algebras  preserve the action (described in \eqref{act2}) of $S_r$ on the terms in the Hochschild complexes. The same applies to the  maps $C_r(\tau'_{m,n}):C_r(A'_m  \otimes A'_n)\longrightarrow C_r(A'_{m+n})$. 
    Further, since $C$ is cocommutative, for any $i\geq 0$ and $x\in C$, the maps  in
\begin{equation}
Hoch^\Phi_{i}(x)=C^{\Phi_i}_*(x):Hoch(A_\bullet)_i=C_*(A_i)\longrightarrow C_*(A'_i)=Hoch(A'_\bullet)_i
\end{equation} described in \eqref{pi3} preserve the action of permutation groups on terms in the Hochschild complexes.  
      
     \smallskip To show that  $(C,Hoch^\Phi)$ is a measuring, we have to verify the condition in \eqref{pi15}. If either $m=0$ or $n=0$, this is clear from \eqref{ut}. Let $m,n\geq 1$, and consider $(a_0^m,\ldots, a_p^m)\in C_p(A_m),(a^n_0,\ldots, a^n_q)\in C_q(A_n)$. Then, for each $x\in C$, we have
   \begin{align*}
       &Hoch^\Phi_{m+n}(x)\left(\tau^{sh}_{m,n}\left((a_0^m,\ldots, a_p^m)\otimes (a_0^n,\ldots, a_q^n)\right)\right)\notag\\
       &=C^{\Phi_{m+n}}_{p+q}(x)\left(\sum_{\sigma\in S_{p,q}}sgn(\sigma)\left(C_{p+q}(\tau_{m,n})\left(\sigma\cdot\left(a_0^m\otimes a_0^n,a_1^m\otimes1,\ldots,a_p^m\otimes 1,1\otimes a_1^n,\ldots ,1\otimes a_q^n\right)\right)\right)\right)\notag\\
       &=C^{\Phi_{m+n}}_{p+q}(x)\left(\sum_{\sigma\in S_{p,q}}sgn(\sigma)\left(\sigma\cdot\left(\tau_{m,n}(a_0^m\otimes a_0^n),\tau_{m,n}(a_1^m\otimes1),\ldots,\tau_{m,n}(a_p^m\otimes 1),\tau_{m,n}(1\otimes a_1^n),\ldots ,\tau_{m,n}(1\otimes a_q^n)\right)\right)\right)\notag\\
       &=\sum_{\sigma\in S_{p,q}}sgn(\sigma)\Big(\sigma\cdot\Big(\Phi_{m+n}(x_{(1)})(\tau_{m,n}(a^m_0\otimes a_0^n)),\Phi_{m+n}(x_{(2)})(\tau_{m,n}(a_1^m\otimes1)),\ldots,\Phi_{m+n}(x_{(p+1)})(\tau_{m,n}(a_p^m\otimes 1)),\notag\\
       &\qquad\Phi_{m+n}(x_{(p+2)})(\tau_{m,n}(1\otimes a_1^n)),\ldots ,\Phi_{m+n}(x_{(p+q+1)})(\tau_{m,n}(1\otimes a_q^n))\Big)\Big)\notag
       \end{align*}

      On the other hand, we have
       \begin{align*}
           &\tau'^{sh}_{m,n}\left(Hoch_m^\Phi(x_{(1)})(a_0^m,\ldots, a_p^m)\otimes Hoch_n^\Phi(x_{(2)})(a_0^n,\ldots, a_p^n)\right)\notag\\
           &=\tau'^{sh}_{m,n}\left(\left(C^{\Phi_{m}}_{p}(x_{(1)})\otimes C^{\Phi_{n}}_{q}(x_{(2)})\right)\left((a_0^m,\ldots, a_p^m)\otimes (a^n_0,\ldots, a_q^n)\right)\right)\notag\\
           &=\tau'^{sh}_{m,n}\left(\left(\Phi_{m}(x_{(1)})(a_0^m),\Phi_{m}(x_{(2)})(a^m_1),\ldots, \Phi_{m}(x_{(p+1)})(a^m_p)\right)\otimes \left(\Phi_{n}(x_{(p+2)})(a^n_0), \Phi_{n}(x_{(p+3)})(a_1^n),\ldots,\Phi_{n}(x_{(p+q+2)})(a^n_q)\right)\right)\notag\\
           &=C_{p+q}(\tau'_{m,n})\sum_{\sigma\in S_{p,q}}sgn(\sigma)\Big(\sigma\cdot \Big(\Phi_{m}(x_{(1)})(a_0^m)\otimes \Phi_{n}(x_{(p+2)})(a^n_0),\Phi_{m}(x_{(2)})(a_1^m)\otimes 1,\ldots, \Phi_{m}(x_{(p+1)})(a_p^m)\otimes 1,\notag\\
           &\qquad 1\otimes \Phi_{n}(x_{(p+3)})(a_1^n),\ldots,1\otimes\Phi_{n}(x_{(p+q+2)})(a^n_q)\Big)\Big)\\
           &=\sum_{\sigma\in S_{p,q}}sgn(\sigma)\Big(\sigma\cdot\Big(\Phi_{m+n}(x_{(1)})(\tau_{m,n}(a_0^m\otimes a_0^n)),\Phi_{m+n}(x_{(2)})(\tau_{m,n}(a_1^m\otimes1)),\ldots,\Phi_{m+n}(x_{(p+1)})(\tau_{m,n}(a_p^m\otimes 1)),\\
          &\qquad\Phi_{m+n}(x_{(p+2)})(\tau_{m,n}(1\otimes a_1^n)),\ldots ,\Phi_{m+n}(x_{(p+q+1)})(\tau_{m,n}(1\otimes a_q^n))\Big)\Big)\notag
       \end{align*} 
      where the last equality is obtained by applying the condition \eqref{diag4} in Definition \ref{D3.r1} as well as using the cocommutativity of $C$. Finally, by the definition of $\Phi_0$, we have $Hoch^\Phi_{0}(x)(1)=\Phi_0(x)(1)=\epsilon(x)1$. This proves the result.
      \end{proof}

   \section{Sweedler product and coalgebra measuring of multiplicative sequences of algebras}\label{section3}

    If $A$ is a $k$-algebra equipped with product $\mu:A\otimes A\longrightarrow A$ and $(C,\Delta,\epsilon)$ is a $k$-coalgebra, we know that the product on the convolution algebra $Hom_k(C,A)$ is given by
     \begin{equation}\label{con3.2}
        {Hom}_k(C,A)\otimes {Hom}_k(C,A)\longrightarrow  {Hom}_k(C,A)\qquad f\otimes g\longmapsto \mu\circ(f\otimes g)\circ\Delta.
     \end{equation} The unit element in $Hom_k(C,A)$ is given by the composition $C\xrightarrow{\epsilon}k\xrightarrow{u}A$, where $u:k\longrightarrow A$ denotes the unit map
     on $A$. 
     
     \smallskip
     We know that $\Phi:C\longrightarrow Hom_k(A,A')$ is a coalgebra measuring if and only if the corresponding map $A\longrightarrow Hom_k(C,A')$ is a morphism of algebras. We first extend this to multiplicative sequences of algebras.

    \begin{lem}\label{lem3.3}
        Let $(A_\bullet,\tau) \in \mathscr{M}ult_k$ and let $(C,\Delta,\epsilon)$ be a cocommutative $k$-coalgebra. Then, the collection $[C, A]_\bullet$ given by setting
        \begin{equation}
        [C,A]_0:=k \qquad [C,A]_i:=Hom_k(C,A_i)\qquad i>0
        \end{equation} forms a multiplicative sequence of algebras, equipped with $[\Delta,\tau]=\{[\Delta,\tau]_{m,n}\}_{m,n\geq 0}$, where
        \begin{equation}\label{3.8yg}
            [\Delta,\tau]_{m,n}: [C,A]_m\otimes  [C,A]_n\longrightarrow [C,A]_{m+n} \qquad f\otimes g\longmapsto \tau_{m,n}\circ (f\otimes g)\circ \Delta
        \end{equation}
     for $m$, $n\geq 1$ and $[\Delta,\tau]_{0,l}=id=[\Delta,\tau]_{l,0}$ for all $l\geq 0$. 
    \end{lem}
    \begin{proof}
       We first show that $[\Delta,\tau]_{m,n}$ is an algebra morphism for all $m,n\geq 0$. From the definition, this is clear when either $m=0$ or $n=0$.  For $m,n\geq 1$, we take $(f\otimes g),(f'\otimes g')\in [C,A]_m\otimes [C,A]_n=  {Hom}_k(C,A_m)\otimes  {Hom}_k(C,A_n) $. For any $x\in C$, we now have
        \begin{align*}
            [\Delta,\tau]_{m,n}((f\otimes g)\cdot (f'\otimes g'))(x)
            &=[\Delta,\tau]_{m,n}((f\cdot f')\otimes (g\cdot g'))(x)\notag\\
            &=\tau_{m,n}\Big( (f\cdot f')(x_{(1)})\otimes (g\cdot g')(x_{(2)})\Big)\tag{by \eqref{3.8yg}}\\
            &=\tau_{m,n}\Big(f(x_{(1)})f'(x_{(2)})\otimes g(x_{(3)})g'(x_{(4)})\Big)\tag{by \eqref{con3.2}}\\
            &=\tau_{m,n}\Big((f(x_{(1)})\otimes g(x_{(3)}))\cdot (f'(x_{(2)})\otimes g'(x_{(4)}))\Big)\notag\\
            &=\tau_{m,n}\Big((f(x_{(1)})\otimes g(x_{(3)}))\Big)\cdot  \tau_{m,n}\Big((f'(x_{(2)})\otimes g'(x_{(4)})\Big)\tag{as $\tau_{m,n}$ is an algebra morphism}\\
            &=\tau_{m,n}\Big((f(x_{(1)})\otimes g(x_{(2)}))\Big)\cdot  \tau_{m,n}\Big((f'(x_{(3)})\otimes g'(x_{(4)})\Big) \tag{as $C$ is cocommutative}\\
            &=\left([\Delta,\tau]_{m,n}(f\otimes g)(x_{(1)})\right)\cdot  \left([\Delta,\tau]_{m,n}(f'\otimes g')(x_{(2)})\right)\\
            &=\Big([\Delta,\tau]_{m,n}(f\otimes g)\cdot [\Delta,\tau]_{m,n}(f'\otimes g')\Big)(x).
        \end{align*}
        
          It may also be verified that  $[\Delta,\tau]_{m,n}: [C,A]_m\otimes  [C,A]_n\longrightarrow [C,A]_{m+n} $ preserves the unit. This shows that 
        each $[\Delta,\tau]_{m,n}$ is a morphism of algebras.  It remains to show that   the collection 
        $[\Delta,\tau]=\left\{[\Delta,\tau]_{m,n}\right\}_{m,n\geq 0}$ satisfies the associativity condition \eqref{diag1} in Definition \ref{def2.1}. For any $l,m,n\geq 1$ and $f\in [C,A]_l$, 
        $g\in [C,A]_m$, $h\in [C,A]_n$, we  see that
        \begin{align*}
           \left( [\Delta,\tau]_{l+m,n}\circ\left([\Delta,\tau]_{l,m}\otimes  [C, A]_n\right)\right) (f\otimes g\otimes h)
            &=[\Delta,\tau]_{l+m,n}\Big(\big(\tau_{l,m}\circ(f\otimes g)\circ \Delta \big)\otimes h)\Big)\\
            &=\tau_{l+m,n} \Big((\tau_{l,m}\circ (f\otimes g)\circ \Delta)\otimes h\Big)\circ \Delta\\
            &= \tau_{l+m,n}(\tau_{l,m}\otimes A_n)(f\otimes g\otimes h)\circ (\Delta\otimes C)\circ \Delta \tag{by condition \eqref{diag1} on $A_\bullet$}\\
            &=\tau_{l,m+n}(A_l\otimes\tau_{m,n})(f\otimes g\otimes h)\circ(C\otimes \Delta)\circ \Delta\tag{by coassociativity of $\Delta$}\\
            &=\left([\Delta,\tau]_{l,m+n}\circ\Big(  [C, A]_l\otimes[\Delta,\tau]_{m,n}\Big)\right)(f\otimes g\otimes h).
        \end{align*}
        This proves the result.
    \end{proof}
    \smallskip

     \begin{lem}
       Let $C$ be a cocommutative coalgebra. Then, the  assignment
        $
             A_\bullet\longmapsto[C,A]_\bullet$ determines a functor $[C,-]:\mathscr Mult_k\longrightarrow \mathscr Mult_k$.
     \end{lem}
     \begin{proof} Let $f_\bullet:(A_\bullet,\tau)\longrightarrow (A'_\bullet,\tau')$ be a morphism of multiplicative sequences in the sense of Definition \ref{def2.1}. It is clear that the induced maps \begin{equation}\label{eq3.9.} [C,f]_i:[C,A]_i=Hom_k(C,A_i)\longrightarrow Hom_k(C,A'_i)=[C,A']_i \qquad i>0
     \end{equation} together with $[C,f]_0:=id:[C,A]_0=k\longrightarrow k=[C,A']_0$ are morphisms of algebras. To show that $[C,f]:=\{[C,f]_i\}_{\geq 0} $ is a morphism in $\mathscr Mult_k$, we must show that it satisfies the compatibility condition in \eqref{diag2}. In other words, for all $i$, $j\geq 0$, we must verify that
        \begin{equation}
            [C,f]_{i+j}\circ [\Delta,\tau]_{i,j}=[\Delta,\tau']_{i,j}\circ \left([C,f]_i\otimes [C,f]_j\right).
        \end{equation}
     For this, we take  $s\in[C,A]_i$ and $t\in[C,A]_j$. Then, for all $i,j\geq 1$, we have 
       \begin{align*}
           \left([C,f]_{i+j}\circ [\Delta,\tau]_{i,j}\right)(s\otimes t)
           &=[C,f]_{i+j}( \tau_{i,j}\circ(s\otimes t)\circ \Delta)\tag{by \eqref{3.8yg}}\\
           &=f_{i+j}\circ \tau_{i,j}\circ (s\otimes t)\circ\Delta\tag{by \eqref{eq3.9.}}\\
           &=\tau'_{i,j}\circ(f_i\otimes f_j)\circ (s\otimes t)\circ\Delta\tag{as $f_\bullet$ is a morphism in $\mathscr Mult_k$}\\
           &=\tau'_{i,j}\circ\left((f_i\circ s)\otimes(f_j\circ t)\right)\circ \Delta\\
           &=\tau'_{i,j}\circ\left([C,f]_i(s)\otimes[C,f]_j(t)\right)\circ \Delta\\
           &=[\Delta,\tau']_{i,j}\left([C,f]_i(s)\otimes[C,f]_j(t)\right)
           =\left([\Delta,\tau']_{i,j}\circ \left([C,f]_i\otimes [C,f]_j\right)\right)(s\otimes t)
       \end{align*}
       This completes the proof.
     \end{proof}
   \begin{thm}\label{prop3.3}
        Let $(A_\bullet,\tau)$, $(A'_\bullet,\tau')\in \mathscr Mult_k$. Let $(C,\Delta,\epsilon)$ be a cocommutative $k$-coalgebra. Suppose that we have a collection $\Phi=\left\{\Phi_i:C\longrightarrow  {Hom}_k(A_i,A_i')\right\}_{i\geq 0}$   of linear maps, with $\Phi_0(x)(1)=\epsilon(x)1$ for all $x\in C$. We set
        \begin{equation}\label{eq3.8.1}
    f_i:A_i\longrightarrow [C,A']_i=Hom_k(C,A_i)\qquad a^i\longmapsto f_i(a^i):=(x\mapsto \Phi_i(x)(a^i))\qquad \forall i> 0,
    \end{equation}
    as well as  $f_0=id:A_0=k\longrightarrow   k=[C,A']_0$. Then $(C,\Phi)$ is a coalgebra measuring from $(A_\bullet,\tau)$ to $(A'_\bullet,\tau')$ if and only if $f:=\left\{f_i\right\}_{i\geq0}: A_\bullet\longrightarrow  [C, A']_\bullet$ is a morphism in $ \mathscr Mult_k$. 
    \end{thm}
    \begin{proof}
   For $i>0$, we know that $f_i:A_i\longrightarrow [C,A']_i=Hom_k(C,A'_i)$ is a morphism of algebras if and only if $(C,\Phi_i)$ is a coalgebra measuring from $A_i$ to $A_i'$. Therefore, to complete the proof, it suffices to show that the diagram on the left hand side below commutes if and only if the diagram on the right hand side commutes for each $x\in C$.
        \begin{equation}\label{3.12ert}
            \begin{tikzcd}[row sep=3em,column sep=5em]
   A_m \otimes A_n\arrow[r, "f_m \otimes f_{n}"]\arrow[d, "\tau_{m,n}"']
   & {[C,A']_{m}} \otimes  {[C,A']_n}
  \arrow[d,  "{[\Delta,\tau']_{m,n}}",right]\\
  A_{m+n}\arrow[r, "f_{m+n}"']
  & {[C,A']}_{m+n}
         \end{tikzcd}
        \qquad \textit{$\Leftrightarrow$}\qquad
    \begin{tikzcd}[ row sep=3em,column sep=7em]
A_m \otimes A_n
  \arrow[r, "\Phi_m(x_{(1)}) \otimes \Phi_{n}(x_{(2)})"]
  \arrow[d, "\tau_{m,n}"']
&
A'_{m} \otimes A'_n
  \arrow[d, "\tau'_{m,n}"]
\\
A_{m+n}
  \arrow[r, "\Phi_{m+n}(x)"']
&
A'_{m+n}.
\end{tikzcd}
\end{equation}
For any $a^m\in A_m,a^n\in A_n,$ and $x\in C$, we have
\begin{align*}
    [\Delta,\tau']_{m,n} (f_m(a^m)\otimes f_n(a^n))(x) &= \tau'_{m,n}(f_m(a^m)(x_{(1)}) \otimes f_n(a^n)(x_{(2)}))\tag{by (\ref{3.8yg})}\\
    &=\tau'_{m,n}  \left(\Phi_m(x_{(1)})(a^m)\otimes\Phi_n(x_{(2)})(a^n)\right).\tag{by \eqref{eq3.8.1}}
\end{align*} By \eqref{eq3.8.1}, we also have
\begin{equation}
f_{m+n}(\tau_{m,n}(a^m\otimes a^n))(x)=\Phi_{m+n}(x)(\tau_{m,n}(a^m\otimes a^n))
\end{equation} From \eqref{3.12ert}, the result is now clear. 
 
    \end{proof}

\smallskip

Given a $k$-coalgebra $C$, it was shown by Anel and Joyal (see \cite[Proposition 3.4.3]{AJ}) that the functor that associates a $k$-algebra $A$ to the algebra 
$Hom_k(C,A)$ has a left adjoint, known as the Sweedler product. Our next aim in this section is to find a full subcategory of $\mathscr Mult_k$ such that the restriction of the functor $[C,-]:\mathscr Mult_k\longrightarrow \mathscr Mult_k$ has a left adjoint.

\smallskip

Accordingly, we define a category $\mathcal Comm_k^{\geq 0}$, where an object $(B_\bullet,\xi)\in \mathcal Comm_k^{\geq 0}$ consists of 
\begin{enumerate}[(a)]
    \item a collection $B_\bullet=\{B_i\}_{i\geq 0}$ of commutative $k$-algebras, with $B_0=k$, and
    \item a collection $\xi=\{\xi_{i,j}:B_i\longrightarrow B_j\}_{0\leq i\leq j}$ of $k$-algebra homomorphisms such that the following conditions hold:
    \begin{enumerate} 
        \item[(i)] for each $i\geq 0$, we have $\xi_{i,i}=id:B_i\longrightarrow B_i$, and
        \item[(ii)] for  $0\leq i\leq j\leq k$ we have $\xi_{j,k}\circ \xi_{i,j}=\xi_{i,k}$.
    \end{enumerate}
\end{enumerate}

A morphism $f_\bullet:(B_\bullet,\xi)\longrightarrow (B_\bullet',\xi')$ in $\mathcal Comm_k^{\geq 0}$  consists of a collection $\left\{f_i:B_i\longrightarrow B_i'\right\}_{i\geq 0}$ of $k$-algebra homomorphisms  such that $\xi'_{i,j}\circ f_i=f_j\circ \xi_{i,j}$ for all $0\leq i\leq j$. An object $(B_\bullet,\xi)\in \mathcal Comm_k^{\geq 0}$ determines a multiplicative sequence
$B_\bullet=\{B_i\}_{i\geq 0}$  of algebras given by
\begin{equation}\label{pi3r}
    \tau_{i,j}:B_i\otimes B_j\longrightarrow B_{i+j}\qquad (b^i\otimes b^j)\longmapsto \xi_{i,i+j}(b^i)\xi_{j,i+j}(b^j)
\end{equation} As such, $\mathcal Comm_k^{\geq 0}$ may be treated as a full subcategory of $\mathscr Mult_k$.

 \smallskip  
Now let $(C,\Delta,\epsilon)$ be a cocommutative $k$-coalgebra and let 
$A_\bullet\in \mathscr Mult_k$.  We set $(C\Box A)_0:=k$, and for each $n\geq1$, we let  $(C\Box A)_n$ be the commutative $k$-algebra  generated by the symbols 
\begin{equation} 
\{\mbox{$x\Box a^i$ $\vert$  $x\in C$, $a^i\in A_i$, $1\leq i\leq n$}\}
\end{equation} subject to the following relations
\begin{enumerate}[(a)]
        \item for each $1\leq i\leq n$,  the canonical map $C\times A_i\longrightarrow (C\Box A)_n$ defined by $(x,a^i)\mapsto x\Box a^i$  is $k$-bilinear,
        \item for each $x\in C$, the following identities hold 
        \begin{enumerate}[(1)] 
            \item for $1\leq i\leq n$ and $a_1^i$, $a_2^i\in A_i$, we have $x\Box a^i_1a^i_2=(x_{(1)}\Box a^i_1)(x_{(2)}\Box a^i_2)$, and
            \item for $a^i\in A_i$, $a^j\in A_j$ and $i+j\leq n$, we have $x\Box (\tau_{i,j}(a^i\otimes a^j))=(x_{(1)}\Box a^i)(x_{(2)}\Box a^j)$
        \end{enumerate}
        \item  for all $1\leq i\leq n$ and $x\in C$, we have $x\Box 1_{A_i}=\epsilon(x)$.
    \end{enumerate}

    \smallskip
    \begin{lem}\label{lem3.4}
        Let $(C,\Delta,\epsilon)$ be a cocommutative $k$-coalgebra and $(A_\bullet,\tau) \in \mathscr Mult_k$.

        \medskip
        \noindent
        (a) The collection $(C\Box A)_\bullet=\left\{(C\Box A)_n\right\}_{n\geq 0}$ along with maps $\xi=\{\xi_{i,j}\}_{0\leq i\leq j}$ determined by  
          \begin{equation}\label{316dp}
              \xi_{i,j}:(C\Box A)_i\longrightarrow (C\Box A)_j\qquad x\Box a^t\longmapsto x\Box a^t
          \end{equation}
        for   $x\in C,a^t\in A_t$ with $1\leq t\leq i$, gives an object of $ \mathcal Comm^{\geq 0}_k$. Further, the assignment $A_\bullet\longmapsto (C\Box A)_\bullet$ determines a functor $C\Box-:\mathscr Mult_k\longrightarrow \mathcal Comm^{\geq 0}_k$.
        
        \medskip
        \noindent
        (b) There is a  coalgebra measuring $(C,\Upsilon=\{\Upsilon_n\}_{n\geq 0})$ from $A_\bullet$ to $(C\Box A)_\bullet$, where $\Upsilon_n$ is determined by 
        \begin{equation}
        \begin{array}{c}
        \Upsilon_0:C\otimes A_0=C\otimes k\longrightarrow k=(C\Box A)_0\qquad x\otimes s\mapsto \epsilon(x)s \\
          \Upsilon_n: C\otimes A_n\longrightarrow (C\Box A)_n\qquad x\otimes a^n\mapsto x\Box a^n \qquad\forall n\geq 1
          \end{array}
        \end{equation}   where $(C\Box A)_\bullet\in \mathcal Comm_k^{\geq 0}$ is treated as an object of $\mathscr Mult_k$.
    \end{lem}
\begin{proof}
        (a) By construction, for $i\leq j$, we note that the relations between the symbols defining $(C\Box A)_i$ continue to hold in $(C\Box A)_j$.   It follows that the assignment in \eqref{316dp} determines a $k$-algebra homomorphism $\xi_{i,j}:(C\Box A)_i\longrightarrow (C\Box A)_j$. The result is now clear.

       \medskip
      (b) We note that  $(C\Box A)_\bullet\in \mathcal Comm_k^{\geq 0}$ is treated as an object of $\mathscr Mult_k$ equipped with  structure maps $(\Delta\Box\tau)=\{(\Delta\Box\tau)_{m,n}\}_{m,n\geq 0}$, where  
        \begin{equation}\label{eq3.18.}
            (\Delta\Box\tau)_{m,n}:(C\Box A)_m\otimes (C\Box A)_n\longrightarrow (C\Box A)_{m+n}\qquad
            (x\Box a^i)\otimes(y\Box a^j)\longmapsto (x\Box a^i)(y\Box a^j)  
        \end{equation}
         for $m,n\geq 1$, $1\leq i\leq m,1\leq j\leq n$. First, we verify that $(C,\Upsilon_n)$ is a measuring from $A_n$ to $(C\Box A)_n$ for each $n\geq 0$. For $n=0$, this is clear. We consider $n\geq1$ and  choose $a^n_1,a^n_2\in A_n$, $x\in C$. By the relations defining $(C\Box A)_n$, we see that 
        $$ \Upsilon_n(x\otimes a^n_1a^n_2)
            =(x\Box a^n_1a^n_2)
            =(x_{(1)}\Box a^n_1)(x_{(2)}\Box a^n_2)
            =\Upsilon_n(x_{(1)}\otimes a^n_1)\Upsilon_n(x_{(2)}\otimes a^n_2)\qquad \qquad\Upsilon_n(x\otimes 1)=x\Box 1=\epsilon(x)1.$$
            
       It remains to show that $\Upsilon$ satisfies the compatibility condition in \eqref{diag4}. The result is clear for the case $m=0$ or $n=0$. For $a^m\in A_m,a^n\in A_n$ with $m,n\geq 1$ and $x\in C$, we have
        \begin{align*}
          \Upsilon_{m+n}\left(x\otimes (\tau_{m,n}(a^m\otimes a^n))\right)
          =x\Box (\tau_{m,n}(a^m\otimes a^n))
          =(x_{(1)}\Box a^m)(x_{(2)}\Box a^n)
          &=(\Delta\Box\tau)_{m,n}((x_{(1)}\Box a^m)\otimes(x_{(2)}\Box a^n))\\
          &=(\Delta\Box\tau)_{m,n}\Big(\Upsilon_m(x_{(1)}\otimes a^m)\otimes\Upsilon_n(x_{(2)}\otimes a^n)\Big).
        \end{align*}
        This completes the proof.
    \end{proof}

    \begin{thm}\label{prop3.6}
        Let $(C,\Delta,\epsilon)$ be a cocommutative $k$-coalgebra and $A_\bullet\in \mathscr Mult_k$. Then, the coalgebra measuring $(C,\Upsilon=\left\{\Upsilon_n: C\otimes A_n\longrightarrow (C\Box A)_n\right\}_{n\geq 0})$  has the following universal property: for any $B_\bullet\in \mathcal Comm_k^{\geq 0}$ and any coalgebra measuring $(C,\Phi=\left\{\Phi_n: C\otimes A_n\longrightarrow B_n\right\}_{n\geq 0})$ from $A_\bullet$ to $B_\bullet$, there is a unique morphism $g_\bullet=\left\{g_n:(C\Box A)_n \longrightarrow B_n\right\}_{n\geq 0}$ in $\mathcal Comm_k^{\geq 0}$ such that the following diagram 
      \begin{equation}\label{dig319c}
      \begin{array}{c}
      \begin{tikzpicture}[>=stealth]
         \node (A) at (0,0) {$C\otimes A_n$};
         \node (B) at (1.3,1.3) {$B_n$};
         \node (C) at (2.6,0) {$(C \Box A)_n$};
         \draw[->] (A) -- node[above left] {$\Phi_n$} (B);
         \draw[->] (A) -- node[below left] {$\qquad\Upsilon_n$} (C);
         \draw[->, dashed] (C) -- node[right] {$g_n$} (B);
      \end{tikzpicture}
      \end{array}
     \end{equation}
     commutes for each $n\geq 0$. 
    \end{thm}

    \begin{proof} 
   We let  $\tau^A=\left\{\tau^A_{m,n}: A_m\otimes A_n\longrightarrow A_{m+n}\right\}_{m,n\geq 0}$ be the structure maps of 
the object $A_\bullet\in  \mathscr Mult_k$. Similarly, we   let $\tau^B=\left\{\tau^B_{m,n}: B_m\otimes B_n\longrightarrow B_{m+n}\right\}_{m,n\geq 0}$ be the structure maps of  $B_\bullet\in  \mathcal Comm_k^{\geq 0}$ when regarded as an object of $\mathscr Mult_k$ as explained in \eqref{pi3r}. Now suppose that we have a coalgebra measuring $\left(C,\Phi=\left\{\Phi_n:C\otimes A_n\longrightarrow B_n\right\}_{n\geq 0}\right)$ from $A_\bullet$ to $B_\bullet$.
    For $n=0$, we set $g_0=id:(C\Box A)_0=k\longrightarrow k=B_0$, and for $n\geq 1$, we set
    \begin{equation}\label{eq3.17yu}
        \begin{array}{c}
             g_n:(C\Box A)_n\longrightarrow B_n \\
              x\Box a^t\longmapsto \tau^B_{t,n-t}(\Phi_t(x\otimes a^t)\otimes 1)=\tau^B_{n-t,t}(1\otimes \Phi_t(x\otimes a^t))
        \end{array}
    \end{equation}
    where $x\in C$ and $a^t\in A_t$ for $1\leq t\leq n$, where  the equality in \eqref{eq3.17yu} follows from the commutativity of $B_n$. Further since $\Phi$ is a measuring, we see that
    \begin{equation}\label{eq3.20p}
          \tau^B_{t,n-t}(\Phi_t(x\otimes a^t)\otimes1)=\tau^B_{t,n-t}(\Phi_{t}(x_{(1)}\otimes a^t)\otimes \Phi_{n-t}(x_{(2)}\otimes 1))=\Phi_n(x\otimes\tau^A_{t,n-t}( a^t\otimes 1)) .
    \end{equation}
    Similarly, we can also check that $ \tau^B_{n-t,t}(1\otimes \Phi_t(x\otimes a^t))=\Phi_n(x\otimes\tau^A_{n-t,t}(1\otimes a^t))$. We will now show that $g=\{g_n\}_{n\geq 0}$ is a morphism in $\mathcal Comm_k^{\geq 0}$ from $(C\Box A)_\bullet$ to $B_\bullet$. 

\smallskip 
We first verify that $g_n$ is well-defined. This is clear for $n=0$. For $n\geq 1$, we take $a_1^i,a_2^i\in A_i$, where $1\leq i\leq n$. Using the fact that $\Phi_n$ is a measuring, we have 
            \begin{align*}
             g_n(x\Box a_1^ia_2^i)=\tau^B_{i,n-i}\left(\Phi_i(x\otimes a_1^ia_2^i)\otimes 1\right)
            &=\Phi_n\left(x\otimes \tau^A_{i,n-i}( a_1^ia_2^i\otimes 1)\right)\tag{by \eqref{eq3.20p}}\\
           & =\Phi_n\left(x\otimes \left(\tau^A_{i,n-i}( a_1^i\otimes 1)\right)\left(\tau^A_{i,n-i}( a_2^i\otimes 1)\right)\right)\\
            &=\Phi_n\left(x_{(1)}\otimes \tau^A_{i,n-i}( a_1^i\otimes 1)\right)\Phi_n\left(x_{(2)}\otimes \tau^A_{i,n-i}( a_2^i\otimes 1)\right)\tag{as $\Phi_n$ is a measuring}\\
            &=g_n(x_{(1)}\Box a^i_1)g_n(x_{(2)}\Box a^i_2)
            \end{align*}
        For  $a^i\in A_i$ and $a^j\in A_j$ where $ i,j\geq 1 $ and $i+j=t \leq n$, we have

\begin{align*}
    g_n\left(x\Box \tau^A_{i,j}(a^i\otimes a^j)\right)
            &=\tau_{t,n-t}^B\left(\Phi_{i+j}\left(x\otimes \tau_{i,j}^A(a^i\otimes a^j)\right)\otimes 1\right)\tag{putting $i+j=t$}\\
            &=\tau_{t,n-t}^B\left(\tau_{i,j}^B\left(\Phi_{i}(x_{(1)}\otimes a^i) \otimes \Phi_j(x_{(2)}\otimes a^j)\right)\otimes 1\right)\tag{as $\Phi$ is a measuring}\\
            &=\tau_{t,n-t}^B\left(\tau_{i,j}^B\left(\Phi_{i}(x_{(1)}\otimes a^i) \otimes 1\right) \tau_{i,j}^B\left(1 \otimes \Phi_j(x_{(2)}\otimes a^j)\right)\otimes 1\right)\tag{as $\tau^B_{i,j}$ is an algebra morphism}\\
            &=\tau_{t,n-t}^B\left(\tau_{i,j}^B\left(\Phi_{i}(x_{(1)}\otimes a^i) \otimes 1\right)\otimes 1\right)\tau_{t,n-t}^B\left( \tau_{i,j}^B\left(1 \otimes \Phi_j(x_{(2)}\otimes a^j)\right)\otimes 1\right)\tag{as $\tau^B_{t,n-t}$ is an algebra morphism}\\
            &=\tau_{i,j+n-t}^B\left(\Phi_{i}(x_{(1)}\otimes a^i) \otimes \tau_{j,n-t}^B (1\otimes 1)\right)\tau_{j+i,n-t}^B\left( \tau_{j,i}^B( \Phi_j(x_{(2)}\otimes a^j)\otimes 1)\otimes 1\right)\tag{by \eqref{diag1} and commutativity of $B_{i+j}$}\\
            &=g_n(x_{(1)}\Box a^i)\tau_{j,i+n-t}^B\left( \Phi_j(x_{(2)}\otimes a^j)\otimes\tau_{i,n-t}^B( 1\otimes 1)\right)\tag{by \eqref{diag1}}\\
            &=g_n(x_{(1)}\Box a^i)g_n(x_{(2)}\Box a^j)
\end{align*} From \eqref{eq3.17yu}, we also see  that $g_n(x\Box 1_{A_i})=\epsilon(x)$ for $i\leq n$. This shows that each $g_n:(C\Box A)_n\longrightarrow B_n$ is a   morphism of algebras. From \eqref{eq3.17yu}, it is also clear that the maps $g=\{g_n\}_{n\geq 0}$ are well behaved with respect to the structure maps 
$(C\Box A)_m\longrightarrow (C\Box A)_n$ and $B_m\longrightarrow B_n$ for $m\leq n$ defining $(C\Box A)_\bullet$ and $B_\bullet$ respectively as objects of $\mathcal Comm_k^{\geq 0}$. It follows that $g=\{g_n\}_{n\geq 0}:(C\Box A)_\bullet\longrightarrow B_\bullet$ is a morphism in $\mathcal Comm_k^{\geq 0}$. Using \eqref{eq3.17yu} and \eqref{eq3.20p}, we have (for $x\in C$ and $a^n\in A_n$)
\begin{equation}
g_n\Upsilon_n(x\otimes a^n)=g_n(x\Box a^n)=\Phi_n(x\otimes \tau^A_{n,0}(a^n\otimes 1))=\Phi_n(x\otimes a^n)
\end{equation} It follows that the diagram \eqref{dig319c} commutes. The uniqueness of $g=\{g_n\}_{n\geq 0}$ is clear. 
    \end{proof}

    We are now ready to construct a left adjoint to the functor $[C,-]:\mathscr{M}ult_k\longrightarrow \mathscr Mult_k$ restricted to the subcategory  $\mathcal Comm_k^{\geq 0}$, which we continute to denote by $[C,-] : \mathcal Comm_k^{\geq 0} \longrightarrow \mathscr{M}ult_k$.

    \smallskip
    \begin{Thm}\label{Thm3.9}
        Let $C$ be a cocommutative $k$-coalgebra. For  $A_\bullet\in \mathscr Mult_k$ and  $B_\bullet\in \mathcal Comm^{\geq 0}_k$, we have a natural isomorphism 
        \begin{equation}\label{3.23nk}
            \mathcal Comm_k^{\geq 0}\left((C\Box A)_\bullet,B_\bullet\right)\cong {\mathscr Mult_k}\left(A_\bullet, [C,B]_\bullet\right).
     \end{equation} In other words, $C\Box-:\mathscr Mult_k\longrightarrow \mathcal Comm_k^{\geq 0}$ is  left adjoint to the functor $ [C,-] : \mathcal Comm_k^{\geq 0} \longrightarrow \mathscr{M}ult_k$.
    \end{Thm}

    \begin{proof} 
    
     From Proposition \ref{prop3.3}, we have a one one correspondence between coalgebra measurings $(C,\Phi)$ from $A_\bullet$ to $B_\bullet$ and morphisms from $A_\bullet$ to $[C,B]_\bullet$ in $\mathscr Mult_k$. On the other hand, Proposition \ref{prop3.6} gives a one one  correspondence between coalgebra measurings $(C,\Phi)$ from $A_\bullet$ to $B_\bullet$ and morphisms from $(C\Box A)_\bullet$ to $B_\bullet$ in $\mathcal Comm_k^{\geq 0}$.  Together, this proves the isomorphism in \eqref{3.23nk}. 
     
     \smallskip
      More explicitly, for any $f_\bullet=\left\{f_n:(C\Box A)_n\longrightarrow B_n\right\}_{n\geq 0}\in \mathcal Comm_k^{\geq 0}\left((C\Box A)_\bullet,B_\bullet\right)$ the corresponding morphism $f'_\bullet=\{f'_n\}_{n\geq 0}\in {\mathscr Mult_k}\left(A_\bullet, [C,B]_\bullet\right)$ is given by
     \begin{equation*}
     \begin{array}{c}
        f'_0=id:A_0=k\longrightarrow k=[C,B]_0\\
           f'_n:A_n\longrightarrow  [C,B]_n\qquad a^n\mapsto f'_n(a^n):=(x\longmapsto f_n(x\Box {a^n}))\qquad n\geq 1
     \end{array}     
     \end{equation*}
     for $a^n\in A_n$ and $x\in C$. 
    Conversely, if $g_\bullet=\{g_n\}_{n\geq 0}\in {\mathscr Mult_k}(A_\bullet, [C,B]_\bullet)$,  the corresponding morphism $g'_\bullet=\{g'_n\}_{n\geq 0}\in \mathcal Comm_k^{\geq 0}\left((C\Box A)_\bullet,B_\bullet\right)$ is  given by
     \begin{equation*}
         \begin{array}{c}
         g'_0=id:(C\Box A)_0=k\longrightarrow k=B_0 \\
            g'_n:(C\Box A)_n\longrightarrow B_n\qquad x\Box a^t\mapsto \tau^B_{t,n-t}(g_t(a^t)(x)\otimes 1)
         \end{array}
     \end{equation*}
     for $a^t\in A_t,x\in C$ with $1\leq t\leq n$, where $\left\{\tau^B_{m,n}:B_m\otimes B_n\longrightarrow B_{m+n}\right\}_{m,n\geq 0}$ are the structure maps on $B_\bullet$ when treated as an object of $\mathscr Mult_k$.  
    \end{proof}

	\section{Universal measuring coalgebras and  multiplicative sequences}
    We let $Co\mathcal Alg_k$ denote the category of $k$-coalgebras. It is well known that the forgetful functor $Co\mathcal Alg_k\longrightarrow \mathcal Vect_k$ has a right adjoint $\mathfrak C: \mathcal Vect_k\longrightarrow Co\mathcal Alg_k$. In other words, for $C\in  Co\mathcal Alg_k$ and $V\in \mathcal Vect_k$, we have an isomorphism
    \begin{equation}\label{eq4.1p}
             \mathcal Vect_k(C, V)\cong Co\mathcal Alg_k(C,\mathfrak C(V)).
    \end{equation}
    
    We will now show that for  multiplicative sequences  $A_\bullet$, $A_\bullet' \in \mathscr{M}ult_k$, there is a unique coalgebra measuring which is universal among all coalgebra measurings from $A_\bullet$ to $A_\bullet'$.

    \begin{thm}\label{Prop4.1}
        Let $A_\bullet,A_\bullet' \in \mathscr{M}ult_k$. Then, there exists a cocommutative $k$-coalgebra $\mathcal M_c(A_\bullet, A_\bullet')$ and a measuring 
        \begin{equation}(\mathcal M_c(A_\bullet, A_\bullet'),\Phi)=(\mathcal M_c(A_\bullet, A_\bullet'),\Phi=\left\{\Phi_n:\mathcal M_c(A_\bullet, A_\bullet')\longrightarrow Hom_k(A_n,A_n')\right\}_{n\geq 0})
        \end{equation} satisfying the following universal property: for any coalgebra measuring $\left(C,\Theta=\left\{\Theta_n:C\longrightarrow Hom_k(A_n,A_n')\right\}_{n\geq 0}\right)$ from $A_\bullet$ to $A_\bullet'$ with $C$ cocommutative, there exists a unique coalgebra morphism $\zeta: C\longrightarrow \mathcal M_c(A_\bullet, A_\bullet')$ such that the following diagram 
         \begin{equation}\label{42digu}
         \begin{array}{c}
    \begin{tikzpicture}[>=stealth]
\node (A) at (0,0) {$C$};
\node (B) at (2.0,1.5) {$\mathcal M_c(A_\bullet,A_\bullet')$};
\node (C) at (4.6,0) {$ {Hom}_k(A_n,A_n')$};

\draw[->, dashed] (A) -- node[above left] {$\zeta$} (B);
\draw[->] (A) -- node[below left] {$\Theta_n$} (C);
\draw[->] (B) -- node[right] {$\Phi_n$} (C);
\end{tikzpicture}
\end{array}
\end{equation}
commutes for each $n\geq 0.$
    \end{thm}

    \begin{proof}
        We let $H$ denote the product of the family $\left\{ {Hom}_k(A_n,A_n')\right\}_{n\geq 0}$ of $k$-vector spaces along with canonical projections $\pi_n:H\longrightarrow  {Hom}_k(A_n,A_n')$ for $n\geq 0$.  The adjunction in \eqref{eq4.1p} gives us the   cofree coalgebra $\mathfrak C(H)$ over $H$ and the canonical morphism $\pi(H):\mathfrak C(H)\longrightarrow H$ in $\mathcal Vect_k$. We now set $\mathcal M_c(A_\bullet, A_\bullet'):=\sum D$, where the sum is taken over all the cocommutative subcoalgebras of $\mathfrak C(H)$ such that $\left\{(\pi_n\circ\pi(H))\big|_D: D\longrightarrow Hom_k(A_n, A_n')\right\}_{n\geq 0}$ is a measuring from $A_\bullet$ to $A_\bullet'$.  It is clear that  the sum $\mathcal M_c(A_\bullet, A_\bullet')$ is a cocommutative coalgebra and the collection $\Phi:=\left\{\Phi_n:=(\pi_n\circ\pi(H))\big|_{\mathcal M_c(A_\bullet, A_\bullet')}\right\}_{n\geq 0}$   gives a measuring from $A_\bullet$ to $A_\bullet'$.\\
        
       We now consider a measuring $\left(C,\Theta=\left\{\Theta_n: C\longrightarrow Hom_k(A_n,A_n')\right\}_{n\geq 0}\right)$ from $A_\bullet$ to $A_\bullet'$, with $C$ cocommutative. Applying the adjunction in \eqref{eq4.1p}, we note that  $\Theta=\{\Theta_n\}_{n\geq 0}:C\longrightarrow H=\underset{n\geq 0}{\prod}Hom_k(A_n,A'_n)$  corresponds to a coalgebra map $\zeta: C\longrightarrow \mathfrak C(H)$ such that $\pi(H)\circ \zeta=\Theta$. For each $n\geq 0$, we now have a commutative diagram
        \begin{equation}\label{44digu}
        \begin{array}{c}
\begin{tikzcd}[column sep=large, row sep=large]
& C \arrow[dl, dashed, "\zeta"'] 
    \arrow[d, "\Theta"] 
    \arrow[dr, "{\Theta_n}"] & \\
\mathfrak C(H) \arrow[r, "\pi(H)"] 
& H \arrow[r, "\pi_n"] 
& { {Hom}}_k(A_n,A_n')
\end{tikzcd}
\end{array}
\end{equation}
         Since $C$ is cocommutative, so is the subcoalgebra $\zeta(C)\subseteq \mathfrak C(H)$. By the definition of $\mathcal M_c(A_\bullet, A_\bullet')$, it follows that 
         $\zeta(C)\subseteq \mathcal M_c(A_\bullet, A_\bullet')\subseteq \mathfrak C(H)$. The commutativity of \eqref{42digu} is now clear from \eqref{44digu}. 
    \end{proof}

    We denote by $CoCo\mathcal Alg_k$  the category of cocommutative coalgebras over $k$. It is well known that $CoCo\mathcal Alg_k$ is a symmetric monoidal category. Our next aim is to show that multiplicative sequences of algebras can be enriched over the symmetric monoidal category $CoCo\mathcal Alg_k$.

    \begin{thm}\label{Prop4.2}
        Let $(A_\bullet,\tau),( A_\bullet',\tau')$, $(A_\bullet'',\tau'')\in \mathscr{M}ult_k$. Let $\left(C,\Phi=\left\{\Phi_n:C\longrightarrow Hom_k(A_n,A_n')\right\}_{n\geq 0}\right)$ be a coalgebra measuring from $A_\bullet$ to $A_\bullet'$ and $\left(C',\Phi'=\left\{\Phi'_n:C\longrightarrow Hom_k(A_n',A_n'')\right\}_{n\geq 0}\right)$ be a coalgebra measuring from  $A_\bullet'$ to $A_\bullet''$. Then, the pair $\left(C\otimes C', \Phi'\circ\Phi=\left\{ (\Phi'\circ\Phi)_n\right\}_{n\geq 0}\right)$ determines a measuring from $A_\bullet$ to $A_\bullet''$, where 
        \begin{equation}\label{eq4.3p}
        \begin{array}{c}
              (\Phi'\circ\Phi)_n:C\otimes C'\xrightarrow{\quad\Phi_n\otimes \Phi'_n\quad} {Hom}_k(A_n,A_n')\otimes  {Hom}_k(A_n',A_n'')\xrightarrow{\quad\circ\quad} {Hom}_k(A_n,A_n'')\qquad \forall n\geq 0
        \end{array}
        \end{equation}
    \end{thm}

    \begin{proof}
        It is clear that for each $n\geq 0$, the pair $(C\otimes C',(\Phi'\circ\Phi)_n)$ is a measuring from $A_n$ to $A_n''$. We must verify the compatibility condition  \eqref{diag4} appearing in 
        Definition \ref{D3.r1}. This is immediate if $m=0$ or $n=0$. We take $m,n\geq 1$. For   $a^m\in A_m,a^n\in A_n$, $ x\in C$ and $x'\in C'$, we now have
        \begin{align*}
            \left((\Phi'\circ\Phi)_{m+n}(x\otimes x')\right)\left(\tau_{m,n}(a^m\otimes a^n)\right)
            & =\Phi'_{m+n}(x')\left( \Phi_{m+n}(x)\left(\tau_{m,n}(a^m\otimes a^n)\right)\right) \tag{by   \eqref{eq4.3p} }\\
            &=\Phi'_{m+n}(x')\left(\tau_{m,n}'\left(\Phi_m(x_{(1)})(a^m)\otimes \Phi_n(x_{(2)})(a^n)\right)\right)\tag{as $\Phi$ is a measuring}\\
            &=\tau_{m,n}'' \left(\Phi'_m(x_{(1)}')\left(\Phi_m(x_{(1)})(a^m)\right)\otimes\Phi'_n(x_{(2)}')\left( \Phi_n(x_{(2)})(a^n)\right)\right)\tag{as $\Phi'$ is a measuring}\\
            &=\tau_{m,n}'' \left((\Phi'\circ\Phi)_{m}(x\otimes x')_{(1)}(a^m)\otimes (\Phi'\circ\Phi)_{n}(x\otimes x')_{(2)}(a^n)\right)\tag{by \eqref{eq4.3p}}.
        \end{align*}
        This completes the proof.
    \end{proof}

     \begin{Thm}\label{T4.3xq} The category $MULT_k$ given by setting      $Obj(MULT_k):=Obj(\mathscr Mult_k)$ and
\begin{equation}
MULT_k(A_\bullet,A_\bullet'):=\mathcal M_c(A_\bullet,A_\bullet') \qquad A_\bullet,A_\bullet'\in Obj(MULT_k)
\end{equation}
  is enriched over $CoCo\mathcal{A}lg_k$.
    \end{Thm}

    \begin{proof}
    Let $A_\bullet,A_\bullet'\in {M}ULT_k$. By definition, we know that the Hom object $\mathcal M_c(A_\bullet,A_\bullet')$ lies in $CoCo\mathcal Alg_k$.  Given $A_\bullet,A_\bullet',A_\bullet''\in MULT_k$,  it follows as in Proposition \ref{Prop4.2} that  the coalgebra $ \mathcal M_c(A_\bullet,A_\bullet')\otimes \mathcal M_c(A_\bullet',A_\bullet'')$ gives a measuring from $A_\bullet$ to $A''_\bullet$. Applying the universal property in Proposition \ref{Prop4.1}, we now have an induced morphism
        \begin{equation}\label{eq4.4p}
            \mathcal M_c(A_\bullet,A_\bullet')\otimes \mathcal M_c(A_\bullet',A_\bullet'')\longrightarrow  \mathcal M_c(A_\bullet,A_\bullet''),
        \end{equation} 
        which determines the composition of Hom objects in $CoCo\mathcal Alg_k$. We consider the unit object $k$ in the symmetric monoidal category $ CoCo\mathcal Alg_k$. For any  $A_\bullet\in MULT_k$, it is easy to see that the collection of linear maps
        \begin{equation}
            k\longrightarrow  {Hom}_k(A_n,A_n)\qquad x\mapsto (s\mapsto xs)\qquad n\geq 0
        \end{equation}
         forms a coalgebra measuring from $A_\bullet$ to itself. By the universal property in Proposition \ref{Prop4.1}, this induces a morphism $k\longrightarrow \mathcal M_c(A_\bullet, A_\bullet)$ of cocommutative coalgebras. Together with the composition of Hom objects in \eqref{eq4.4p}, it is now clear that $MULT_k$ is enriched over $CoCo\mathcal Alg_k$.
    \end{proof}

   We now recall that if $A$, $A'$ are $k$-algebras, then there is a  cocommutative coalgebra $Q_c(A, A')$ as well as a measuring 
$Q_c(A,A')\longrightarrow Hom_k(A,A')$ that is universal among all cocommutative measurings from $A$ to $A'$. This coalgebra $Q_c(A,A')$ is the cocommutative part of the universal measuring coalgebra of Sweedler  (see \cite[Theorem 7.0.4]{SW}) for the algebras $A$, $A'$. In Definition \ref{def3.5}, we considered coalgebra measurings between graded algebras in the category $\mathcal C_k$ of chain complexes. As in  \cite[Theorem 7.0.4]{SW}, it can be shown that for $T,T'\in \mathscr Gr\mathcal Alg(\mathcal C_k)$, there is a  cocommutative coalgebra $Q_c(T,T')$ that gives a measuring that is universal among cocommutative measurings from $T$ to $T'$.   

\smallskip 
We now recall from Definition \ref{DL2.2xr} that $Hoch(A_\bullet)\in \mathscr Gr\mathcal Alg(\mathcal C_k)$ for any multiplicative sequence of algebras $A_\bullet\in 
\mathscr Mult_k$. Accordingly, for any $A_\bullet, A'_\bullet\in 
\mathscr Mult_k$, we can consider  the cocommutative coalgebra $Q_c(Hoch(A_\bullet), Hoch(A_\bullet'))$. We now denote by  $\widetilde{MULT_k}$ the category whose objects are the same as those of $\mathscr Mult_k$ and whose Hom objects are given by 
 $\widetilde{MULT_k}(A_\bullet,A'_\bullet):=Q_c(Hoch(A_\bullet), Hoch(A_\bullet'))$ for any $A_\bullet,A'_\bullet\in \widetilde{MULT_k}$.  It is clear that   this provides another enrichment of multiplicative sequences of algebras over cocommutative coalgebras.  We will now show that there is a functor between the enriched categories $MULT_k$ and $\widetilde{MULT_k}$.

 \begin{lem}\label{lem4.4}
        Let $A_\bullet,A_\bullet'\in \mathscr Mult_k$. Then, there exists a canonical morphism of cocommutative coalgebras 
        \begin{equation}
            \gamma(A_\bullet,A_\bullet'):\mathcal M_c(A_\bullet,A_\bullet')\longrightarrow Q_c(Hoch(A_\bullet), Hoch(A_\bullet')).
        \end{equation}
    \end{lem}

    \begin{proof}
        Let  $A_\bullet,A_\bullet'\in \mathscr Mult_k$. We consider the measuring $(\mathcal M_c(A_\bullet, A_\bullet'),\Phi)$ from $A_\bullet$ to $A'_\bullet$ as in Proposition \ref{Prop4.1}. Applying  Theorem \ref{Thm3.6}, it follows that $(\mathcal M_c(A_\bullet,A_\bullet'),\Phi)$ induces a measuring $(\mathcal M_c(A_\bullet,A_\bullet'),Hoch^{\Phi})$ of graded algebras in $\mathcal C_k$  from $Hoch(A_\bullet)$ to $Hoch(A'_\bullet)$. By the universal property of $Q_c(Hoch(A_\bullet), Hoch(A_\bullet'))$, we now obtain a  morphism $\gamma(A_\bullet,A_\bullet'):\mathcal M_c(A_\bullet,A_\bullet')\longrightarrow Q_c(Hoch(A_\bullet), Hoch(A_\bullet'))$ in $CoCo\mathcal Alg_k$.
    \end{proof}

We now fix some notation: for $T$, $T'\in  \mathscr Gr\mathcal Alg(\mathcal C_k)$, we denote by $\Phi(T_i,T'_i):Q_c(T,T')\longrightarrow \mathcal C_k(T_i,T'_i)$, $i\geq 0$ the maps that determine the universal cocommutative measuring $Q_c(T,T')
\longrightarrow   \mathscr Gr(\mathcal C_k)(T,T')$ from $T$ to $T'$. In particular, we have maps 
\begin{equation}\label{4.10p}
\Phi(C_*(A_i),C_*(A'_i)):Q_c(Hoch(A_\bullet), Hoch(A_\bullet'))\longrightarrow   \mathcal C_k(C_*(A_i),C_*(A'_i))\qquad i\geq 0
\end{equation}
 for $A_\bullet,A_\bullet'\in \mathscr Mult_k$.

    \begin{Thm}\label{thm4.5}
        There is a $CoCo\mathcal Alg_k$-enriched functor $\gamma: MULT_k\longrightarrow \widetilde{MULT_k}$ which is identity on objects and whose mapping on Hom objects is given by 
        \begin{equation*}
            \gamma(A_\bullet,A_\bullet'):MULT_k(A_\bullet,A_\bullet')=\mathcal M_c(A_\bullet,A_\bullet')\longrightarrow Q_c(Hoch(A_\bullet), Hoch(A_\bullet'))=\widetilde{MULT}_k(A_\bullet,A_\bullet')
        \end{equation*}
        for $A_\bullet,A_\bullet'\in MULT_k$.
    \end{Thm}

    \begin{proof}
        For  $A_\bullet,A_\bullet',A_\bullet''\in MULT_k$, we need to show that the  following diagram in  $CoCo\mathcal Alg_k$   is commutative
         \begin{equation}\label{diag4.8}
            \begin{tikzcd}[row sep=2.5em,  column sep=5em]
            \mathcal M_c(A_\bullet,A_\bullet')\otimes \mathcal M_c(A'_\bullet,A_\bullet'')
  \arrow[r, "\circ"]
  \arrow[d, "{\gamma(A_\bullet,A_\bullet')\otimes \gamma(A_\bullet',A_\bullet'')} "']
&
\mathcal M_c(A_\bullet,A_\bullet'')
  \arrow[d, "{\gamma(A_\bullet,A_\bullet'')}"]
\\
Q_c(Hoch(A_\bullet), Hoch(A_\bullet'))\otimes Q_c(Hoch(A_\bullet'), Hoch(A_\bullet''))
  \arrow[r,"\circ"']
&
Q_c(Hoch(A_\bullet), Hoch(A_\bullet''))
\end{tikzcd}
        \end{equation}
        Because of the universal property of $Q_c(Hoch(A_\bullet), Hoch(A_\bullet''))$, it is enough to show that the two compositions appearing in  \eqref{diag4.8} are equal when further composed with each of the maps  
\begin{equation} \Phi(C_*(A_i),C_*(A_i'')): Q_c(Hoch(A_\bullet), Hoch(A_\bullet''))\longrightarrow \mathcal C_k(C_*(A_i),C_*(A_i''))\qquad i\geq 0
\end{equation} Since all the maps appearing in \eqref{diag4.8} are determined by universal properties of measuring coalgebras, it follows  from the construction in Theorem \ref{Thm3.6} that the outer rectangle as well as the lower rectangle in the following diagram commutes.
\begin{equation}\label{rect412}
\begin{CD}
  \mathcal M_c(A_\bullet,A_\bullet')\otimes \mathcal M_c(A'_\bullet,A_\bullet'') @>\circ>> \mathcal M_c(A_\bullet,A_\bullet'')\\
@V\gamma(A_\bullet,A_\bullet')\otimes \gamma(A_\bullet',A_\bullet'')VV @VV\gamma(A_\bullet,A_\bullet'')V\\
Q_c(Hoch(A_\bullet), Hoch(A_\bullet'))\otimes Q_c(Hoch(A_\bullet'), Hoch(A_\bullet'')) @>\circ>>  Q_c(Hoch(A_\bullet), Hoch(A_\bullet''))\\ 
@V \Phi(C_*(A_i),C_*(A_i'))\otimes  \Phi(C_*(A_i'),C_*(A_i''))VV @VV \Phi(C_*(A_i),C_*(A_i''))V \\
 \mathcal C_k(C_*(A_i),C_*(A_i'))\otimes  \mathcal C_k(C_*(A'_i),C_*(A_i''))@>\circ>> \mathcal C_k(C_*(A_i),C_*(A_i''))\\
\end{CD}
\end{equation} for each $i\geq 0$. From the above reasoning, it now follows that \eqref{diag4.8} commutes. 

\smallskip
 Finally, when $k$ is treated as a coalgebra, we know that the $p$-th iterated coproduct gives $\Delta^p(1)={1\otimes 1\otimes \cdots \otimes 1}$ 
($(p+1)$-times) and it is clear from the construction in Theorem \ref{Thm3.6} that the composition $k\longrightarrow\mathcal M_c(A_\bullet,A_\bullet)\xrightarrow{\gamma(A_\bullet,A_\bullet)}Q_c(Hoch(A_\bullet),Hoch(A_\bullet))$ is equal to the canonical map $k\longrightarrow Q_c(Hoch(A_\bullet),Hoch(A_\bullet))$. This completes the proof.
    \end{proof}

\section{Bimodules over multiplicative sequences of algebras and Hochschild complexes}

Let $\left(A_\bullet=\{A_n\}_{n\geq 0},\tau=\{\tau_{m,n}:A_m\otimes A_n\longrightarrow A_{m+n}\}_{m,n\geq 0}\right)\in \mathscr Mult_k$. For  $m$, $n\geq 0$, we note that an $A_{m+n}$-bimodule $U$  may be treated as an $A_m\otimes A_n$-bimodule by restriction of scalars, applying the $k$-algebra morphism $\tau_{m,n}:A_m\otimes A_n\longrightarrow A_{m+n}$. 
Using this observation, we now introduce two notions, one of ``left bimodule'' and the other of ``right bimodule'' over a multiplicative sequence of algebras.

    \begin{defn}\label{D5.1kos}
            Let $(A_\bullet,\tau)\in \mathscr Mult_k$. A left $A_\bullet$-bimodule $(U_\bullet,\vartheta)$ consists of 
            \begin{enumerate}[(a)]
                \item a collection $U_\bullet=\{U_n\}_{n\geq 0}$, where each $U_n$ is an $A_n$-bimodule, and
                \item  a collection $\vartheta=\left\{\vartheta_{m,n}:A_m\otimes U_n\longrightarrow U_{m+n}\right\}_{m,n\geq 0}$, where each $\vartheta_{m,n}$ is a homomorphism of $(A_m\otimes A_n)$-bimodules such that the following diagram
                
               \begin{equation}\label{diag5.2p}
               \begin{tikzcd}[row sep=3em,  column sep=4em]
                  A_l\otimes A_m\otimes U_n\arrow[r, "A_l\otimes\vartheta_{m,n} "]
                   \arrow[d, "\tau_{l,m}\otimes U_n"']
                   & A_l\otimes U_{m+n}  \arrow[d, "\vartheta_{l,m+n}"] \\
                   A_{l+m} \otimes U_n\arrow[r, "\vartheta_{l+m,n}"']
                   &  U_{l+m+n}
               \end{tikzcd}
               \end{equation}
            commutes for all $l,m,n\geq 0$. For any $n\geq 0$, we always assume that $\vartheta_{0,n}=id:U_n\cong k\otimes U_n=A_0\otimes U_n\longrightarrow U_n $. 
            \end{enumerate}

            Let $(U_\bullet,\vartheta),(U_\bullet',\vartheta')$ be left $A_\bullet$-bimodules. A morphism $g_\bullet:(U_\bullet,\vartheta)\longrightarrow (U_\bullet',\vartheta')$ of left $A_\bullet$-bimodules is a collection $g_\bullet=\left\{g_n:U_n\longrightarrow U_n'\right\}_{n\geq 0}$, where each $g_n$ is a morphism of $A_n$-bimodules such that the following diagram  commutes
            \begin{equation}
            \begin{tikzcd}[row sep=2.5em,  column sep=3em]
                A_m \otimes U_n
                \arrow[r, "A_m \otimes g_{n}"]
                \arrow[d, "\vartheta_{m,n}"']
                &A_{m} \otimes U'_n
                \arrow[d, "\vartheta_{m,n}'"]\\
                U_{m+n} \arrow[r, "g_{m+n}"']
                &U'_{m+n}
            \end{tikzcd}
            \end{equation}
            for all $m,n\geq 0$. We denote the category of left $A_\bullet$-bimodules by $A_\bullet{\small{-}}BMod$. Similarly, we can define the category $BMod{\small{-}}A_\bullet$ of right $A_\bullet$-bimodules.
     \end{defn}

      We now present some examples of left bimodules over multiplicative sequences of algebras. Indeed, if $f_\bullet:(A_\bullet,\tau)\longrightarrow (A_\bullet',\tau')$ is a morphism in $\mathscr Mult_k$, then $A_\bullet'$  becomes a left $A_\bullet$-bimodule equipped with the following maps:

     \begin{equation}
         \vartheta_{m,n}:A_m\otimes A_n'\longrightarrow A_{m+n}'\qquad(a^m\otimes {a'}^n)\mapsto \tau_{m,n}'(f_m(a^m)\otimes {a'}^n)\quad \forall m,n\geq 0
     \end{equation}
    for $a^m\in A_m$ and ${a'}^n\in A_n'$.

    \begin{eg}
      \emph{  Let $(A_\bullet,\tau)\in\mathscr Mult_k$. Then, $A_\bullet$ is a left $A_\bullet$-bimodule equipped with the structure maps $\vartheta_{m,n}=\tau_{m,n}:A_m\otimes A_n\longrightarrow A_{m+n}$ for all $m,n\geq 0$. }
    \end{eg}

     \begin{eg}
        \emph{
        Let $A$ be a commutative $k$-algebra having   product map $\mu:A\otimes A\longrightarrow A$. Since $A$ is commutative, we note that $\mu$ is a morphism of algebras. We set $A_0:=k$ and $A_i:=A$ for $i\geq 1$, along with $\tau_{m,n}:=\mu:A\otimes A\longrightarrow A$ for all $m,n\geq 1$. It is easy to verify that $A_\bullet:=\{A_i\}_{i\geq 0}$ together with $\tau:=\{\tau_{m,n}\}_{m,n\geq 0}$ forms a multiplicative sequence of algebras. Let $U$ be an $A$-module with structure map $\vartheta: A\otimes U\longrightarrow  U$. Since $A$ is commutative, we may treat $U$ as an $A$-bimodule.   We may now verify that the collection $U_\bullet:=\{U_n=U\}_{n\geq 0}$ together with  $\left\{\vartheta_{m,n}=\vartheta: A\otimes U\longrightarrow U\right\}_{m\geq 1,n\geq 0}$ forms a left $A_\bullet$-bimodule.}
     \end{eg}

     \begin{eg}
        \emph{ For $n\geq 1$, let $B_n$ denote the braid group generated by the symbols $t_1,t_2,\ldots,t_{n-1}$ subject to the relations:}
         \begin{equation*}
             t_it_{i+1}t_i=t_{i+1}t_it_{i+1},\qquad t_it_j=t_jt_i\quad\textit{ for } \quad |i-j|>1,
         \end{equation*}
        \emph{and let $B_0$ be the trivial group. If $k[B_n]$ denotes the group algebra of $B_n$ over $k$, we know from \cite[\S 4.1]{DAV} that the collection $k[B_\bullet]=\{k[B_n]\}_{n\geq 0}$ forms an object of $\mathscr Mult_k$, with  structure maps $\tau_{m,n}:k[B_m]\otimes k[B_n]\longrightarrow k[B_{m+n}]$ determined by 
        \begin{equation*}
        t_i\otimes 1\mapsto t_i\qquad 1\otimes t_j\mapsto t_{j+m}
        \end{equation*}
        for $m$, $n\geq 1$.
       We now fix an element $0\ne q\in k$. Then for each $n\geq 1$, the quotient  $H_n(q):=k[B_n]/I_n$, where $I_n$ is the two sided ideal of $k[B_n]$ generated by the elements $(t_i-q)(t_i+q^{-1})$ with $1\leq i\leq n-1$, is called the Hecke algebra. In \cite[\S 4.1]{DAV} it is also shown that the collection $H_\bullet(q)=\{H_n(q)\}_{n\geq 0}$ with $H_0(q)=k$, forms a multiplicative sequence of algebras, whose structure maps are descended from those of $k[B_\bullet]$.  It follows that the collection $\left\{k[B_n]\longrightarrow k[B_n]/I_n=H_n(q)\right\}_{n\geq 0}$ of canonical maps forms a morphism in $\mathscr Mult_k$. This shows that $H_\bullet(q)$ is a left $k[B_\bullet]$-bimodule.}
     \end{eg}

    \begin{eg}
        \emph{ For $n\geq 1$, let $k[S_n]$ be the group algebra of the permutation group $S_n$ over $k$. We let $S_0$ be the trivial group. Then, we know from \cite[\S 3.1]{DAV} that the collection  $k[S_\bullet]:=\{k[S_n]\}_{n\geq 0}$ forms an object of $\mathscr Mult_k$, where the structure maps $ \tau_{m,n}:k[S_m] \otimes k[S_n]\longrightarrow k[S_{m+n}]$ are induced by the canonical embeddings $S_m\times S_n\hookrightarrow S_{m+n}$. For an associative $k$-algebra $A$, let $SA_n:=A^{\otimes n}\ast S_n$ denote the cross product algebra, 
        which coincides with $A^{\otimes n}\otimes k[S_n]$ as a vector space, but whose multiplication is determined by (see \cite[\S 5]{DAV})
        \begin{equation*}
        ((a_1\otimes ...\otimes a_n)\ast \sigma)((a_1'\otimes ...\otimes a_n')\ast \sigma')=(a_1a'_{\sigma^{-1}(1)}\otimes ...\otimes a_na'_{\sigma^{-1}(n)})\ast \sigma\sigma'
        \end{equation*}
        From \cite[$\S$ 5]{DAV}, we know that the collection $SA_\bullet:=\{SA_n\}_{n\geq 0}$ determines an object of $\mathscr Mult_k$. Furthermore, \cite[\S 5.2]{DAV} shows that for each $n\geq 0$, the algebra $k[S_n]$ is naturally embedded into the algebra $SA_n$ which defines a morphism from $k[S_\bullet]$ to $SA_\bullet$ in $\mathscr Mult_k$. Hence, $SA_\bullet$ is a left $k[S_\bullet]$-bimodule.}
        
    \end{eg}

      We now recall the notion of measuring comodules between modules over algebras  (see \cite{MB}, \cite{MH}). Let $A,A'$ be $k$-algebras, $U$ be a left $A$-module, and $U'$ be a left $A'$-module. If $(C,\Phi:C\longrightarrow Hom_k(A,A'))$ is a coalgebra measuring from $A$ to $A'$, then, a (left) measuring comodule over $(C,\Phi)$ from $U$ to $U'$ is a pair $(P,\Psi)$ consisting of
   \begin{enumerate}[(a)]
     \item a left $C$-comodule $P$, and 
     \item a $k$-linear map $\Psi: P\longrightarrow  {Hom}_k(U,U')$ such that for each $a\in A$, $u\in U$, $y\in P$, we have
     \begin{equation*}
         \Psi(y)(a\cdot u)=\sum\Phi(y_{(0)})(a)\cdot\Psi(y_{(1)})(u)
     \end{equation*}
     where  the coaction $\Delta_P:P\longrightarrow C\otimes P$ is given by $\Delta_P(y)=y_{(0)}\otimes y_{(1)}\in C\otimes P$ by suppressing the summation and  in Sweedler notation.
 \end{enumerate} 
  Our next aim is to define the notion of measuring comodule between left bimodules over multiplicative sequences of algebras. For this, we recall from \cite[Lemma 5.4]{AB2} that if $A$, $A'$ are $k$-algebras and $(C,\Phi)$ is a coalgebra measuring from $A$ to $A'$ with $C$ cocommutative, then there is an induced coalgebra measuring $(C,\Phi^e)$ between the enveloping algebras $A^e:=A\otimes A^{op}$ and $A'^e:=A'\otimes A'^{op}$ given by
     \begin{equation}
         \Phi^e: C\longrightarrow  {Hom}(A^e,A'^e)\qquad \Phi^e(x)(a_1\otimes a_2)\longmapsto \Phi(x_{(1)})(a_1)\otimes \Phi(x_{(2)})(a_2)
     \end{equation}
     for all $x\in C$  and $a_1,a_2\in A$. We know that an $A$-bimodule  $U$ becomes a  left $A^e=A\otimes A^{op}$-module by means of the action
  \begin{equation*}
    A^e\otimes U= A\otimes A^{op}\otimes U\longrightarrow U\qquad(a_1\otimes a_2)\cdot u=(a_1\cdot u\cdot a_2) 
  \end{equation*}
  for $a_1\in A,a_2\in A^{op}$ and $u\in U$. We first make the following definition.

 \begin{defn}\label{def5.4}
   Let $A,A'$ be $k$-algebras and $(C,\Phi:C\longrightarrow Hom_k(A,A'))$ be a coalgebra measuring from $A$ to $A'$, with $C$  cocommutative. Let $U$ be an $A$-bimodule and $U'$ be an $A'$-bimodule. A  measuring comodule $(P,\Psi)$ over $(C,\Phi)$ from $U$ to $U'$ consists of a left $C$-comodule $P$ and a linear map
$
\Psi: P\longrightarrow Hom_k(U,U')
$
 such that 
     $(P,\Psi)$ is a measuring comodule over $(C,\Phi^e:C\longrightarrow Hom_k(A^e,A'^e))$ from $U$ to $U'$, where $U$ and $U'$ are regarded as left modules over $A^e$ and $A'^e$ respectively.
 \end{defn}
\smallskip

 We are now ready  to introduce measuring comodules between left bimodules over multiplicative sequences of algebras.
 
\begin{defn}\label{defn5.8}
    Let $A_\bullet,A_\bullet'\in \mathscr Mult_k$ and let $(U_\bullet,\vartheta)\in A_\bullet{\small{-}}BMod,(U_\bullet',\vartheta')\in A'_\bullet{\small{-}}BMod$ be left bimodules. Let $(C,\Phi=\{\Phi_n:C\longrightarrow Hom_k(A_n,A_n')\}_{n\geq 0})$ be a coalgebra measuring from $A_\bullet$ to $A_\bullet'$, with $C$  cocommutative. Then, a measuring comodule $(P,\Psi)$ over $(C,\Phi)$ from $U_\bullet$ to $U_\bullet'$ consists of
    \begin{enumerate}[(a)]
        \item a left $C$-comodule $P$, and 
        \item a collection $\Psi=\left\{\Psi_n:P\longrightarrow  {Hom}_k(U_n,U_n')\right\}_{n\geq 0}$ of $k$-linear maps such that 
        \begin{enumerate}[(1)]
            \item for each $n\geq 0$, $(P,\Psi_n)$ is a measuring comodule over $(C,\Phi_n)$ from $U_n$ to $U_n'$. In other words, the pair $(P,\Psi_n)$ is a measuring comodule over $(C,\Phi_n^e: C\longrightarrow Hom_k(A_n^e,A_n'^e))$ from $U_n$ to $U_n'$, where $U_n$ and $U_n'$ are regarded as left modules over $A_n^e$ and $A_n'^e$ respectively, and
            \item  for each $y\in P$ the following diagram 
            \begin{equation}\label{diag5.10}
                \begin{tikzcd}[row sep=2.5em,  column sep=3em]
          A_m \otimes U_n
  \arrow[r, "\vartheta_{m,n}"]
  \arrow[d, "\sum \Phi_m(y_{(0)})\otimes\Psi_n(y_{(1)})"']
&
U_{m+n}
  \arrow[d, "\Psi_{m+n}(y)"]
\\
A'_{m}\otimes U'_n
  \arrow[r, "\vartheta'_{m,n}"']
&
U'_{m+n}
\end{tikzcd}
            \end{equation}
       commutes for all $m,n\geq 0$, where the coaction $\Delta_P:P\longrightarrow C\otimes P$ is given by $\Delta_P(y)=y_{(0)}\otimes y_{(1)}\in C\otimes P$.
        \end{enumerate}
    \end{enumerate}
\end{defn}
\smallskip

We will now construct universal measuring comodules between left bimodules over multiplicative sequences of algebras. Since $k$ is a field, the coalgebra $C$ is flat over $k$ and it follows (see \cite[\S 1]{RW}) that the category $C{\small{-}}\mathcal CoMod$  of left $C$-comodules is a Grothendieck category.  Since the forgetful functor $C{\small{-}}\mathcal CoMod\longrightarrow \mathcal Vect_k$ preserves limits, it must have a  right adjoint $\mathfrak R_C: \mathcal Vect_k\longrightarrow C{\small{-}}\mathcal CoMod$ (see \cite[Proposition 8.3.27]{MK}). In other words, for any $k$-vector space $V$ and any left $C$-comodule $P$, there is a natural isomorphism
\begin{equation}\label{5.11p}
    \mathcal Vect_k(P,V)\cong C{\small{-}}\mathcal CoMod(P,\mathfrak R_C(V)).
\end{equation}
We now have the following result. 
\begin{Thm}\label{Thm5.4}
    Let $A_\bullet,A_\bullet'\in \mathscr Mult_k$. Let $U_\bullet\in A_\bullet{\small{-}}BMod$, $U_\bullet'\in A'_\bullet{\small{-}}BMod$ and  let $(C,\Phi=\{\Phi_n:C\longrightarrow Hom_k(A_n,A_n')\}_{n\geq 0})$ be a coalgebra measuring from $A_\bullet$ to $A_\bullet'$, where $C$ is cocommutative. Then, there exists a left $C$-comodule $\mathcal P_C(U_\bullet,U_\bullet')$ and a measuring comodule \begin{equation}
        \left(\mathcal P_C(U_\bullet,U_\bullet'),\Psi=\{\Psi_n:\mathcal P_C(U_\bullet,U_\bullet')\longrightarrow  {Hom}_k(U_n,U_n')\}_{n\geq 0}\right)
    \end{equation} 
     over $(C,\Phi)$ from $U_\bullet$ to $U_\bullet'$ satisfying the following universal property: for any measuring comodule $(P,\daleth=\{\daleth_n:P\longrightarrow Hom_k(U_n,U_n')\}_{n\geq 0})$ over $(C,\Phi)$ from $U_\bullet$ to $U_\bullet'$, there exists a unique morphism $\eta:P\longrightarrow \mathcal P_C(U_\bullet,U_\bullet')$ of left $C$-comodules such that the following diagram 
    \begin{equation*}
         \begin{tikzpicture}[>=stealth]
\node (A) at (0,0) {$P$};
\node (B) at (2.0,1.5) {$\mathcal P_C(U_\bullet,U_\bullet')$};
\node (C) at (4.7,0) {$ {Hom}_k(U_n,U_n')$};

\draw[->, dashed] (A) -- node[above left] {$\eta$} (B);
\draw[->] (A) -- node[below left] {$\daleth_n$} (C);
\draw[->] (B) -- node[right] {$\Psi_n$} (C);
\end{tikzpicture}
    \end{equation*}
    commutes for each  $n\geq 0$.
\end{Thm}

\begin{proof}
      We set $H:=\underset{n\geq 0}{\prod}{Hom}_k(U_n,U_n')$, together with canonical projections $\pi_n:H\longrightarrow  {Hom}_k(U_n,U_n')$ for all $n\geq 0$. By the adjunction in \eqref{5.11p}, we obtain the left $C$-comodule $\mathfrak R_C(H)$ and the canonical morphism $\Lambda(H):\mathfrak R_C(H)\longrightarrow  H$ of vector spaces. We now set $\mathcal P_C(U_\bullet, U_\bullet'):=\sum Q$, where the sum is taken over all left $C$-subcomodules $Q$ of $\mathfrak R_C(H)$ such that $\Big\{(\pi_n\circ\Lambda(H))\big|_{Q}:Q\longrightarrow Hom_k(U_n,U_n')\Big\}_{n\geq 0}$ is a measuring comodule over $(C,\Phi)$ from $U_\bullet$ to $U_\bullet'$. It is  clear that $\mathcal P_C(U_\bullet, U_\bullet')$ is a left $C$-comodule, and the collection $\Psi:=\left\{\Psi_n:=(\pi_n\circ\Lambda(H))\big|_{\mathcal P_C(U_\bullet, U_\bullet')}:\mathcal P_C(U_\bullet, U_\bullet')\longrightarrow  {Hom}_k(U_n,U_n')\right\}_{n\geq 0}$ determines a measuring comodule over $(C,\Phi)$ from $U_\bullet$ to $U_\bullet'$.
      \smallskip 
      
      We now consider a measuring comodule $(P,\daleth=\{\daleth_n: P\longrightarrow Hom_k(U_n, U_n')\}_{n\geq 0})$ over $(C,\Phi)$ from $U_\bullet$ to $U_\bullet'$. From the adjunction in \eqref{5.11p}, we note that the map $\daleth=\{\daleth_n\}_{n\geq 0}:P\longrightarrow   H=\underset{n\geq 0}{\prod}Hom_k(U_n,U_n')$ in $ \mathcal Vect_k$ corresponds to a left $C$-comodule morphism $\eta: P\longrightarrow \mathfrak R_C(H)$ such that $\Lambda(H)\circ \eta=\daleth$. Therefore, for each $n\geq 0$, we have the following commutative diagram 

          \begin{equation}\label{5.13digu}
        \begin{array}{c}
\begin{tikzcd}[column sep=large, row sep=large]
& P \arrow[dl, dashed, "\eta"'] 
    \arrow[d, "\daleth"] 
    \arrow[dr, "{\daleth_n}"] & \\
\mathfrak R_C(H) \arrow[r, "\Lambda(H)"] 
& H \arrow[r, "\pi_n"] 
& { {Hom}}_k(U_n,U_n')
\end{tikzcd}
\end{array}
\end{equation}
    From \eqref{5.13digu},  it follows  that $\left\{(\pi_n\circ\Lambda(H))\big|_{\eta(P)}:\eta(P)\longrightarrow Hom_k(U_n,U_n')\right\}_{n\geq 0}$ is a measuring comodule over $(C,\Phi)$ from $U_\bullet$ to $U_\bullet'$. Hence, $\eta(P)\subseteq \mathcal P_C(U_\bullet,U_\bullet')$. The result is now clear.
\end{proof}

\smallskip

 \begin{thm}\label{prop5.11}
     Let $A_\bullet,A_\bullet',A_\bullet''\in \mathscr Mult_k$. Let $(U_\bullet,\vartheta),(U'_\bullet,\vartheta')$ and $(U''_\bullet,\vartheta'')$ be left bimodules over $A_\bullet,A_\bullet'$ and $A_\bullet''$, respectively. Suppose we have
     \begin{enumerate}[(a)]
         \item a coalgebra measuring $(C,\Phi=\{\Phi_n:C\longrightarrow Hom_k(A_n,A_n')\}_{n\geq 0})$, where $C$ is cocommutative, and a measuring comodule $(P,\Psi=\{\Psi_n:P\longrightarrow Hom_k(U_n,U_n')\}_{n\geq 0})$ over $(C,\Phi)$ from $U_\bullet$ to $U_\bullet'$,
         \item a coalgebra measuring $(C',\Phi'=\{\Phi'_n:C'\longrightarrow Hom_k(A'_n,A_n'')\}_{n\geq 0})$ with $C'$ cocommutative, and a measuring comodule $(P',\Psi'=\{\Psi'_n:P'\longrightarrow Hom_k(U'_n,U_n'')\}_{n\geq 0})$ over $(C',\Phi')$ from $U'_\bullet$ to $U_\bullet''$.
     \end{enumerate}
     Then, the pair $(P\otimes P',\Psi'\circ \Psi:=\{(\Psi'\circ \Psi)_n\}_{n\geq 0})$ determines a measuring comodule over $(C\otimes C',\Phi'\circ \Phi=\{(\Phi'\circ\Phi)_n:C\otimes C'\longrightarrow Hom_k(A_n,A_n'')\}_{n\geq 0})$ from $U_\bullet$ to $U_\bullet''$, where
     \begin{equation}\label{eq5.18}
         (\Psi'\circ \Psi)_n: P\otimes P'\xrightarrow[]{\quad\Psi_n\otimes \Psi'_n\quad}Hom_k(U_n,U_n')\otimes Hom_k(U_n',U_n'')\xrightarrow[]{\quad\circ\quad}Hom_k(U_n,U_n'')\quad\forall n\geq 0
     \end{equation}
 \end{thm}

\begin{proof}
By Proposition \ref{Prop4.2}, we know that  $\Phi'\circ \Phi=\{(\Phi'\circ \Phi)_n:C\otimes C'\longrightarrow Hom_k(A_n,A_n'')\}_{n\geq 0}$ is a coalgebra measuring from $A_\bullet$ to $A_\bullet''$. For $y\in P$, $y'\in P'$, we know that the coaction $\Delta_{P\otimes P'}:P\otimes P'\longrightarrow (C\otimes C')
\otimes (P\otimes P')$  is given by  $\Delta_{P\otimes P'}(y\otimes y')=(y_{(0)}\otimes y'_{(0)})\otimes (y_{(1)}\otimes y'_{(1)})$, where $\Delta_P(y)=y_{(0)}\otimes y_{(1)}$ and $\Delta_{P'}(y')=y'_{(0)}\otimes y'_{(1)}$ are the respective coactions on $P$ and $P'$ respectively. In a manner similar to the proof of 
Proposition \ref{Prop4.2}, it can now be verified that $(P\otimes P',\Psi'\circ \Psi:=\{(\Psi'\circ \Psi)_n\}_{n\geq 0})$ as defined in \eqref{eq5.18} satisfies the conditions for being a measuring comodule over $(C\otimes C',\Phi'\circ \Phi)$ from $U_\bullet$ to $U_\bullet''$.  

\end{proof}

\begin{cor}\label{cor5.12p}
    Under the same assumptions as in Proposition \ref{prop5.11}, there is a canonical morphism of left $(C\otimes C')$-comodules 
    \begin{equation}
        \mathcal P_C(U_\bullet,U_\bullet')\otimes \mathcal P_{C'}(U'_\bullet,U_\bullet'')\longrightarrow \mathcal P_{C\otimes C'}(U_\bullet,U_\bullet'').
    \end{equation}
\end{cor}
\begin{proof}
    This follows directly from Proposition \ref{prop5.11} by applying the universal property of $\mathcal P_{C\otimes C'}(U_\bullet,U_\bullet'')$ in Theorem 
\ref{Thm5.4}.
\end{proof}

    We now consider the ``global category of comodules'' $CoMod^c_k$ whose objects are pairs $(C,P)$, where $C$ is a cocommutative $k$-coalgebra and $P$ is a $C$-comodule.   A morphism $(f,g):(C,P)\longrightarrow (C',P')$ in $CoMod^c_k$ consists of a $k$-coalgebra map $f:C\longrightarrow C'$ and a morphism $g:P\longrightarrow P'$ of left $C'$-comodules, where $P$ is treated as a $C'$-comodule by corestriction of scalars along $f$.   For $(C,P),(C',P')\in CoMod^c_k$, the operation $(C,P)\otimes(C',P'):=(C\otimes C',P\otimes P')$ makes $CoMod^c_k$ a symmetric monoidal category. We will now construct a category $BMod_{\mathscr Mult_k}$ that is enriched over $CoMod^c_k$.

\begin{Thm}\label{ils5.11}
    Let $BMod_{\mathscr Mult_k}$ be the category whose objects are pairs $(A_\bullet,U_\bullet)$, where $A_\bullet$ is a multiplicative sequence of algebras and $U_\bullet$ is a left $A_\bullet$-bimodule, and whose Hom objects are   
        \begin{equation}
            BMod_{\mathscr Mult_k}((A_\bullet,U_\bullet),(A_\bullet',U_\bullet')):=(\mathcal M_c(A_\bullet,A_\bullet'),\mathcal P_{\mathcal M_c(A_\bullet,A_\bullet')}(U_\bullet,U_\bullet'))\in CoMod^c_k
        \end{equation}
        for $(A_\bullet,U_\bullet),(A_\bullet',U_\bullet')\in BMod_{\mathscr Mult_k}$. Then $BMod_{\mathscr Mult_k}$ is enriched over $CoMod^c_k$.
\end{Thm}

\begin{proof}
    Let $(A_\bullet,U_\bullet),(A_\bullet',U_\bullet')\in BMod_{\mathscr Mult_k}$. By definition, it is clear that the Hom object $(\mathcal M_c(A_\bullet, A_\bullet'),\mathcal P_{\mathcal M_c(A_\bullet, A_\bullet')}(U_\bullet, U_\bullet'))\in CoMod^c_k$. Let $(A_\bullet,U_\bullet),(A_\bullet',U_\bullet'),(A''_\bullet,U_\bullet'')\in BMod_{\mathscr Mult_k}$. By putting $C=\mathcal M_c(A_\bullet, A_\bullet')$ and $C'=\mathcal M_c(A'_\bullet, A_\bullet'')$ in  Corollary \ref{cor5.12p},  we obtain a morphism 
    \begin{equation}\label{5.17yl}
        \mathcal P_{\mathcal M_c(A_\bullet, A_\bullet')}(U_\bullet,U_\bullet')\otimes \mathcal P_{\mathcal M_c(A'_\bullet, A_\bullet'')}(U'_\bullet,U_\bullet'')\longrightarrow \mathcal P_{\mathcal M_c(A_\bullet, A_\bullet')\otimes \mathcal M_c(A'_\bullet, A_\bullet'')}(U_\bullet,U_\bullet'')
    \end{equation}
    of $\mathcal M_c(A_\bullet, A_\bullet')\otimes \mathcal M_c(A'_\bullet, A_\bullet'')$-comodules. We observe that the map in 
\eqref{5.17yl}   can be treated as a morphism of $\mathcal M_c(A_\bullet, A_\bullet'')$-comodules, using the map $\mathcal M_c(A_\bullet,A_\bullet')\otimes \mathcal M_c(A_\bullet',A_\bullet'')\longrightarrow \mathcal M_c(A_\bullet,A_\bullet'')$
    of $k$-coalgebras. This gives us the following map in $CoMod_k^c$
    \begin{equation}\label{eq5.22}
        \left(\mathcal M_c(A_\bullet,A_\bullet')\otimes \mathcal M_c(A'_\bullet,A_\bullet''),\mathcal P_{\mathcal M_c(A_\bullet,A_\bullet')}(U_\bullet,U_\bullet')\otimes \mathcal P_{\mathcal M_c(A'_\bullet,A_\bullet'')}(U'_\bullet,U_\bullet'')\right)\longrightarrow
\left(\mathcal M_c(A_\bullet,A_\bullet''),\mathcal P_{\mathcal M_c(A_\bullet,A_\bullet')\otimes \mathcal M_c(A'_\bullet,A_\bullet'')}(U_\bullet,U_\bullet'')\right).
    \end{equation}
    By defintion,  the $k$-coalgebra map $\mathcal M_c(A_\bullet,A_\bullet')\otimes \mathcal M_c(A_\bullet',A_\bullet'')\longrightarrow \mathcal M_c(A_\bullet,A_\bullet'')$ arises from the universal property of $\mathcal M_c(A_\bullet, A_\bullet'')$ determined by Proposition \ref{Prop4.1}. By the commutativity of \eqref{42digu}, it is now clear that  the collection $\{\mathcal P_{{\mathcal M_c(A_\bullet, A_\bullet')}\otimes {\mathcal M_c(A'_\bullet, A_\bullet'')}}(U_\bullet,U_\bullet'')\longrightarrow Hom_k(U_n,U_n'')\}_{n\geq 0}$ determines a measuring comodule over $\mathcal M_c(A_\bullet,A_\bullet'')$. By the universal property of $\mathcal P_{\mathcal M_c(A_\bullet,A_\bullet'')}(U_\bullet,U_\bullet'')$ in Theorem \ref{Thm5.4}, we now have the following morphism
    \begin{equation}\label{5.22p}
       \left(\mathcal M_c(A_\bullet, A_\bullet''),\mathcal P_{\mathcal M_c(A_\bullet, A_\bullet')\otimes \mathcal M_c(A'_\bullet, A_\bullet'')}(U_\bullet,U_\bullet'')\right)\longrightarrow \left(\mathcal M_c(A_\bullet, A_\bullet''),\mathcal P_{\mathcal M_c(A_\bullet, A_\bullet'')}(U_\bullet,U_\bullet'')\right)
    \end{equation}
    in $CoMod_k^c$. Combining \eqref{eq5.22} and \eqref{5.22p}, we  have the required composition 
\begin{equation} BMod_{\mathscr Mult_k}((A_\bullet,U_\bullet),(A_\bullet',U_\bullet'))
\otimes
BMod_{\mathscr Mult_k}((A'_\bullet,U'_\bullet),(A_\bullet'',U_\bullet''))\longrightarrow {BMod}_{\mathscr Mult_k}((A_\bullet,U_\bullet),(A_\bullet'',U_\bullet''))
\end{equation} 
 Finally, let $(A_\bullet, U_\bullet)\in BMod_{\mathscr Mult_k}$. We know from Theorem \ref{T4.3xq} that there is a morphism $k\longrightarrow \mathcal M_c(A_\bullet,A_\bullet)$ of $k$-coalgebras, which makes $k$ a left $\mathcal M_c(A_\bullet,A_\bullet)$-comodule. 
    It is easy to check that the collection $\{\Psi_n:k\longrightarrow Hom_k(U_n,U_n)\}_{n\geq 0}$, where $\Psi_n(t):=t\cdot id_{U_n}$, forms a measuring comodule, when $k$ is treated as a left $\mathcal M_c(A_\bullet,A_\bullet)$-comodule. By the universal property of $\mathcal P_{\mathcal M_c(A_\bullet,A_\bullet)}(U_\bullet,U_\bullet)$, we now have a morphism
    \begin{equation}
       (k,k)\longrightarrow \left(\mathcal M_c(A_\bullet,A_\bullet),\mathcal P_{\mathcal M_c(A_\bullet,A_\bullet)}(U_\bullet,U_\bullet)\right)= BMod_{ \mathscr Mult_k}((A_\bullet, U_\bullet),(A_\bullet, U_\bullet)). 
    \end{equation}
    in $CoMod_k^c$.
    \end{proof}

     We now recall from \cite[\S 1.1.1]{LD} that if $A$ is a $k$-algebra and $U$ is an $A$-bimodule, then the Hochschild complex $C_*(A,U)$ is defined by 
     \begin{equation*}
          C_*(A,U):\hspace{.5cm}\cdots\xrightarrow[]{\quad b\quad}C_n(A,U)\xrightarrow[]{\quad b\quad}C_{n-1}(A,U)\xrightarrow[]{\quad b\quad}\cdots \xrightarrow[]{\quad b\quad}C_0(A,U),
     \end{equation*}
     where $C_n(A,U)=U\otimes A^{\otimes n}$ and the differentials are given by 
     \begin{equation}\label{fdp1}
     \begin{array}{c}
           b:C_n(A,U)=U\otimes A^{\otimes n}\longrightarrow U\otimes A^{\otimes {n-1}}=C_{n-1}(A,U)\\
            (u,a_1,\ldots,a_n)\longmapsto (u\cdot a_1,a_2,\ldots,a_n)+\sum_{i=1}^{n-1}(-1)^i(u,a_1,\ldots,a_ia_{i+1},\ldots,a_n)+(-1)^n(a_n\cdot u,a_1,\ldots,a_{n-1})
     \end{array}
     \end{equation}
     where we write $ (u,a_1,\ldots,a_n)$ for an element $(u\otimes a_1\otimes \ldots\otimes a_n)\in C_n(A,U)=U\otimes A^{\otimes n}$. If $g:U\longrightarrow V$ is a morphism of $A$-bimodules, there is an induced chain map
     
     \begin{equation}\label{eq6.3}
    C_*(A,g):C_*(A,U)\longrightarrow C_*(A,V)\qquad (u,a_1,\ldots, a_n)\longmapsto (g(u),a_1,\ldots, a_n)\quad \forall n\geq 0.
     \end{equation}
   For $ (u,a_1,\ldots,a_n)\in C_n(A,U)$ and $\sigma\in S_n$, we set $\sigma\cdot  (u,a_1,\ldots,a_n):=(u,a_{\sigma^{-1}(1)},...,a_{\sigma^{-1}(n)})$. We know from \cite[\S 4.2.1]{LD} that if $A,A'$ are $k$-algebras, $U$ is an $A$-bimodule and $U'$ is an $A'$-bimodule, then the shuffle product on Hochschild complexes is determined by
     \begin{equation}\label{eq6.4}
     \begin{array}{c}
          sh_{p,q}: C_p(A,U)\otimes C_q(A',U')\longrightarrow C_{p+q}(A\otimes A',U\otimes U')  \\
          (u,a_1,\ldots,a_p)\otimes(u',a'_1,\ldots,a'_q)\longmapsto \underset{\sigma\in S_{p,q}}{\sum} sgn(\sigma)\left(\sigma\cdot(u\otimes u',a_1\otimes1,\ldots,a_p\otimes1,1\otimes a'_1,\ldots,1\otimes a'_q)\right),
     \end{array}   
     \end{equation}
    where $p,q\geq 0$ and $S_{p,q}$ is the set of all $(p,q)$-shuffles as explained in Section 2. 
    
    \smallskip 
    For $(A_\bullet,\tau)\in \mathscr Mult_k$, we now recall from Definition \ref{DL2.2xr} the graded algebra $Hoch(A_\bullet)\in  \mathscr Gr\mathcal Alg(\mathcal C_k)$ equipped with the multiplication 
\begin{equation}  
\tau_{m,n}^{sh}:C_*(A_m,A_m)\otimes C_*(A_n,A_n)=C_*(A_m)\otimes C_*(A_n)\xrightarrow{\quad sh\quad }C_*(A_m\otimes A_n)\xrightarrow{\qquad C_*(\tau_{m,n})\qquad}C_*(A_{m+n})=C_*(A_{m+n},A_{m+n})
\end{equation} 
    If  $U_\bullet\in A_\bullet{\small{-}}BMod$ is a left $A_\bullet$-bimodule, we will now show that   the collection $C_*(A_\bullet,U_\bullet)=\left\{C_*(A_n,U_n)\right\}_{n\geq 0}\in   \mathscr Gr(\mathcal C_k)$ forms a graded left module over the graded algebra $Hoch(A_\bullet)$.

       \begin{lem}\label{L5.12dh}
           Let $(A_\bullet,\tau)\in \mathscr{M}ult_k$ and $(U_\bullet,\vartheta)\in A_\bullet{\small{-}}BMod$. Then, $C_*(A_\bullet,U_\bullet)=\{C_*(A_n,U_n)\}_{n\geq 0}\in  \mathscr Gr(\mathcal C_k)$ becomes a graded left module over $Hoch(A_\bullet)=\left\{C_*(A_n)=C_*(A_n,A_n)\right\}_{n\geq 0}\in   \mathscr Gr\mathcal Alg(\mathcal C_k)$ by means of the action 
           \begin{equation}\label{eq5.7p}
\begin{array}{c}
\varsigma:Hoch(A_\bullet)\otimes  C_*(A_\bullet,U_\bullet)\longrightarrow  C_*(A_\bullet,U_\bullet)\\
               \varsigma_{m,n}:C_p(A_m,A_m)\otimes  C_q(A_n,U_n)\xrightarrow {\quad sh_{p,q}\quad}C_{p+q}(A_m\otimes A_n,A_m\otimes U_n)\xrightarrow{\quad C_{p+q}(\tau_{m,n},\vartheta_{m,n})\quad}  C_{p+q}(A_{m+n},U_{m+n})\\
\end{array}
           \end{equation}
          for all $m,n,p,q\geq 0$. 
       \end{lem}

       \begin{proof} The first map $sh_{p,q}$ in \eqref{eq5.7p} is the shuffle product obtained as in \eqref{eq6.4}.  By definition, $\vartheta_{m,n}:A_m\otimes U_n\longrightarrow U_{m+n}$ is a morphism of $(A_{m}\otimes A_n)$-bimodules, which induces the morphism $C_*(A_m\otimes A_n,\vartheta_{m,n}):C_*(A_m\otimes A_n,A_m
\otimes U_n)\longrightarrow C_*(A_m\otimes A_n,U_{m+n})$ of complexes as in \eqref{eq6.3}. Further, since the $(A_{m}\otimes A_n)$-bimodule structure on the $A_{m+n}$-bimodule $U_{m+n}$ is obtained by restriction of scalars along the morphism $\tau_{m,n}:A_m\otimes A_n\longrightarrow A_{m+n}$,  it is clear from \eqref{fdp1} that we have a morphism $C_*(\tau_{m,n},U_{m+n}):C_*(A_m\otimes A_n,U_{m+n})\longrightarrow C_*(A_{m+n},U_{m+n})$ between Hochschild complexes. Setting $C_{p+q}(\tau_{m,n},\vartheta_{m,n}):=C_{p+q}(\tau_{m,n},U_{m+n})\circ C_{p+q}(A_m\otimes A_n,\vartheta_{m,n})$, the result is clear. 

       \end{proof}

\smallskip 

We will now see that a measuring comodule between bimodules over algebras induces a morphism between their Hochschild complexes.

\begin{lem}\label{lem5.4pp}
   Let $A$, $A'$ be $k$-algebras, and $(C,\Phi:C\longrightarrow Hom_k(A,A'))$ be a coalgebra measuring from $A$ to $A'$, with $C$ cocommutative. Let $ U$ and $U'$ be bimodules over $A$ and $A'$ respectively, and $(P,\Psi:P\longrightarrow  {Hom}_k(U,U'))$ be a measuring comodule over $(C,\Phi)$ from $U$ to $U'$. Then for each $y\in P$, we have a morphism $C_*^\Psi(y):C_*(A,U)\longrightarrow C_*(A',U')$ of complexes given by
   
\begin{equation}\label{eq5.14p}
\begin{array}{c}
        C_n^\Psi(y):C_n(A,U)=U\otimes A^n\longrightarrow U'\otimes A'^n=C_n(A',U')\\
        C_n^\Psi(y)(u,a_1,\ldots,a_n):=\sum\left(\Psi(y_{(n)})(u),\Phi(y_{(0)})(a_1),\ldots,\Phi(y_{(n-1)})(a_n)\right)\\
      \end{array}
\end{equation}
for $(u,a_1,\ldots,a_n)\in C_n(A,U)$, where the iterated coaction on $y\in P$ is given by $\Delta^{n}_P(y)=  \sum y_{(0)}\otimes y_{(1)}\otimes\ldots\otimes y_{(n)}\in  C^{\otimes n}\otimes P$.
\end{lem}

\begin{proof}
   For $n\geq 0$, we know   that the Hochschild differential $b:C_n(A,U)\longrightarrow C_{n-1}(A,U)$ is given by $b=\sum_{i=0}^n(-1)^ib_i$ where 
        \begin{align*}
            & b_0(u,a_1,\ldots,a_n):=(ua_1,\ldots,a_i a_{i+1},\ldots, a_n)\\
            & b_i(u,a_1,\ldots,a_n):=(u,a_1,\ldots,a_i a_{i+1},\ldots, a_n) \qquad\textit{for $0<i < n-1$} \\
            & b_n(u,a_1,\ldots,a_n):=(a_nu,a_1.\ldots,a_{n-1})
        \end{align*}
     Therefore, for $y\in P$, it is enough to verify that  $C_n^\Psi(y)$ commutes with each $b_i$.  Since $\Phi:C\longrightarrow Hom_k(A,A')$ is a coalgebra measuring, for $0 < i< n-1$ and $(u,a_1,\ldots,a_n)\in C_n(A,U)$ we have 
     \begin{equation}\label{5.23ks}
     \begin{array}{ll}
      C_{n-1}^\Psi(y)(b_i(u,a_1,\ldots,a_n))
         &=C_{n-1}^\Psi(y)(u,a_1,\ldots,a_i a_{i+1},\ldots, a_n)\\
         &=\sum\left(\Psi(y_{(n-1)})(u),\Phi(y_{(0)})(a_1),\ldots,\Phi(y_{(i-1)})(a_ia_{i+1}),\ldots,\Phi(y_{(n-2)})(a_n)\right)\\
         &=\sum\left(\Psi(y_{(n)})(u),\Phi(y_{(0)})(a_1),\ldots,\Phi(y_{(i-1)})(a_i)\Phi(y_{(i)})(a_{i+1}),\ldots,\Phi(y_{(n-1)})(a_n)\right)\\
        & =b_i\left(C^\Psi_{n}(y)(u,a_1,\ldots,a_n)\right)\\
     \end{array}
     \end{equation}
 By Definition \ref{def5.4}, we know that 
     $(P,\Psi)$ is a measuring comodule over $(C,\Phi^e:C\longrightarrow Hom_k(A^e,A'^e))$ from $U$ to $U'$, where $U$, $U'$ are seen  as left modules over $A^e=A\otimes A^{op}$, $A'^e=A'\otimes A'^{op}$ respectively.  Accordingly, for $a\in A$, $u\in U$ and $y\in P$, we have
     \begin{equation*}
     \begin{array}{c}
     \Psi(y)(au)=\Psi(y)((a\otimes 1)\cdot u)=\Phi^e(y_{(0)})(a\otimes 1)\cdot \Psi(y_{(1)})(u)=(\Phi(y_{(0)})(a)\otimes \Phi(y_{(1)})(1))\cdot \Psi(y_{(2)})(u)=\Phi(y_{(0)})(a)\Psi(y_{(1)})(u)\\
          \Psi(y)(ua)=\Psi(y)((1\otimes a)\cdot u)=\Phi^e(y_{(0)})(1\otimes a)\cdot \Psi(y_{(1)})(u)=(\Phi(y_{(0)})(1)\otimes \Phi(y_{(1)})(a))\cdot \Psi(y_{(2)})(u)=\Psi(y_{(1)})(u)\Phi(y_{(0)})(a)\\
     \end{array}
     \end{equation*} Since $C$ is cocommutative, it now follows by a computation similar to \eqref{5.23ks} that $C_n^\Psi(y)$  commutes with $b_0$ and $b_n$. 
    
\end{proof}

 \begin{lem}
       Let $(A_\bullet,\tau),(A_\bullet',\tau')\in \mathscr Mult_k$ and $(C,\Phi=\{\Phi_n:C\longrightarrow Hom_k(A_n,A_n')\}_{n\geq 0})$ be a coalgebra measuring from $A_\bullet$ to $A_\bullet'$, with $C$ cocommutative.
    If $(U_\bullet,\vartheta)\in {A_\bullet\small{-}}BMod,(U'_\bullet,\vartheta')\in A'_\bullet\small{-}BMod$ and $(P,\Psi=\{\Psi_n:P\longrightarrow  {Hom}_k(U_n,U_n')\}_{n\geq 0})$ is a measuring comodule over $(C,\Phi)$ from $U_\bullet$ to $U_\bullet'$, then for each $y\in P$, the collection 
    \begin{equation}
     \widetilde{\Psi}(y)=\left\{\widetilde{\Psi_n}(y):=C_*^{\Psi_n}(y):C_*(A_n,U_n)\longrightarrow C_*(A_n',U_n')\right\}_{n\geq 0}
    \end{equation} determines a morphism $\widetilde{\Psi}(y): C_*(A_\bullet,U_\bullet)\longrightarrow  C_*(A'_\bullet,U'_\bullet)$ in $\mathscr Gr(\mathcal C_k)$.
 \end{lem}
 \begin{proof}
 By Definition \ref{defn5.8},  for each $n\geq 0$, $(P,\Psi_n)$ is a measuring comodule over $(C,\Phi_n:C\longrightarrow Hom_k(A_n,A'_n))$ from $U_n$ to $U_n'$. The result is now clear by applying Lemma \ref{lem5.4pp}. 
 \end{proof}

We are now ready to show that a measuring comodule $(P,\Psi=\{\Psi_n\}_{n\geq 0})$ between left bimodules $U_\bullet$ and $U'_\bullet$ over multiplicative sequences of algebras $A_\bullet$ and $A'_\bullet$ respectively, induces a measuring comodule from $C_*(A_\bullet,U_\bullet)$ to $C_*(A_\bullet',U_\bullet')$. For this, we first need the notion of a measuring comodule between graded modules over graded algebras in the monoidal category $\mathcal C_k$.

\begin{defn}\label{D5.tj}
    Let $T=\{T_{n,*}\}_{n\geq 0}$ and $T'=\{T_{n,*}\}_{n\geq 0}$ be graded algebras in $\mathcal C_k$. Let $C$ be a cocommutative $k$-coalgebra  and let  $(C,\Xi:C\longrightarrow \mathscr Gr(\mathcal C_k)(T,T'))$ be a coalgebra measuring from $T$ to $T'$ in the sense of Definition \ref{def3.5},  where  $\Xi(x)=\left\{\Xi_i(x):T_{i,*}\longrightarrow T'_{i,*}\right\}_{i\geq 0}$ for $x\in C$. Let $Z=\{Z_{n,*}\}_{n\geq 0}\in \mathscr Gr(\mathcal C_k)$ be a graded left  $T$-module  and let $Z'=\{Z'_{n,*}\}_{n\geq 0}\in \mathscr Gr(\mathcal C_k)$ be a graded left $T'$-module equipped respectively with the actions
    \begin{equation}
        \nu=\left\{\nu_{m,n}:T_{m,*}\otimes Z_{n,*}\longrightarrow Z_{m+n,*}\right\}_{m,n\geq 0}\qquad \textit{and} \qquad\nu'=\left\{\nu'_{m,n}:T'_{m,*}\otimes Z'_{n,*}\longrightarrow Z'_{m+n,*}\right\}_{m,n\geq 0}
    \end{equation} 
     A measuring comodule $(P,\beth)$ over $(C,\Xi)$ from $Z$ to $Z'$ consists of
\begin{enumerate}[(a)]
    \item a left $C$-comodule $P$ with coaction given by $\Delta_P(y)=\sum y_{(0)}\otimes y_{(1)}\in C\otimes P$ for $y\in P$, and
    \item a map $\beth:P\longrightarrow \mathscr Gr(\mathcal C_k)(Z,Z')$ such that the collection $\left\{\beth(y)=\left\{\beth_i(y):Z_{i,*}\longrightarrow Z'_{i,*}\right\}_{i\geq 0}\right\}_{y\in P}$ of $k$-linear maps satisfies 
    \begin{equation}\label{eq5.15p}
        \beth_{m+n}(y)\left(\nu_{m,n}(t^m_p\otimes z^n_q)\right)=\sum \nu_{m,n}'\left(\left(\Xi_m(y_{(0)})(t^m_p)\right)\otimes \left(\beth_n(y_{(1)})(z^n_q)\right)\right)
    \end{equation}
    for all $m,n,p,q\geq 0$ and $t^m_p\in T_{m,p}, z^n_q\in Z_{n,q}$.
\end{enumerate}

\end{defn}

\begin{thm}\label{prop5.5}
    Let $(A_\bullet,\tau),(A_\bullet',\tau')\in \mathscr Mult_k$. Let $(C,\Phi=\{\Phi_n:C\longrightarrow Hom_k(A_n,A_n')\}_{n\geq 0})$ be a coalgebra measuring from $A_\bullet$ to $A_\bullet'$ with $C$ cocommutative, and $(C,Hoch^{\Phi}: C\longrightarrow \mathscr Gr(\mathcal C_k)(Hoch(A_\bullet),Hoch(A_\bullet')))$ be the corresponding coalgebra measuring from $Hoch(A_\bullet)$ to $Hoch(A_\bullet')$.

\smallskip
    Let $(U_\bullet,\vartheta)\in {A_\bullet\small{-}}BMod$ and $(U'_\bullet,\vartheta')\in A'_\bullet\small{-}BMod$.  Let $(P,\Psi=\{\Psi_n:P\longrightarrow  {Hom}_k(U_n,U_n')\}_{n\geq 0})$ be a measuring comodule over $(C,\Phi)$ from $U_\bullet$ to $U_\bullet'$. Then, the pair $(P,\widetilde{\Psi})$ forms a measuring comodule over $(C,Hoch^{\Phi})$ from $C_*(A_\bullet,U_\bullet)$ to $C_*(A_\bullet',U_\bullet')$, where
    \begin{equation}\label{eq5.24pp}
        \begin{array}{c}
              \widetilde{\Psi}:P\longrightarrow \mathscr Gr(\mathcal C_k)(C_*(A_\bullet,U_\bullet),C_*(A_\bullet',U_\bullet')) \\
             y\longmapsto \widetilde{\Psi}(y)=\left\{\widetilde{\Psi_n}(y):=C_*^{\Psi_n}(y):C_*(A_n,U_n)\longrightarrow C_*(A_n',U_n')\right\}_{n\geq 0}.
        \end{array}
    \end{equation}
     
\end{thm}

\begin{proof}
Using the graded left module structures of $C_*(A_\bullet,U_\bullet)=\{C_*(A_n,U_n)\}_{n\geq 0}\in  \mathscr Gr(\mathcal C_k)$ and  $C_*(A'_\bullet,U'_\bullet)=\{C_*(A'_n,U'_n)\}_{n\geq 0}\in  \mathscr Gr(\mathcal C_k)$  over $Hoch(A_\bullet)$ and $Hoch(A'_\bullet)$ respectively, as described in Lemma \ref{L5.12dh}, the result follows from a computation similar to the proof of Theorem \ref{Thm3.6}.
\end{proof}

  Let $A,A'$ be $k$-algebras and let $(Q_c(A,A'),\Phi:Q_c(A,A')\longrightarrow Hom_k(A,A'))$ be the universal cocommutative measuring coalgebra   from $A$ to $A'$ as mentioned in Section 4.  By \cite[ Lemma 5.2, Lemma 5.4]{MB}, for any left modules $L$ and $L'$ over $A$ and $A'$ respectively, there is a universal measuring comodule 
  $(R_{Q_c(A,A')}(L,L'),\Psi: R_{Q_c(A,A')}(L,L')\longrightarrow Hom_k(L, L'))$ over $(Q_c(A,A'),\Phi)$ from $L$ to $L'$. Again as in Section 4, for $T, T'\in \mathscr Gr\mathcal Alg(\mathcal C_k)$, we consider the cocommutative coalgebra $Q_c(T,T')$ that gives a measuring that is universal among cocommutative measurings from $T$ to $T'$.   If $Z$, $Z'\in \mathscr Gr(\mathcal C_k)$ are graded left modules over $T$, $T'\in \mathscr Gr\mathcal Alg(\mathcal C_k)$ respectively, it can be shown in a manner similar to  \cite{MB} that there is a comodule $R_{Q_c(T,T')}(Z,Z')$ that is universal among comodule measurings over $Q_c(T,T')$ from $Z$ to $Z'$. 
  
  \smallskip
  For $A_\bullet\in \mathscr Mult_k$, we know that $Hoch(A_\bullet)\in \mathscr Gr\mathcal Alg(\mathcal C_k)$. If $U_\bullet\in A_\bullet{\small{-}}BMod$, we know from Lemma \ref{L5.12dh} that $C_*(A_\bullet, U_\bullet)\in \mathscr Gr(\mathcal C_k)$ is a graded left module over $Hoch(A_\bullet)$.  Accordingly, we can define a category $\widetilde{BMod}_{\mathscr Mult_k}$ whose objects are the same as those of $BMod_{\mathscr Mult_k}$ and whose Hom objects are given by 
\begin{equation}
    \widetilde{BMod}_{\mathscr Mult_k}\left((A_\bullet,U_\bullet),(A_\bullet',U_\bullet')\right):=\left(Q_c(Hoch(A_\bullet),Hoch(A_\bullet')),R_{Q_c(Hoch(A_\bullet),Hoch(A_\bullet'))}(C_*(A_\bullet,U_\bullet),C_*(A_\bullet',U_\bullet'))\right)\in CoMod_k^c
\end{equation}
for any $(A_\bullet,U_\bullet),(A_\bullet',U_\bullet')\in \widetilde{BMod}_{\mathscr Mult_k}$. We conclude this section by constructing a $CoMod_k^c$-enriched functor between the categories ${BMod}_{\mathscr Mult_k}$ and $\widetilde{BMod}_{\mathscr Mult_k}$.

\begin{Thm}\label{Thm5.17gm}
  For $(A_\bullet,U_\bullet),(A_\bullet',U_\bullet')\in{BMod}_{\mathscr Mult_k}$, there is a canonical morphism in $CoMod_k^c$
  \begin{equation}\label{529ub}
\begin{CD}
       \left(\mathcal M_c(A_\bullet,A_\bullet'),\mathcal P_{\mathcal M_c(A_\bullet,A_\bullet')}(U_\bullet,U_\bullet')\right) 
        \\
        @VV \left(\gamma(A_\bullet,A_\bullet'),\lambda(U_\bullet,U_\bullet')\right)V \\ 
             \Big(Q_c(Hoch(A_\bullet),Hoch(A_\bullet')), 
            R_{Q_c(Hoch(A_\bullet),Hoch(A_\bullet'))}(C_*(A_\bullet,U_\bullet),C_*(A_\bullet',U_\bullet'))\Big)\\
\end{CD}
    \end{equation} 
    Then, there is a $CoMod^c_k$-enriched functor $(\gamma,\lambda): BMod_{\mathscr Mult_k}\longrightarrow \widetilde{BMod}_{\mathscr Mult_k}$ which is identity on objects and whose mapping on Hom objects is given by \eqref{529ub}.  
\end{Thm}

\begin{proof}
  Let $(A_\bullet,U_\bullet),(A_\bullet',U_\bullet')\in{BMod}_{\mathscr Mult_k}$. From Lemma \ref{lem4.4} we have the canonical coalgebra map
    \begin{equation}\label{eq5.32pp}
      \gamma(A_\bullet, A_\bullet'):\mathcal M_c(A_\bullet,A_\bullet')\longrightarrow Q_c(Hoch(A_\bullet),Hoch(A_\bullet'))
    \end{equation}
     We now consider the universal measuring comodule $(\mathcal P_{\mathcal M_c(A_\bullet,A_\bullet')}(U_\bullet,U_\bullet'),\Psi)$ from $U_\bullet$ to $U_\bullet'$ over the universal coalgebra measuring $(\mathcal M_c(A_\bullet,A_\bullet'),\Phi)$ as in Theorem \ref{Thm5.4}. By Proposition \ref{prop5.5}, this induces a measuring comodule $(\mathcal P_{\mathcal M_c(A_\bullet,A_\bullet')}(U_\bullet,U_\bullet'),\widetilde{\Psi})$ over $( \mathcal M_c(A_\bullet,A_\bullet'),Hoch^{\Phi})$ from $C_*(A_\bullet,U_\bullet)$ to $C_*(A_\bullet',U_\bullet')$. 
     
     \smallskip By the construction of $\gamma(A_\bullet, A_\bullet')$ in Lemma \ref{lem4.4}, we note that $(\mathcal P_{\mathcal M_c(A_\bullet,A_\bullet')}(U_\bullet,U_\bullet'),\widetilde{\Psi})$  is still a measuring from $C_*(A_\bullet,U_\bullet)$ to $C_*(A_\bullet',U_\bullet')$ when $\mathcal P_{\mathcal M_c(A_\bullet,A_\bullet')}(U_\bullet,U_\bullet')$ is treated as a 
     $ Q_c(Hoch(A_\bullet),Hoch(A_\bullet'))$-comodule. 
Applying the  universal property of $R_{Q_c(Hoch(A_\bullet),Hoch(A_\bullet'))}(C_*(A_\bullet,U_\bullet),C_*(A_\bullet',U_\bullet'))$ we now obtain a unique map
    \begin{equation}\label{531wq}
        \lambda(U_\bullet,U_\bullet'):\mathcal P_{\mathcal M_c(A_\bullet,A_\bullet')}(U_\bullet,U_\bullet')\longrightarrow R_{Q_c(Hoch(A_\bullet),Hoch(A_\bullet'))}(C_*(A_\bullet,U_\bullet),C_*(A_\bullet',U_\bullet') )
    \end{equation} of  $ Q_c(Hoch(A_\bullet),Hoch(A_\bullet'))$-comodules. The maps in \eqref{eq5.32pp} and \eqref{531wq} together determine the morphism $\left(\gamma(A_\bullet,A_\bullet'),\lambda(U_\bullet,U_\bullet')\right)$  in $CoMod_k^c$. The fact that this determines a $CoMod^c_k$-enriched functor $(\gamma,\lambda): BMod_{\mathscr Mult_k}\longrightarrow \widetilde{BMod}_{\mathscr Mult_k}$  can now be proved in a manner similar to the proof of Theorem \ref{thm4.5}. 
\end{proof}

    \section{Comultiplicative sequences of coalgebras}

In this section, we study the coalgebraic counterpart of multiplicative sequences of algebras. We will see that a measuring between multiplicative sequences of algebras can be given more generally by a comultiplicative sequence of coalgebras, instead of a single coalgebra. We will study   comultiplicative sequences of coalgebras with respect to extensions of the Sweedler product, as well as measurings between Hochschild homologies.

      \begin{defn}\label{defn6.1}
        A comultiplicative sequence of $k$-coalgebras $(Q_\bullet,\delta)$ consists of 
        \begin{enumerate}[(a)]
            \item a collection $Q_\bullet=\left\{(Q_n,\Delta_n:Q_n\longrightarrow Q_n\otimes Q_n,\epsilon_n:Q_n\longrightarrow k)\right\}_{n\geq 0}$ of coassociative counital $k$-coalgebras, and 
            \item a collection $\delta=\left\{\delta_{m,n}: Q_{m+n}\longrightarrow Q_m\otimes Q_n\right\}_{m,n\geq 0}$ of $k$-coalgebra morphisms satisfying the following coassociativity condition,          i.e., the diagram 
        \end{enumerate}
            \begin{equation}\label{6.1rdig}
\begin{tikzcd}
Q_{l+m+n}
\arrow[r, "\delta_{l+m,n} "]
\arrow[d, "\delta_{l,m+n}"']
& Q_{l+m} \otimes Q_n \arrow[d, "\delta_{l,m}\otimes Q_n"] \\
Q_l \otimes Q_{m+n}
\arrow[r, "Q_l\otimes\delta_{m,n}"']
& Q_l \otimes Q_m \otimes Q_n
\end{tikzcd}
\end{equation}
commutes for all $l,m,n\geq 0$. 
We always assume that $Q_0=k$ and each $\delta_{n,0}=id:Q_n\longrightarrow Q_n\otimes Q_0=Q_n\otimes k\cong Q_n$ and similarly that each $\delta_{0,n}=id$.

        \smallskip
Let $(Q_\bullet,\delta),(Q_\bullet',\delta')$ be comultiplicative sequences of $k$-coalgebras. A morphism $f_\bullet:(Q_\bullet,\delta)\longrightarrow (Q_\bullet',\delta')$ of comultiplicative sequences is a collection $\left\{f_n:Q_n\longrightarrow Q_n'\right\}_{n\geq 0}$ of coalgebra morphisms such that the following diagram
\begin{equation}\label{diag6.2p}
\begin{tikzcd}
Q_{m+n}
\arrow[r, "f_{m+n} "]
\arrow[d, "\delta_{m,n}"']
& Q'_{m+n}  \arrow[d, "\delta_{m,n}'"] \\
Q_m \otimes Q_n
\arrow[r, "f_m\otimes f_n"']
&  Q_m' \otimes Q_n'
\end{tikzcd}
\end{equation}
commutes for all $m,n\geq 0$. We will denote the category of comultiplicative sequences of $k$-coalgebras by $\mathcal Co\mathscr Mult_k$. We will often denote an object $(Q_\bullet,\delta)$ of $\mathcal Co\mathscr Mult_k$ simply by $Q_\bullet$.
    \end{defn}

For $(Q_\bullet,\delta)\in \mathcal Co\mathscr Mult_k$ and any $m$, $n\geq 0$,  $x^{m+n}\in Q_{m+n}$, we will often write $\delta_{m,n}(x^{m+n})=x^{m(1)}\otimes x^{n(2)}\in Q_m\otimes Q_n$ by suppressing summation signs. This distinguishes it from $\Delta_m(x^m)=x^m_{(1)}\otimes x^m_{(2)}\in Q_m\otimes Q_m$ (suppressing summation signs) for $x^m\in Q_m$ given by the coproduct $\Delta_m$ on the coalgebra $Q_m$ for $m\geq 0$.

\begin{eg}
   \emph{Let $(C,\Delta,\epsilon)$ be a cocommutative $k$-coalgebra. We consider $Q_\bullet=\{Q_n\}_{n\geq0}$ by setting $Q_0:=k$ and $Q_n:=C$ for $n\geq 1$. Then, $Q_\bullet$ equipped with $\{\delta_{m,n}:=\Delta: Q_{m+n}=C
\longrightarrow C\otimes C=Q_m\otimes Q_n\}_{m,n\geq 1}$ is a comultiplicative sequence of coalgebras.}
\end{eg}

\begin{eg}\label{eg6.3p}
\emph{For each $n\geq 1$, we let $Q_n$ be the $k$-vector space with basis $\{x^{(0)},x^{(1)},\ldots,x^{(4^n)}\}$  and set $Q_0:=k\cong k\cdot x^{(0)}$. Each $Q_n$ is a cocommutative coalgebra with comultiplication $\Delta_n$ and the counit $\epsilon_n$ given on basis elements by}
\begin{equation*}
\begin{array}{c}
    \Delta_n(x^{(k)})=\sum_{i+j=k}x^{(i)}\otimes x^{(j)},\qquad 
\epsilon_n(x^{(0)}):=1, \qquad \epsilon_n(x^{(k)}):=0, \textrm{ }k> 0\\
\end{array}
\end{equation*}
\emph{It may be verified that $Q_\bullet:=\{Q_n\}_{n\geq 0}$ along   with $\delta:=\{\delta_{m,n}\}_{m,n\geq 0}$ given as follows is a comultiplicative sequence of coalgebras, which we will use again in Example \ref{eg6.8xc}.}
\begin{equation*}
  \delta_{m,n}:Q_{m+n}\longrightarrow Q_m\otimes Q_n\qquad   \delta_{m,n}(x^{(k)})=\sum_{\substack{i\cdot2^n+j\cdot 2^m=k\\0\leq i\leq 4^m , 0\leq j\leq 4^n}}x^{(i)}\otimes x^{(j)} \qquad \forall m,n\geq 1
\end{equation*}
\end{eg}

We will see in Proposition \ref{prop5.4} below  how to construct more examples of comultiplicative sequences of coalgebras starting from multiplicative sequences of algebras. First, we describe the convolution between a comultiplicative sequence and a multiplicative sequence. 

     \begin{thm}\label{prop5.2}
        Let $(A_\bullet,\tau)\in \mathscr Mult_k$ and $(Q_\bullet,\delta)\in \mathcal Co\mathscr Mult_k$. Then, the collection $ [Q, A]_\bullet=\left\{[Q, A]_n\right\}_{n\geq 0}$, given by setting 
        \begin{equation*}
            [Q, A]_0:=k\qquad [Q, A]_n:=Hom_k(Q_n,A_n)\qquad \forall n\geq1
        \end{equation*} 
        is an object of $\mathscr Mult_k$, equipped with $[\delta,\tau]=\left\{[\delta,\tau]_{m,n}\right\}_{m,n\geq 0}$ 
        \begin{equation}\label{eq5.3}
\begin{array}{c}
            [\delta,\tau]_{m,n}:  [Q, A]_m\otimes  [Q, A]_n\longrightarrow  Hom_k(Q_m\otimes Q_n,A_m\otimes A_n)\longrightarrow  [Q, A]_{m+n}\qquad 
\forall m,n\geq 1
\\
\mbox{$[\delta,\tau]_{m,n}(f\otimes g):=\tau_{m,n}\circ (f\otimes g)\circ \delta_{m,n}$}\\
\end{array}
        \end{equation}
        for $f\in  [Q, A]_m,g\in  [Q, A]_n$.
    \end{thm}

    \begin{proof}
     For each $n\geq 1$, we note that $ [Q, A]_n=Hom_k(Q_n,A_n)$ is a $k$-algebra equipped with the standard convolution product as described in \eqref{con3.2}. Since $Q_m$, $Q_n$ are coalgebras and $A_m$, $A_n$ are algebras, it follows from  \cite[Lemma 3.7.5]{AJ} that the canonical map
\begin{equation}\label{6.4rt}
 [Q, A]_m\otimes  [Q, A]_n=Hom_k(Q_m,A_m)\otimes Hom_k(Q_n,A_n) \longrightarrow  Hom_k(Q_m\otimes Q_n,A_m\otimes A_n)
\end{equation} is a $k$-algebra morphism. Now  since $\tau_{m,n}:A_m\otimes A_n\longrightarrow A_{m+n}$ is an algebra morphism and $\delta_{m,n}:Q_{m+n}\longrightarrow Q_m\otimes Q_n$ is a coalgebra morphism, it is clear from the definition in \eqref{eq5.3}  that $[\delta,\tau]_{m,n}$ is a morphism of algebras. Further, using the associativity condition \eqref{diag1} on $\tau$ and the coassociativity condition \eqref{6.1rdig} on $\delta$, it can be verified that the diagram 
        \begin{equation*}
\begin{tikzcd}[row sep=3em,  column sep=6em]
 [Q,A]_l\otimes  [Q,A]_m\otimes  [Q,A]_n
\arrow[r, "{[\delta,\tau]_{l,m}\otimes Id }"]
\arrow[d, "{Id\otimes [\delta,\tau]_{m,n}}"']
&  {[Q,A]_{l+m} \otimes  [Q,A]_n} \arrow[d, "{[\delta,\tau]_{l+m,n}}"] \\
 {[Q,A]_l \otimes  [Q,A]_{m+n}}
\arrow[r, " {[\delta,\tau]_{l,m+n}}"']
&  {[Q,A]_{l+m+n}}
\end{tikzcd}
\end{equation*} 
commutes for all $l,m,n\geq 0$. This shows that $ [Q, A]_\bullet\in \mathscr Mult_k$.
\end{proof}

   \begin{cor}
       For $(Q_\bullet,\delta)\in\mathcal Co\mathscr Mult_k$, the association $Q_\bullet\mapsto [Q,k]_\bullet=\{Hom_k(Q_i,k)\}_{i\geq 0}$ determines a functor $[-,k]_\bullet:\mathcal Co\mathscr Mult_k^{op}\longrightarrow \mathscr Mult_k$.
    \end{cor}

    \begin{proof}
This is clear from Proposition \ref{prop5.2}, by taking $A_\bullet$ to be the  
 multiplicative sequence $A_\bullet=\left\{A_i=k\right\}_{i\geq 0}$ whose structure maps are all given by the multiplication on $k$. 
    \end{proof}

    We now recall from \cite[Proposition 6.0.2]{SW} that if $A$ is a $k$-algebra, then its finite dual 
\begin{equation} \mathring{A}=\left\{g\in A^*=Hom_k(A,k)\;\middle|\; ker(g)\mbox{ contains an ideal of finite codimension in $A$} \right\}\subseteq A^*=Hom_k(A,k)
\end{equation} is a $k$-coalgebra. Further, for  algebras $A$, $A'$, there is an isomorphism $(A\otimes A')\mathring{}\cong\mathring{A}\otimes \mathring{A'}$  (see \cite[Lemma 6.0.1]{SW} ) of coalgebras.

      \begin{thm}\label{prop5.4}
        Let $(A_\bullet,\tau)\in \mathscr Mult_k$. Then, the collection $\mathring{A}_\bullet=\{\mathring{A}_i\}_{i\geq 0}$ of finite duals forms an object of $\mathcal Co\mathscr Mult_k$, equipped with the coalgebra morphisms
       $\mathring{A}_{m+n}\xrightarrow{\quad\mathring{\tau}_{m,n}\quad} ({A}_m\otimes A_n)\mathring{}\xrightarrow{\quad\cong\quad}  \mathring{A}_m\otimes  \mathring{A}_n $ for $m,n\geq 0$. The association  $A_\bullet\longmapsto \mathring{A}_\bullet$ determines a functor $(-)\mathring{}:\mathscr Mult_k\longrightarrow \mathcal Co\mathscr Mult_k^{op}.$ 
    \end{thm}

    \begin{proof} We know that each $\mathring{A}_i$ is a coalgebra. The result is now clear from the fact that we have isomorphisms $({A}_m\otimes A_n)\mathring{}\cong \mathring{A}_m\otimes  \mathring{A}_n $ of coalgebras for all $m$, $n\geq 0$.
       
    \end{proof}

We are now ready to describe an adjunction between the categories $\mathscr Mult_k$ and $\mathcal Co\mathscr Mult_k$.

    \begin{Thm}
    Let $(A_\bullet,\tau)\in \mathscr Mult_k$ and $(Q_\bullet,\delta)\in\mathcal Co\mathscr Mult_k$. Then, there is a natural isomorphism 
        \begin{equation*}
            \mathscr Mult_k(A_\bullet, [Q,k]_\bullet)\cong\mathcal Co\mathscr Mult_k(Q_\bullet,\mathring{A}_\bullet).
        \end{equation*}
     In other words, the functor $[-,k]_\bullet:\mathcal Co\mathscr Mult_k^{op}\longrightarrow \mathscr Mult_k$ is right adjoint to the functor $(-)\mathring{}:\mathscr Mult_k\longrightarrow \mathcal Co\mathscr Mult_k^{op}$.
    \end{Thm}

    \begin{proof}
        Let $f_\bullet=\{f_n:A_n\longrightarrow [Q,k]_n\}_{n\geq 0}\in \mathscr Mult_k(A_\bullet, [Q,k]_\bullet)$. We set $g_0:Q_0=k\longrightarrow k=\mathring{A}_0$ and for $n\geq 1$ we define 
        \begin{equation}
            f'_n: Q_n\longrightarrow \mathring{A}_n\qquad x^n\mapsto f'_n(x^n):=( a^n\longmapsto f_n(a^n)(x^n))
        \end{equation} for $a^n\in A_n,x^n\in Q_n$. Since $f_n$ is an algebra morphism, by the adjunction in \cite[Theorem 6.0.5]{SW}, the map $f'_n$ is a coalgebra morphism. Since $
        f_\bullet\in \mathscr Mult_k(A_\bullet, [Q,k]_\bullet)$, the left hand side diagram in \eqref{fdp6.7} below commutes in the category of $k$-algebras and it follows from  the adjunction in \cite[Theorem 6.0.5]{SW} that the diagram on the right hand side of \eqref{fdp6.7}    commutes in the category of $k$-coalgebras.
        \begin{equation}\label{fdp6.7}
        \begin{array}{ccc}
        \begin{CD}
        A_m\otimes A_n @>f_m\otimes f_n>> Q_m^*\otimes Q_n^* @>>> Hom_k(Q_m\otimes Q_n,k)\\
        @V\tau_{m,n}VV @V[\delta,k]_{m,n}VV @VHom_k(\delta_{m,n},k)VV \\
        A_{m+n} @>f_{m+n}>> Q_{m+n}^*@>id>> Q_{m+n}^*=Hom_k(Q_{m+n},k)\\
        \end{CD} & \Rightarrow & 
        \begin{CD}
        Q_m\otimes Q_n @>f'_m\otimes f'_n>> \mathring{A}_m\otimes \mathring{A}_n =(A_m\otimes A_n)\mathring{}\\
        @A\delta_{m,n}AA @AA\mathring{\tau}_{m,n}A \\
        Q_{m+n} @>f'_{m+n}>> \mathring{A}_{m+n}\\
        \end{CD} \\
        \end{array}
        \end{equation}
      From this, it follows that $f'_\bullet=\{f'_n\}_{n\geq 0} $ is a morphism in $\mathcal Co\mathscr Mult_k$. 
        Conversely, for  $g_\bullet=\{g_n\}_{n\geq 0}\in \mathcal Co\mathscr Mult_k(Q_\bullet,\mathring{A}_\bullet)$, we have the map $g'_n:A_n\longrightarrow Q_n^*$ of $k$-algebras  corresponding to the map $g_n:Q_n\longrightarrow \mathring{A}_n$ of $k$-coalgebras for each $n\geq 0$. By reversing the arguments, we see that $g'_\bullet=\{g'_n\}_{n\geq 0}$ is  a morphism in $\mathscr Mult_k$ from $A_\bullet$ to $[Q,k]_\bullet$. This gives the required adjunction.
    \end{proof}

     \begin{defn}\label{def6.6}
            Let $(A_\bullet,\tau),(A_\bullet',\tau')\in \mathscr Mult_k$. A comultiplicative measuring sequence $(Q_\bullet,\mho_\bullet)$ of coalgebras  from $(A_\bullet,\tau)$ to $(A_\bullet',\tau')$ consists of the following data
            \begin{enumerate}[(a)]
                \item a comultiplicative sequence of coalgebras $(Q_\bullet,\delta)$, where each $Q_n$ is cocommutative, and
                \item a collection $\mho_\bullet=\left\{\mho_n:Q_n\longrightarrow  {Hom}_k(A_n,A_n')\right\}_{n\geq 0}$ of $k$-linear maps such that
                \begin{enumerate}[(1)]
                    \item for $n=0$, we have
                    \begin{equation}
                        \mho_0(x):A_0=k\longrightarrow k=A_0'\qquad s\longmapsto  xs
                    \end{equation} for each $x\in Q_0=k$,
                    \item   for each $n\geq 1$, the pair $(Q_n,\mho_n:Q_n\longrightarrow  {Hom}_k(A_n,A_n'))$ forms a coalgebra measuring from $A_n$ to $A_n'$, and
                    \item for each $x^{m+n}\in Q_{m+n}$ with $\delta_{m,n}(x^{m+n})=x^{m(1)}\otimes x^{n(2)}\in Q_m\otimes Q_n$ and $m$, $n\geq 0$, the following diagram commutes
                    \begin{equation}\label{diag6.4}
                \begin{tikzcd}
                A_m \otimes A_n
                \arrow[r, "\tau_{m,n}"]
                \arrow[d, "\mho_m(x^{m(1)})\otimes\mho_{n}(x^{n(2)})"']
                &A_{m+n}
                \arrow[d, "\mho_{m+n}(x^{m+n})"]\\
                A_m'\otimes A_{n}' \arrow[r, "\tau_{m,n}'"']
                &A'_{m+n}
            \end{tikzcd}
            \end{equation}
 
                \end{enumerate}    
            \end{enumerate}
We will sometimes write  a comultiplicative measuring sequence $(Q_\bullet,\mho_\bullet)$ of coalgebras as $(Q_\bullet,\delta,\mho_\bullet)$, when we need to emphasize the structure maps of the comultiplicative sequence $(Q_\bullet,\delta)$. 
        \end{defn}

    If we consider $\mho_\bullet=\left\{\mho_n:Q_n\longrightarrow  {Hom}_k(A_n,A_n')\right\}_{n\geq 0}$ appearing in Definition \ref{def6.6}  as a collection of maps $\left\{\mho_n:Q_n\otimes A_n\longrightarrow A'_n\right\}_{n\geq 0}$, we note that the condition in \eqref{diag6.4} may also be expressed as
    \begin{equation}\label{diag4prm2}
        \mho_{m+n}\big(x^{m+n}\otimes\tau_{m,n}(a^m\otimes a^n)\big)= \tau_{m,n}'\Big(\mho_m(x^{m(1)}\otimes a^m)\otimes\mho_n(x^{n(2)}\otimes a^n)\Big)\quad \forall m,n\geq 0
    \end{equation}
    for all  $a^m\in A_m$, $a^n\in A_n$ and $x^{m+n}\in Q_{m+n}$ with $\delta_{m,n}(x^{m+n})=x^{m(1)}\otimes x^{n(2)}$, with summation signs suppressed.

    \smallskip

   \begin{eg}\label{eg6.8xc}
       \emph{We recall the comultiplicative sequence $(Q_\bullet,\delta)$ of coalgebras defined in Example \ref{eg6.3p}, where $Q_n$ has dimension $4^n+1$ for $n\geq 1$. Let $(A_\bullet,\tau)$ be the multiplicative sequence given by setting $A_0:=k$ and $A_n:=k[t_1,\ldots,t_n]$ for all $n\geq 1$. For each $m,n\geq 1$ the structure maps $\tau_{m,n}:A_m\otimes A_n\longrightarrow A_{m+n}$ are algebra homomorphisms induced by the assignments}
       \begin{equation*}
       \tau_{m,n}(t_i\otimes 1) =t_i^{2^n} \quad \textit{if } 1\leq i\leq m  \quad \textit{and}\quad
            \tau_{m,n}(1\otimes t_j) =t_{m+j}^{2^m} \quad \textit{if } 1\leq j\leq n.
      \end{equation*}
       \emph{We now set $\mho_0:Q_0= k\longrightarrow Hom_k(A_0,A_0)$ by assigning each $x=c\cdot x^{(0)}$ to the map $\mho_0(x):A_0=k\longrightarrow k=A_0$ given by $\mho_0(x)(s)=cs$ for all $s\in k$. For $n\geq 1$  we set}
       \begin{equation*}
         \mho_n:Q_n\longrightarrow Hom_k(A_n,A_n) \qquad   \mho_n(x^{(k)})(t_1^{\gamma_1}\cdots t_n^{\gamma_n})=
\begin{cases}
t_1^{\gamma_1}\cdots t_n^{\gamma_n}, & k=\sum_{i=1}^n\gamma_i,\\
0, & \mbox{{otherwise}}
\end{cases}
       \end{equation*}
       \emph{We may now verify that the pair $(Q_\bullet,\mho_\bullet:=\{\mho_n\}_{n\geq 0})$ forms a comultiplicative measuring sequence from $(A_\bullet,\tau)$ to itself in the sense of Definition \ref{def6.6}}.
   \end{eg}

         \smallskip

        \begin{thm}\label{prop6.7}
        Let $(A_\bullet,\tau),(A'_\bullet,\tau')\in\mathscr Mult_k$ and $(Q_\bullet,\delta)\in \mathcal Co\mathscr Mult_k$. Let $\mho_\bullet=\left\{\mho_n:Q_n\longrightarrow  {Hom}_k(A_n,A_n')\right\}_{n\geq 0}$ be a family of $k$-linear maps, with $\mho_0(x)(1)=x$ for all $x\in Q_0=k$. We now set
        \begin{equation}\label{eq6.5}
        g_n:A_n\longrightarrow [Q,A']_n= {Hom}_k(Q_n, A'_n)\qquad a^n\longmapsto g_n(a^n):=\left(x^n\longmapsto \mho_n(x^n)(a^n)\right)\qquad\forall n\geq1,
    \end{equation}
    and $g_0=id:A_0=k\longrightarrow k=[Q,A']_0$.
    Then, $(Q_\bullet,\mho_\bullet)$ is a comultiplicative measuring sequence of coalgebras  from $(A_\bullet,\tau)$ to $(A'_\bullet,\tau')$ if and only if $g_\bullet=\{g_n\}_{n\geq0}: A_\bullet\longrightarrow  [Q, A']_\bullet$ is a morphism in $\mathscr Mult_k$.
    \end{thm}

    \begin{proof}
       The proof follows in a manner similar to Proposition \ref{prop3.3}.  

    \end{proof}

\smallskip

For $(Q_\bullet,\delta)\in \mathcal Co\mathscr Mult_k$, it is clear by Proposition \ref{prop5.2} that  the association $A_\bullet\longmapsto [Q,A]_\bullet$ determines a functor $[Q,-]_\bullet:\mathscr Mult_k\longrightarrow \mathscr Mult_k$. We now show that the restriction of the functor $[Q,-]_\bullet:\mathscr Mult_k\longrightarrow \mathscr Mult_k$ to the full subcategory $\mathcal Comm_k^{\geq 0}$ of $\mathscr Mult_k$ defined in Section \ref{section3} has a left adjoint. This left adjoint will be a functor 
\begin{equation}
(Q\Box-)_\bullet:\mathscr Mult_k\longrightarrow Comm_k^{\geq 0}
\end{equation} which extends the Sweedler product like construction in Section 3.  When $Q_\bullet$ is the comultiplicative sequence given by taking $Q_i=C$, $i\geq 1$ for a single cocommutative coalgebra $C$, we recover the adjunction in Theorem \ref{Thm3.9}.

 \smallskip  
We now take $(Q_\bullet,\delta)\in \mathcal Co\mathscr Mult_k$ and 
$A_\bullet\in \mathscr Mult_k$.  We set $(Q\Box A)_0:=k$, and for each $n\geq1$, we let  $(Q\Box A)_n$ be the commutative $k$-algebra  generated by the symbols 
\begin{equation} 
\{\mbox{$x^i\Box a^i$ $\vert$  $x^i\in Q_i$, $a^i\in A_i$, $1\leq i\leq n$}\}
\end{equation} subject to the following relations for  $1\leq i\leq n$
\begin{enumerate}[(a)]
        \item the canonical map $Q_i\times A_i\longrightarrow (Q\Box A)_n$ defined by $(x^i,a^i)\mapsto x^i\Box a^i$  is $k$-bilinear,

            \item for  $a_1^i$, $a_2^i\in A_i$ and $x^i\in Q_i$ with coproduct $\Delta_i(x^i)=x^i_{(1)}\otimes x^i_{(2)}\in Q_i\otimes Q_i$,  we have $x^i\Box a^i_1a^i_2=(x^i_{(1)}\Box a^i_1)(x^i_{(2)}\Box a^i_2)$, 

            \item for $a^i\in A_i$, $a^j\in A_j$ with $i+j\leq n$ and $x^{i+j}\in Q_{i+j}$ with $\delta_{i,j}(x^{i+j})=x^{i(1)}\otimes x^{j(2)}
\in Q_i\otimes Q_j$, we have $x^{i+j}\Box (\tau_{i,j}(a^i\otimes a^j))=(x^{i(1)}\Box a^i)(x^{j(2)}\Box a^j)$,
        \item  for $x^i\in Q_i$, we have $x^i\Box 1_{A_i}=\epsilon_i(x^i)$, where $\epsilon_i:Q_i\longrightarrow k$ is the counit on the coalgebra $Q_i$.
    \end{enumerate}

     \begin{lem}\label{lem6.6}
        Let $(A_\bullet,\tau)\in \mathscr Mult_k$ and $(Q_\bullet,\delta)\in \mathcal Co\mathscr Mult_k $, with each $Q_i$ cocommutative for $i\geq 0$.

        \medskip 
        \noindent
        (a) The collection $(Q\Box A)_\bullet=\left\{(Q\Box A)_n\right\}_{n\geq 0}$ is an object of $\mathcal Comm^{\geq 0}_k$, equipped with maps $\xi=\{\xi_{m,n}\}_{0\leq m\leq n}$  determined by
            \begin{equation}\label{eq6.13t}
                \xi_{m,n}:(Q\Box A)_m\longrightarrow (Q\Box A)_n\qquad (x^t\Box a^t)\longmapsto (x^t\Box a^t)
            \end{equation}
             for all $x^t\in Q_t, a^t\in A_t $ with $1\leq t\leq m$. The assignment $A_\bullet\mapsto (Q\Box A)_\bullet$ defines a functor $(Q\Box-)_\bullet:\mathscr Mult_k\longrightarrow \mathcal Comm_k^{\geq 0}.$

             \medskip
             \noindent

             (b) There is a comultiplicative measuring sequence $(Q_\bullet,\mho_\bullet=\left\{\mho_n:Q_n\otimes A_n\longrightarrow (Q\Box A)_n\right\}_{n\geq 0})$ of coalgebras from $A_\bullet$ to $(Q\Box A)_\bullet$, with $\mho_n$ determined  by
        \begin{equation*}
        \begin{array}{c}
              \mho_0:k\otimes k\overset{\cong}{\longrightarrow} k \qquad x\otimes s\longmapsto xs \\
               \mho_n: Q_n\otimes A_n\longrightarrow (Q\Box A)_n\qquad x^n\otimes a^n\mapsto x^n\Box a^n\qquad \forall n\geq 1.
        \end{array}
        \end{equation*} where $(Q\Box A)_\bullet\in \mathcal Comm_k^{\geq 0}$ is treated as an object of $\mathscr Mult_k$.
        \end{lem}

        \begin{proof}
          By construction, whenever $m\leq n$, the relations among the generators of $(Q\Box A)_m$ continue to hold in $(Q\Box A)_n$. It follows that  $\xi_{m,n}:(Q\Box A)_m\longrightarrow (Q\Box A)_n$  are algebra morphisms. This proves (a). The proof of (b) follows in a similar manner to that of Lemma \ref{lem3.4}(b). 
        \end{proof}

        \begin{Thm}\label{thm6.10}
            Let   $(Q_\bullet,\delta)\in \mathcal Co\mathscr Mult_k$, with each $Q_i$ cocommutative for $i\geq 0$. For $A_\bullet \in \mathscr Mult_k$, the comultiplicative measuring sequence $(Q_\bullet,\mho_\bullet=\{\mho_n: Q_n\otimes A_n\longrightarrow (Q\Box A)_n\}_{n\geq 0})$ of coalgebras satisfies the following universal property: for any $B_\bullet\in \mathcal Comm_k^{\geq 0}$ and any comultiplicative measuring sequence $(Q_\bullet,\mho'_\bullet=\{\mho'_n: Q_n\otimes A_n\longrightarrow B_n\}_{n\geq 0})$ of coalgebras from $A_\bullet$ to $B_\bullet$, there exists a unique morphism $g_\bullet:=\{g_n\}_{n\geq 0}:(Q\Box A)_\bullet \longrightarrow B_\bullet$ in $\mathcal Comm_k^{\geq 0}$ such that the following diagram 
    \begin{equation}\label{dgi615}
\begin{array}{c}
    \begin{tikzpicture}[>=stealth]
\node (A) at (0,0) {$Q_n\otimes A_n$};
\node (B) at (1.3,1.3) {$B_n$};
\node (C) at (2.6,0) {$(Q \Box A)_n$};

\draw[->] (A) -- node[above left] {$\mho'_n$} (B);
\draw[->] (A) -- node[below ] {$\mho_n$} (C);
\draw[->, dashed] (C) -- node[right] {$g_n$} (B);
\end{tikzpicture}
\end{array}
\end{equation}
commutes for each $n\geq 0$.
        \end{Thm}
        \begin{proof}
           Let $\tau^A=\left\{\tau_{m,n}^A:A_m\otimes A_n\longrightarrow A_{m+n}\right\}_{m,n\geq 0}$ be the structure maps on $A_\bullet \in \mathscr Mult_k$ and $\tau^B=\left\{\tau_{m,n}^B:B_m\otimes B_n\longrightarrow B_{m+n}\right\}_{m,n\geq 0}$ be the structure maps on $B_\bullet\in \mathcal Comm^{\geq0}_k$ when treated as an object of $\mathscr Mult_k$.  Let $(Q_\bullet,\mho'_\bullet=\{\mho'_n: Q_n\otimes A_n\longrightarrow B_n\}_{n\geq 0})$  be a comultiplicative measuring sequence of coalgebras from $A_\bullet$ to $B_\bullet$. We  set $g_0=id: (Q\Box A)_0=k\longrightarrow k=B_0$ and for $n\geq1$ we define  
           \begin{equation}
               g_n:(Q\Box A)_n\longrightarrow B_n\qquad  x^i\Box a^i\longmapsto \tau^B_{i,n-i}(\mho'_i(x^i\otimes a^i)\otimes 1)=\tau^B_{n-i,i}( 1\otimes \mho'_i(x^i\otimes a^i))
           \end{equation}
           for $a^i\in A_i,x^i\in Q_i$ where $1\leq i\leq n$. It can be verified in a  manner similar to Proposition \ref{prop3.6} that each $g_n$ is a well defined morphism of algebras. We claim that $g_\bullet=\{g_n\}_{n\geq 0}$ is a morphism in $\mathcal Comm^{\geq 0}_k$. Since $\mathcal Comm^{\geq 0}_k$ is a full subcategory of $\mathscr Mult_k$, it is enough to verify the following equation 
        \begin{equation}\label{eq6.14}
            g_{m+n}\left((\delta\Box \tau^A)_{m,n}((x^i\Box a^i)\otimes(x^j\Box a^j))\right)= \tau^B_{m,n}\left(g_m(x^i\Box a^i)\otimes g_n(x^j\Box a^j)\right)
        \end{equation}
        holds for all $m,n\geq 0$, $x^i\in Q_i,x^j\in Q_j$ and $a^i\in A_i,a^j\in A_j$ where $1\leq i\leq m$, $1\leq j\leq n$. This is clear for  $m=0$ or $n=0$. For $m,n\geq 1$, we see that
        \begin{align*}
        \tau^B_{m,n}\left(g_m(x^i\Box a^i)\otimes g_n(x^j\Box a^j)\right)
            &=\tau^B_{m,n}\left(\tau^B_{i,m-i}\left(\mho'_i(x^i\otimes a^i)\otimes1\right)\otimes \tau^B_{j,n-j}\left(\mho'_j(x^j\otimes a^j)\otimes1\right)\right)\\
            &=\tau^B_{m,n}\left(\tau^B_{i,m-i}\left(\mho'_i(x^i\otimes a^i)\otimes1\right)\otimes 1\right)\cdot\tau^B_{m,n}\left(1\otimes \tau^B_{j,n-j}\left(\mho'_j(x^j\otimes a^j)\otimes1\right)\right)\\
            &=\tau^B_{m,n}\left(\tau^B_{i,m-i}\left(\mho'_i(x^i\otimes a^i)\otimes1\right)\otimes 1\right)\cdot\tau^B_{m,n}\left(1\otimes \tau^B_{n-j,j}\left(1\otimes \mho'_j(x^j\otimes a^j)\right)\right)\tag{as $B_n$ is commutative }\\
            &=\tau^B_{i,m-i+n}\left(\mho'_i(x^i\otimes a^i)\otimes \tau^B_{m-i,n}(1\otimes 1)\right)\cdot\tau^B_{m+n-j,j}\left(\tau^B_{m,n-j}(1\otimes 1)\otimes \mho'_j(x^j\otimes a^j)\right)\tag{by \eqref{diag1}}\\
            &=g_{m+n}(x^i\Box a^i)\cdot g_{m+n}(x^j\Box a^j)
            =g_{m+n}\left((x^i\Box a^i)(x^j\Box a^j)\right) =g_{m+n}\left((\delta\Box \tau^A)_{m,n}((x^i\Box a^i)\otimes (x^j\Box a^j))\right)
        \end{align*}
        It follows that the diagram \eqref{dgi615} commutes for each $n\geq 0$. The uniqueness of $g_\bullet=\{g_n\}_{n\geq 0}$  is clear. This proves the result.
        \end{proof}

        \begin{Thm}
          Let   $(Q_\bullet,\delta)\in \mathcal Co\mathscr Mult_k$, with each $Q_i$ cocommutative for $i\geq 0$.     For  $A_\bullet\in \mathscr Mult_k$,  $B_\bullet\in \mathcal Comm_k^{\geq 0}$, there is a natural isomorphism
            \begin{equation}\label{eq6.17d}
                  \mathcal Comm_k^{\geq 0}((Q\Box  A)_\bullet, B_\bullet)\cong  \mathscr Mult_k(A_\bullet,  [Q,B]_\bullet). 
            \end{equation}
             In other words, $(Q\Box-)_\bullet:\mathscr Mult_k\longrightarrow \mathcal Comm^{\geq 0}_k$ is left adjoint to the functor  $[Q,-]_\bullet:\mathcal Comm^{\geq 0}_k\longrightarrow \mathscr Mult_k$.
        \end{Thm}
        \begin{proof} This follows by comparing the two sides of \eqref{eq6.17d} using 
            Proposition \ref{prop6.7} and Theorem \ref{thm6.10}.

        \end{proof}

    Let $A_\bullet,A_\bullet'\in \mathscr Mult_k$.    In Theorem \ref{Thm3.6}, we have shown  that a coalgebra measuring $(C,\Phi)$ from $A_\bullet$ to $A_\bullet'$ induces a coalgebra measuring $(C,Hoch^\Phi)$ from $Hoch(A_\bullet)$ to $Hoch(A_\bullet')$. We conclude   by showing a more general result for a comultiplicative measuring sequence of coalgebras between multiplicative sequences.

       \smallskip
       Accordingly, let $(Q_\bullet,\delta,\mho_\bullet=\{\mho_n\}_{n\geq 0})$ be a comultiplicative measuring sequence of coalgebras from $A_\bullet$ to $A_\bullet'$. Then, for each $n\geq 0$, the pair $(Q_n,\mho_n: Q_n\longrightarrow Hom_k(A_n,A_n'))$ is a coalgebra measuring from $A_n$ to $A_n'$. Hence, for each $x^n\in Q_n$ with  iterated coproduct $\Delta_n^{p}(x^n)=x^n_{(1)}\otimes x^n_{(2)}\otimes\ldots\otimes x^n_{(p+1)}\in Q_n^{\otimes p+1} $ for $p\geq 0$,  the result of \cite[ Proposition  2.2]{AB2} gives a map of Hochschild complexes
        \begin{equation}\label{eq6.15.}
        \begin{array}{c}
            Hoch^\mho_{n}(x^n):=C_*^{\mho_n}(x^{n}) :Hoch(A_\bullet)_n=C_*(A_n)\longrightarrow C_*(A_n') =Hoch(A'_\bullet)_n\qquad\forall n\geq 0\\
             C_p^{\mho_n}(x^{n})(a_0^n,a^n_1,\ldots,a_p^n)=(\mho_n(x_{(1)}^{n})(a_0^n),\mho_n(x_{(2)}^{n})(a_1^n),\ldots,\mho_n(x_{(p+1)}^{n})(a_p^n)) \qquad\forall p\geq 0
        \end{array}
        \end{equation}

       We are now ready to show that a comultiplicative measuring sequence $(Q_\bullet,\mho_\bullet)$ of coalgebras from $A_\bullet$ to $A_\bullet'$ induces a comultiplicative measuring sequence between graded algebras $Hoch (A_\bullet)$,  $Hoch(A_\bullet')$ in $\mathcal C_k$. For this, we make explicit the idea of a comultiplicative measuring sequence between graded algebras in the monoidal category $\mathcal C_k$.

    \begin{defn}\label{def6.10}
     Let $T=\{T_{n,*}\}_{n\geq 0}, T'=\{T'_{n,*}\}_{n\geq 0}\in \mathscr Gr\mathcal Alg(\mathcal C_k)$ be graded algebras in $\mathcal C_{k}$ with multiplication maps
      \begin{equation*}
         \mu=\left\{\mu_{m,n}:T_{m,*}\otimes T_{n,*}\longrightarrow T_{m+n,*}\right\}_{m,n\geq 0}\qquad\textit{and}\qquad\mu'=\left\{\mu'_{m,n}:T'_{m,*}\otimes T'_{n,*}\longrightarrow T'_{m+n,*}\right\}_{m,n\geq 0}
     \end{equation*}respectively.  A comultiplicative measuring sequence $(Q_\bullet,\Xi)$ of coalgebras from $T$ to $T'$ consists of
     \begin{enumerate}[(a)]
         \item a comultiplicative sequence of cocommutative coalgebras $(Q_\bullet,\delta)$, and
         \item  a collection  $\Xi=\Big\{\Xi_n:Q_n\longrightarrow  \mathcal C_k(T_{n,*},T'_{n,*})\Big\}_{n\geq 0}$ of $k$-linear maps
    satisfying the following condition
     \begin{equation}\label{eq6.16}
        \Xi_{m+n}(x^{m+n})(\mu_{m,n}(t^m_p\otimes t^n_q))=\mu'_{m,n}\left(\sum\Big(\Xi_m(x^{m(1)})(t^m_p)\Big)\otimes\Big(\Xi_n(x^{n(2)})(t^n_q) \Big)\right)
     \end{equation}
     for all $m,n\geq 0$, $t_p^m\in T_{m,p}$, $t_q^n\in T_{n,q}$ and $x^{m+n}\in Q_{m+n}$ with $\delta_{m,n}(x^{m+n})=x^{m(1)}\otimes x^{n(2)}\in Q_m\otimes Q_n$.
      \end{enumerate}
\end{defn}

\begin{Thm}\label{Thm6.11}
     Let  $A_\bullet,A'_\bullet\in \mathscr{M}ult_k$. Let $(Q_\bullet,\delta, \mho_\bullet=\{\mho_n\}_{n\geq 0})$ be a comultiplicative measuring sequence of coalgebras from $A_\bullet$ to $A_\bullet'$. Then, the pair $(Q_\bullet,Hoch^{\mho}=\{Hoch_{n}^{\mho}:Q_n\longrightarrow  \mathcal C_k(C_*(A_n),C_*(A'_n))\}_{n\geq 0})$ forms a comultiplicative measuring sequence of coalgebras from $Hoch(A_\bullet)$ to $Hoch(A'_\bullet)$, where 
     \begin{equation}
     \begin{array}{c}
          Hoch_{n}^{\mho}:Q_n\longrightarrow  \mathcal C_k(C_*(A_n),C_*(A'_n))  \\
          x^n\longmapsto Hoch_{n}^{\mho}(x^n):=C^{\mho_n}_*(x^n):Hoch(A_\bullet)_n=C_*(A_n)\longrightarrow C_*(A_n')=Hoch(A'_\bullet)_n
     \end{array}
     \end{equation}
\end{Thm}

\begin{proof}
 The proof is similar to that of Theorem \ref{Thm3.6}. As in the proof of Theorem \ref{Thm3.6},  for $r\geq 0$, the maps $C_r(\tau_{m,n}): C_r(A_m  \otimes A_n)\longrightarrow C_r(A_{m+n})$, $C_r(\tau'_{m,n}):C_r(A'_m  \otimes A'_n)\longrightarrow C_r(A'_{m+n})$   preserve the action (described in \eqref{act2}) of $S_r$ on the terms in the Hochschild complexes.  
    Further, since each $Q_n$ is cocommutative, for any $x^n\in Q_n$, the maps 
\begin{equation}
Hoch^\mho_{n}(x^n)=C^{\mho_n}_*(x^n):Hoch(A_\bullet)_n=C_*(A_n)\longrightarrow C_*(A'_n)=Hoch(A'_\bullet)_n
\end{equation} described in \eqref{eq6.15.} preserve the action of permutation groups on terms in the Hochschild complexes.
To complete the proof, we must show that $(Q_\bullet,Hoch^\mho)$ satisfies the condition in \eqref{eq6.16}. If $m=0$ or $n=0$, then the result is clear. We take $m,n\geq 1$ and note that for $(a^m_0,\ldots, a^m_p)\in C_p(A_m),(a^n_0,\ldots, a^n_q)\in C_q(A_n)$ and $x^{m+n}\in Q_{m+n}$ with  iterated coproduct $\Delta^{p+q}_{m+n}(x^{m+n})=x^{m+n}_{(1)}\otimes\ldots \otimes x^{m+n}_{(p+q+1)}\in Q_{m+n}^{\otimes p+q+1}$ and $\delta_{m,n}(x^{m+n})=x^{m(1)}\otimes x^{n(2)}$, we have
   \begin{align}
       &Hoch^{\mho}_{m+n}(x^{m+n})\left(\tau^{sh}_{m,n}\Big((a^m_0,\ldots, a^m_p)\otimes (a^n_0,\ldots, a^n_q)\Big)\right)\notag\\
       &=C_{p+q}^{\mho_{m+n}}(x^{m+n})\left(\sum_{\sigma\in S_{p,q}}sgn(\sigma)\left(\sigma\cdot\left(\tau_{m,n}(a^m_0\otimes a^n_0),\tau_{m,n}(a^m_1\otimes1),\ldots,\tau_{m,n}(a^m_p\otimes 1),\tau_{m,n}(1\otimes a^n_1),\ldots ,\tau_{m,n}(1\otimes a^n_q)\right)\right)\right)\notag\\
       &=\sum_{\sigma\in S_{p,q}}sgn(\sigma)\Big(\sigma\cdot\Big(\mho_{m+n}(x^{m+n}_{(1)})(\tau_{m,n}(a^m_0\otimes a_0^n)),\mho_{m+n}(x^{m+n}_{(2)})(\tau_{m,n}(a^m_1\otimes1)),\ldots,\mho_{m+n}(x^{m+n}_{(p+1)})(\tau_{m,n}(a^m_p\otimes 1)),\notag\\
       &\qquad\mho_{m+n}(x^{m+n}_{(p+2)})(\tau_{m,n}(1\otimes a^n_1)),\ldots ,\mho_{m+n}(x^{m+n}_{(p+q+1)})(\tau_{m,n}(1\otimes a^n_q))\Big)\Big)
       \end{align}
       On the other hand, we see
       \begin{align*}
           &\tau'^{sh}_{m,n}\left(Hoch^{\mho}_{m}(x^{m(1)})(a^m_0,\ldots, a^m_p)\otimes Hoch^{\mho}_n(x^{n(2)})(a^n_0,\ldots, a_q^n)\right)\\
           &=\tau'^{sh}_{m,n}\left(C_p^{\mho_m}(x^{m(1)})(a^m_0,\ldots, a^m_p)\otimes C_q^{\mho_n}(x^{n(2)})(a^n_0,\ldots, a^n_q)\right)\\
           &=\tau'^{sh}_{m,n}\Big(\Big(\mho_{m}(x^{m(1)}_{(1)})(a^m_0),\mho_{m}(x^{m(1)}_{(2)})(a^m_1),\ldots, \mho_{m}(x^{m(1)}_{(p+1)})(a^m_p)\Big)\otimes \Big(\mho_{n}(x^{n(2)}_{(1)})(a^n_0), \mho_{n}(x^{n(2)}_{(2)})(a^n_1),\ldots,\mho_{n}(x^{n(2)}_{(q+1)})(a^n_q)\Big)\Big)\\
           &=C_{p+q}(\tau'_{m,n})\sum_{\sigma\in S_{p,q}}sgn(\sigma)\Big(\sigma\cdot \Big(\mho_{m}(x^{m(1)}_{(1)})(a^m_0)\otimes \mho_{n}(x^{n(2)}_{(1)})(a^n_0),\mho_{m}(x^{m(1)}_{(2)})(a^m_1)\otimes 1,\ldots, \mho_{m}(x^{m(1)}_{(p+1)})(a^m_p)\otimes 1,\\
           &\qquad 1\otimes \mho_{n}(x^{n(2)}_{(2)})(a^n_1),\ldots,1\otimes\mho_{n}(x^{n(2)}_{(q+1)})(a^n_q)\Big)\Big)\notag\\
           &=\sum_{\sigma\in S_{p,q}}sgn(\sigma)\Big(\sigma\cdot\Big(\mho_{m+n}(x^{m+n}_{(1)})(\tau_{m,n}(a^m_0\otimes a_0^n)),\mho_{m+n}(x^{m+n}_{(2)})(\tau_{m,n}(a^m_1\otimes1)),\ldots,\mho_{m+n}(x^{m+n}_{(p+1)})(\tau_{m,n}(a^m_p\otimes 1)),\\
           &\qquad\mho_{m+n}(x^{m+n}_{(p+2)})(\tau_{m,n}(1\otimes a^n_1)),\ldots ,\mho_{m+n}(x^{m+n}_{(p+q+1)})(\tau_{m,n}(1\otimes a_q^n))\Big)\Big)\notag
       \end{align*}
       where the last equality holds by applying the condition \eqref{diag4prm2} on $\mho_\bullet$, the cocommutativity of $Q_{m+n}$ and the fact that 
       $\delta_{m,n}:Q_{m
+n}\longrightarrow Q_m\otimes Q_n$ is a morphism of coalgebras. This proves the result. 
\end{proof}

\end{document}